\documentclass[11pt,a4paper,reqno]{amsart}
\usepackage[english]{babel}
\usepackage[applemac]{inputenc}
\usepackage[T1]{fontenc}
\usepackage{palatino}
\usepackage{amsmath}
\usepackage{amssymb}
\usepackage{amsthm}
\usepackage{amsfonts}
\usepackage{graphicx}
\usepackage{color}
\usepackage{mathtools}

\usepackage[colorlinks = true, citecolor = black]{hyperref}
\title[Hardy spaces and QC maps in the Heisenberg group]{Hardy spaces and quasiconformal maps\\ in the Heisenberg group}
\author[T.\ Adamowicz]{Tomasz Adamowicz}
\address{The Institute of Mathematics, Polish Academy of Sciences \\ ul. Sniadeckich 8, 00-656 Warsaw, Poland}
\email{tadamowi@impan.pl}
\author[K. F\"assler]{Katrin F\"assler}
\address{Department of Mathematics and Statistics\\ University of Jyv\"askyl\"a \\ P.O. Box 35 (MaD),
FI-40014 University of Jyv\"askyl\"a, Finland}
\email{katrin.s.fassler@jyu.fi}

\thanks{T.A.\ was supported by the National Science
Center, Poland (NCN), UMO-2017/25/B/ST1/01955. K.F.\ was supported
by the Swiss National Science Foundation grant 161299, and by the
Academy of Finland grants 321696, 328846. The work was partially
supported by the grant 346300 for IMPAN from the Simons Foundation
and the matching 2015-2019 Polish MNiSW fund. The paper was
completed while K.F.\ visited the Hausdorff Research Institute for
Mathematics in Bonn during the trimester \emph{Interactions
between Geometric measure theory, Singular integrals, and PDE}}
\date{\today}

\newcommand{\R}{\mathbb{R}}

\newcommand{\1}{\mathbf{1}}

\newcommand{\diam}{\operatorname{diam}}

\newcommand{\Om}{\Omega}
\newcommand{\Hei}{{\mathbb{H}}^{1}}

\newcommand{\ep}{\varepsilon}
\newcommand{\om}{\omega}

\def\Barint_#1{\mathchoice
          {\mathop{\vrule width 6pt height 3 pt depth -2.5pt
                  \kern -8pt \intop}\nolimits_{#1}}%
          {\mathop{\vrule width 5pt height 3 pt depth -2.6pt
                  \kern -6pt \intop}\nolimits_{#1}}%
          {\mathop{\vrule width 5pt height 3 pt depth -2.6pt
                  \kern -6pt \intop}\nolimits_{#1}}%
          {\mathop{\vrule width 5pt height 3 pt depth -2.6pt
                  \kern -6pt \intop}\nolimits_{#1}}}

\numberwithin{equation}{section}

\theoremstyle{plain}
\newtheorem{thm}[equation]{Theorem}

\newtheorem{lemma}[equation]{Lemma}
\newtheorem{lemdef}[equation]{Lemma and Definition}

\newtheorem{cor}[equation]{Corollary}
\newtheorem{proposition}[equation]{Proposition}

\theoremstyle{definition}

\newtheorem{definition}[equation]{Definition}

\theoremstyle{remark}
\newtheorem{remark}[equation]{Remark}

\definecolor{blau}{rgb}{0.1,0.0,0.9}
\newcommand{\blue}{\color{blau}}
\definecolor{lila}{rgb}{0.5,0.0,0.5}

\definecolor{Bcolor}{rgb}{0.5,0.0,0.0}

\newcounter{komcounter}
\numberwithin{komcounter}{section}

\newcommand{\kom}[1]{}
\renewcommand{\kom}[1]{{\bf \blue[#1]}}

\usepackage{hyperref}
\hypersetup{%
  colorlinks = true,
  linkcolor  = black
}

\begin{document}

\begin{abstract}
We define Hardy spaces $H^p$, $0<p<\infty$, for quasiconformal
mappings on the Kor\'{a}nyi unit ball $B$ in the first Heisenberg
group $\mathbb{H}^1$. Our definition is stated in terms of the
Heisenberg polar coordinates introduced by Kor\'{a}nyi and
Reimann, and Balogh and Tyson. First, we prove the existence of
$p_0(K)>0$ such that every $K$-quasiconformal map $f:B \to f(B)
\subset \mathbb{H}^1$ belongs to $H^p$ for all $0<p<p_0(K)$.
Second, we give two equivalent conditions for the $H^p$ membership
of a quasiconformal map $f$, one in terms of the radial limits of
$f$, and one using a nontangential maximal function of $f$.  As an application, we characterize Carleson measures on
 $B$ via integral inequalities for quasiconformal mappings on $B$ and their radial limits. Our paper thus
extends results by Astala and Koskela, Jerison and Weitsman,
Nolder,
 and Zinsmeister, from $\mathbb{R}^n$ to $\mathbb{H}^1$.
A crucial difference between the proofs in $\mathbb{R}^n$ and
$\mathbb{H}^1$ is caused by the nonisotropic nature of the
Kor\'{a}nyi unit sphere with its two characteristic points.
\end{abstract}

\maketitle

\tableofcontents

\section{Introduction}
A holomorphic function $f$ on the unit disk $\mathbb{D}$ in
$\mathbb{C}$ belongs to the \emph{Hardy class $H^p$}, for some
$0<p<\infty$, if
\begin{equation}\label{eq:HardyPlane}
\sup_{0<s<1} \left(\int_0^{2\pi}
|f(se^{\mathrm{i}\varphi})|^p\,d\varphi\right)^{1/p}< \infty.
\end{equation}
According to a result by Hardy and Littlewood \cite{MR1555303},
condition \eqref{eq:HardyPlane} holds for a holomorphic function
$f$ on $\mathbb{D}$ if and only if the nontangential maximal
function of $f$ belongs to $L^p(S^1)$. Here and in the rest of
this paper, we consider Lebesgue spaces $L^p$ of $p$-integrable
functions for all exponents $0<p<\infty$, not only in the normed
case $p\geq 1$.

While Hardy spaces play an important role in complex analysis,
they have also spurred the development of real-variable methods
used to study their analogs in $\mathbb{R}^n$, or even on
homogeneous groups
\cite{MR0268655,MR447953,MR657581,MR628971,MR1232192}. Another
line of research, close in spirit to the condition
\eqref{eq:HardyPlane}, is the $H^p$ theory for quasiconformal or
quasiregular mappings on the unit ball in $\mathbb{R}^n$, where
"quasiconformal" and "quasiregular" serve as substitutes for
"conformal" and "holomorphic", respectively, see
\cite{MR2900163,MR624919,MR1013976,MR1307591,MR1943765,MR860655, MR4310177}.

The purpose of the present paper is to develop an $H^p$ theory for
quasiconformal mappings on the unit ball in the first Heisenberg
group $\mathbb{H}^1$, thus extending some of the previously
mentioned results from $\mathbb{R}^n$ to $\mathbb{H}^1$.
Quasiconformal maps in $\mathbb{H}^1$ have been studied
extensively, both as prototypes for quasiconformal maps in
non-Euclidean metric measures spaces \cite{zbMATH00746069} and due
to their connection with complex hyperbolic geometry
\cite{Pansu1,MR1317384}.

We consider mappings defined on the unit ball $B:=B(0,1)\subset
\mathbb{H}^1$ with respect to the \emph{Kor\'{a}nyi norm}
$\|\cdot\|$. This gauge function plays a distinguished role on
$\mathbb{H}^1$ as it is connected to the fundamental solution of
the sub-Laplacian, and it gives rise to a system of \emph{polar
coordinates} $(s,\omega)\in (0,\infty) \times \left[\partial B
\setminus \{z=0\}\right]$, see \cite{zbMATH04023589,MR1942237}.
For the precise definitions, we refer the reader to Section
\ref{s:prelim}; for now we simply mention that
\eqref{eq:HardyPlane} motivates the following definition.

\begin{definition}\label{d:Hp_Heis}
Let $0<p<\infty$. A quasiconformal map $f:B \subset \mathbb{H}^1
\to f(B) \subset \mathbb{H}^1$ belongs to the \emph{Hardy class}
$H^p$ if
\begin{displaymath}
\|f\|_{H^p} := \sup_{0<s<1} \left(\int_{\partial B}
\|f(\gamma(s,\omega))\|^p\,
d\mathcal{S}^3(\omega)\right)^{1/p}<\infty.
\end{displaymath}
\end{definition}
Here, $\mathcal{S}^3$ denotes the $3$-dimensional spherical
Hausdorff measure with respect to a metric that is bi-Lipschitz
equivalent to the Kor\'{a}nyi distance $d$. This is a natural
measure to work with since, with respect to the Kor\'{a}nyi
metric, $\partial B$ has Hausdorff dimension $3$, while the entire
space $\mathbb{H}^1$ is $4$-dimensional. However, unlike in
$\mathbb{R}^n$, we will employ several different canonical
measures on $\partial B$, depending on the context. In addition to
the Hausdorff measure $\mathcal{S}^3|_{\partial B}$, these are the
measure $\sigma_0$ appearing in the polar coordinates formula, and
the measure $\sigma$ from the formula for the modulus of a ring
domain, see Section \ref{s:prelim}. These measures qualitatively
differ from each other in a neighborhood of the north and south
pole of the Kor\'{a}nyi sphere, which are the two
\emph{characteristic points} of $\partial B$. The nonisotropic
nature of $\partial B$ is the main reason for the challenges one
faces when extending the $H^p$ theory to $\Hei$.

\medskip

To motivate our definition of Hardy spaces, we first prove that
every quasiconformal map on the Kor\'{a}nyi unit ball belongs to
some $H^p$ space:
\begin{thm}\label{t:QC_hardy_intro}
For every $K\geq 1$, there exists a constant $p_0=p_0(K)>0$ such
that every $K$-quasiconformal map $f:B \to f(B) \subset
\mathbb{H}^1$ belongs to $H^p$ for all $0<p<p_0$.
\end{thm}
This extends a result by Jerison and Weitsman \cite{MR624919} from
$\mathbb{R}^n$ to $\mathbb{H}^1$. In $\mathbb{R}^n$, more precise
information on the admissible exponents is due to Nolder
\cite{MR1943765}, Astala and Koskela \cite{MR2900163}. In the
second part of our paper, we give necessary and sufficient
conditions for the $H^p$ membership of a quasiconformal map $f:B
\to f(B) \subset \mathbb{H}^1$ for a fixed exponent $p$.

\begin{thm}\label{t:main2} Let $0<p<\infty$.
The following conditions are equivalent for a quasiconformal  map
$f:B\to f(B) \subset \mathbb{H}^1$:
\begin{enumerate}
\item $f\in H^p$, \item the nontangential maximal function $Mf$ of
$f$ belongs to $L^p(\mathcal{S}^3|_{\partial B})$, \item the
Kor\'anyi norm of the radial limit $f^{\ast}$ of $f$ belongs to
$L^p(\mathcal{S}^3|_{\partial B})$.
\end{enumerate}
\end{thm}
This extends earlier work by Zinsmeister \cite{MR860655} and by
Astala and Koskela \cite{MR2900163}. The implication "(3) $\Rightarrow$ (2)" in Theorem~\ref{t:main2} is the most challenging,
 and analogously as in $\R^n$, we also obtain quantitative information, see Proposition~\ref{p:FromRadialLimitToNTMax}.
 Our proof combines elements from
 \cite{MR2900163} (the use Hardy-Littlewood maximal functions on $\partial B$) and \cite{MR860655} (the use of specific Carleson measures on $B$).
However, an important tool in \cite{MR860655,MR2900163}, the
subgroup of M\"obius transformations that keep $B$ invariant, is
not flexible enough in our setting. We use as a substitute a
specific class of $1$-quasiconformal maps that do not necessarily
preserve the unit ball, but nonetheless have useful metric
properties.

In order to put Proposition~\ref{p:FromRadialLimitToNTMax} in a
slightly wider perspective, let us mention that similar
inequalities, for a more restrictive range of integrability
exponents, have been studied for harmonic functions e.g., on
Lipschitz domains in $\R^n$ by Dahlberg, see
\cite[(2.1)]{MR592391}. Integral inequalities for nontangential
maximal functions appear also in connection with the Dirichlet
problem for the sub-Laplacian on certain domains in $\mathbb{H}^1$
or more general $H$-type groups, with rough boundary values, see
\cite[Theorem 1.8]{orponen2020subelliptic}, \cite[Theorem
1.1]{MR1890994}.

Our main application of Theorem~\ref{t:main2} is a characterization of Carleson measures on $B$ in terms of radial limits of quasiconformal
 maps on $B$.
\begin{thm}\label{t:NecCarleson_intro}
Assume that $\mu$ is a Carleson measure on $B$. If $f:B \to f(B)
\subset \mathbb{H}^1$ is $K$-quasiconformal, then
\begin{equation}\label{eq:IntIneqFromCarlesonIntro}
\int_B \|f(q)\|^{ p}\,d\mu(q) \leq C \int_{\partial B}
\|f^{\ast}(\omega)\|^p\,d\mathcal{S}^3(\omega),\quad\text{for all
}0<p<\infty
\end{equation}
where $C$ depends only on $p$, $K$, and the Carleson measure
constant of $\mu$. Conversely, for every $K\geq 1$, there exists
$p(K) < 3$ such that if $p>p(K)$ is fixed and $\mu$ is a Borel
measure for which \eqref{eq:IntIneqFromCarlesonIntro} holds for
all $K$-quasiconformal mappings, then $\mu$ is a Carleson measure.
\end{thm}
In fact we prove a more general result for the $\alpha$-Carleson measures, Theorem \ref{t:NecCarleson}, which
can be applied to relate $H^p$ and
 Bergman-type spaces $A^p$, analogously as in \cite[Theorem
 9.1]{MR2900163}, see Theorem~\ref{t: Ap-spaces}.
Theorems \ref{t:NecCarleson_intro} and \ref{t:NecCarleson} are
Heisenberg versions of results by Nolder~\cite{MR1013976,
MR1307591},
  and Astala and Koskela \cite{MR2900163} which in turn were motivated by Carleson's embedding theorem~\cite{MR141789}.
  We prove the first part of Theorem~\ref{t:NecCarleson} in two steps:
(i) a result for Carleson measures and nontangential maximal
functions in rather general metric spaces (Proposition
\ref{p:NontangentialInequality}) and (ii) the relation between
nontangential maximal functions and radial limits of
quasiconformal maps given by Theorem~\ref{t:main2}, or more
specifically, Proposition~\ref{p:FromRadialLimitToNTMax}. The
second part of Theorem \ref{t:NecCarleson} follows by applying
\eqref{eq:IntIneqFromCarlesonIntro} to specific maps $f$ that are
constructed using the $\mathbb{H}^1$ radial stretch map from
\cite{MR3076803}.

\medskip

In the process of proving the stated theorems, we establish
results of independent interest. Extending a theorem by Jones
\cite{MR554817}, we show that for every quasiconformal map $f$ on
$B$ omitting the origin, $|\nabla_H \log \|f(q)\||\,dq$ defines a
Carleson measure on $B$. This is a consequence of Proposition
\ref{l:5.6}. Related to the Heisenberg polar coordinates, we prove
that there exists a parameter $\kappa$ such that every radial
curve segment connecting a point $\omega \in
\partial B$  to the origin stays inside the nontangential approach region  $\Gamma_{\kappa}(\omega)$ in $B$
 (Proposition \ref{p:curve_cone}).
 This is a nontrivial statement
in $\mathbb{H}^1$, due to the
 non-geodesic feature of the radial curves, and the
non-isotropic nature of $\partial B$. The Heisenberg geometry
enters the picture in other ways, too. For instance, rotations do
not act transitively on $\partial B$. In Section \ref{sect-Mobius}
we introduce a family of canonical maps on $B$ that serve as the
mentioned substitutes for the M\"obius self-maps of the unit ball.
Since these maps do not necessarily keep $B$ invariant, we have to
formulate several of the auxiliary results for a more general
class of domains in $\mathbb{H}^1$. This can be done by replacing
explicit computations for the Euclidean unit ball with abstract
arguments using concepts from metric geometry such as
\emph{corkscrew} and \emph{QED} domains.

\medskip

\textbf{Structure of the paper.} In Section \ref{s:prelim} we
introduce preliminaries about the Heisenberg group and
quasiconformal maps, and state the definitions of radial limits
and nontangential maximal functions. In Section \ref{s:QC_in_Hp}
we prove our first main result: quasiconformal maps on the
Kor\'{a}nyi ball belong to Hardy spaces. Section \ref{s:Carleson}
and Appendix \ref{s:radialsDist} contain  auxiliary results on
Carleson measures and radial curves, respectively. In Section
\ref{s:ZinsmeisterChar} we prove our second main result, a
characterization of the $H^p$-membership of a quasiconformal map,
for a fixed exponent $p$. We apply this result in
Section~\ref{s:CarlesonChar} to characterize Carleson measures on
$B$ using radial limits of quasiconformal maps on $B$.

\textbf{Notation.} If $f,g\geq 0$, the notation $f\lesssim g$
signifies the existence of a positive absolute constant $C$ such
that $f\leq C g$. The notation $f\lesssim_p g$ means that $C$ is
allowed to depend on a parameter ``$p$''. Finally, $f\sim g$ is an
abbreviation of $f\lesssim g \lesssim f$.

\medskip
\textbf{Acknowledgments.} We thank Tuomas Orponen for help with
Proposition \ref{p:TopSph}.

\section{Preliminaries}\label{s:prelim}

\subsection{The Heisenberg group}
The \emph{Heisenberg group} $\mathbb{H}^1$ is the set $\mathbb{R}^3$ endowed with the group product
\begin{displaymath}
(x,y,t) \cdot (x',y',t')= (x+x',y+y',t+t'-2 xy'+ 2yx')
\end{displaymath}
for $(x,y,t),(x',y',t')\in \mathbb{R}^3$. Sometimes we will
identify $(x,y)\in \mathbb{R}^2$ with $z=x+\mathrm{i}y
\in\mathbb{C}$. Similarly, if $g=(g_1,g_2,g_3)$ is a
$\mathbb{H}^1$-valued map, we will occasionally denote $g_I= g_1 +
\mathrm{i} g_2$, or $g_I=(g_1,g_2)$. We collect here the
properties of $\mathbb{H}^1$ that are most relevant for this paper
and we refer the reader to \cite{capogna2007introduction} for more
details.

\subsubsection{Metric structure} We denote the \emph{Kor\'{a}nyi
norm} by
\begin{displaymath}
\|(z,t)\| = (|z|^4 + t^2)^{1/4},\quad (z,t)\in \mathbb{H}^1
\end{displaymath}
and the Kor\'{a}nyi unit ball by $B:=\{p\in \mathbb{H}^1:\,
\|p\|<1\}$. The \emph{Kor\'{a}nyi distance} is the left-invariant
metric given by
\begin{displaymath}
d(p,q):= \|q^{-1}\cdot p\|,\quad p,q\in\mathbb{H}^1.
\end{displaymath}
For $q\in \mathbb{H}^1$ and $A\subset \mathbb{H}^1$, we write
$d(q,A):=\inf_{a\in A} d(q,a)$. A curve $\gamma:[a,b]\to
\mathbb{H}^1$ is rectifiable with respect to $d$ if and only if it
is absolutely continuous as a curve in $\mathbb{R}^3$
 and for almost every $s\in [a,b]$ the tangent vector $\dot{\gamma}(s)$ is contained in the \emph{horizontal plane} $H_{\gamma(s)}\mathbb{H}^1$ given by
\begin{displaymath}
H_{(x,y,t)}\mathbb{H}^1 =
\mathrm{span}\left\{\begin{pmatrix}1\\0\\2y\end{pmatrix},\begin{pmatrix}0\\1\\-2x\end{pmatrix} \right\}.
\end{displaymath}
We denote the left-invariant vector fields
\begin{displaymath}
(x,y,t) \mapsto (1,0,2y)\quad\text{and}\quad (x,y,t)\mapsto
(0,1,-2x)
\end{displaymath}
by $X$ and $Y$, respectively, and identify them with the
differential operators $\partial_x + 2y\, \partial_t$ and
$\partial_y - 2x\, \partial_t$, respectively. The \emph{horizontal
gradient} of a function $u:\Omega \to \mathbb{R}$ on an open set
$\Omega \subset \mathbb{H}^1$ is
\begin{displaymath}
\nabla_H u = (Xu) X + (Yu)Y.
\end{displaymath}
We also equip $H_q\mathbb{H}^1$ with the norm $|\cdot|$ defined by
$|a X_q + b Y_q|:= \sqrt{a^2+b^2}$.

\subsubsection{Polar coordinates and radial curves}
The following polar coordinates
 formula was first proved by Kor\'{a}nyi and Reimann \cite{zbMATH04023589},
 and later in greater generality by Balogh and Tyson \cite[Example
 3.11]{MR1942237}. Here and in the following, integration on
 $\mathbb{H}^1$ is performed with respect to the Lebesgue measure
 on $\mathbb{R}^3$, which is a Haar measure of the group
 $\mathbb{H}^1$. We will use the symbol
 $|A|$ to denote the Lebesgue measure of a set $A$.
\begin{thm}\label{t:polar}
There exists a unique Radon measure $\sigma_0$ on $\partial
B\setminus \{z=0\}$ such that for all $u\in L^1(\mathbb{H}^1)$,
\begin{equation}\label{eq:polarCoordsForm}
\int_{\mathbb{H}^1} u(q)\, dq = \int_{\partial B \setminus
\{z=0\}} \int_0^{\infty} u(\gamma(s,\omega)) s^3\,ds\,
d\sigma_0(\omega),
\end{equation}
where the \emph{radial curves} are given by the horizontal curves
\begin{displaymath}
\gamma(s,(z,t)) = \left(sz e^{-\mathrm{i}\frac{t}{|z|^2}\log
s},s^2 t\right),\quad (z,t)\in \partial B\setminus \{z=0\}.
\end{displaymath}
Moreover, using the parametrization
\begin{equation}\label{eq:spherical_coords}
\partial B \setminus \{z=0\} = \left\{(z,t)=(\sqrt{\cos \alpha}e^{\mathrm{i}\varphi},\sin \alpha):\,
-\tfrac{\pi}{2}<\alpha<\tfrac{\pi}{2},\,0\leq
\varphi<2\pi\right\},
\end{equation}
the measure $\sigma_0$ takes the form $ d\sigma_0 = d\alpha\,
d\varphi. $
\end{thm}
Abusing notation, we will often identify a set $E\subset \partial
B\setminus \{z=0\}$ with the corresponding set in the parameter
space $(-\pi/2,\pi/2)\times [0,2\pi)$.

\subsubsection{Modulus of curve families}\label{ss:mod}
The modulus of curve families is a crucial tool in this paper. We
refer to \cite{MR2466579} for a detailed discussion, and only
recall the relevant definitions and properties. Given a family
$\Gamma$ of curves in $\mathbb{H}^1$, a Borel function
$\rho:\mathbb{H}^1 \to [0,+\infty]$ is said to be
\emph{admissible} for $\Gamma$, denoted $\rho\in
\mathrm{adm}(\Gamma)$, if $\int_{\gamma} \rho \,\mathrm{d}s\geq 1$
for all locally rectifiable $\gamma\in \Gamma$. The
\emph{$4$-modulus of $\Gamma$} is then defined as
\begin{displaymath}
\mathrm{mod}_4(\Gamma) := \inf_{\rho \in \mathrm{adm}(\Gamma)}
\int_{\mathbb{H}^1} \rho^4(q)\, dq.
\end{displaymath}
Given sets $U\subset \Hei$ and $E,F\subset U$, we denote by $\Gamma(E,F,U)$ the family of all curves contained in $U$ that connect $E$ and $F$. Kor\'{a}nyi and Reimann proved in \cite{zbMATH04023589} that the family
\[
\Gamma_{a,b}:=\Gamma(\partial B(0,a), \partial B(0, b), \Hei)
\]
of all curves joining the Kor\'{a}nyi
spheres $\partial B(0,a)$ and $\partial B(0,b)$ for $0<a<b<\infty$
has modulus
\begin{equation}\label{eq:ring}
\mathrm{mod}_4(\Gamma_{a,b}) = \pi^2 \left(\log
\frac{b}{a}\right)^{-3}.
\end{equation}
A slight modification of the argument in \cite{zbMATH04023589}
yields the following useful formula concerning the radial curves
introduced in Theorem \ref{t:polar}.
\begin{proposition}\label{p:Mod_Radial_subet}
Fix $0<r<1$ and a Borel set $E\subset \partial B\setminus
\{z=0\}$. If $\Gamma$ denotes the family of radial curves joining
$\partial B(0,r)$ to $E$, then
\begin{equation}\label{eq:mod_formula}
\mathrm{mod}_4(\Gamma) = \sigma(E) \left(\ln
\frac{1}{r}\right)^{-3},
\end{equation}
where $ d \sigma = \cos^2 \alpha\, d\alpha\,d\varphi$ in the
coordinates given by \eqref{eq:spherical_coords}.
\end{proposition}

\begin{proof} The proof is almost verbatim the same as for
\eqref{eq:ring} with $a=r$, $b=1$, observing that $\sigma(\partial
B \setminus \{z=0\})=\pi^2$. The inequality ``$\geq$'' in
\eqref{eq:mod_formula} is proven like
\cite[(4.5)]{zbMATH04023589}, with the only difference that
instead of integrating over $(\alpha,\varphi)\in
(-\pi/2,\pi,2)\times [0,2\pi]$, the domain of integration now
corresponds to the set $E$. To prove the converse inequality, we
follow the proof of \cite[(4.6)]{zbMATH04023589}, but observe that
it suffices to consider radial curves, instead of arbitrary
rectifiable ones. Moreover, the domain of integration is now
restricted to the part of the annulus $\{q\in \mathbb{H}^1:\,r\leq
\|q\|\leq 1\}$ foliated by segments of the radial curves passing
through the set $E$.
\end{proof}

\subsubsection{Measures on the Kor\'{a}nyi unit sphere} The two points $q_{\pm}:=(0,0,\pm 1)\in \partial B$ are
\emph{characteristic points} of the Kor\'{a}nyi sphere: the
tangent planes to the surface $\partial B$ agree with the
horizontal planes $H_{\pm q}\mathbb{H}^1$ at the respective
points. The distinguished role of these points is reflected in the
behavior of the measures that we consider  on the Kor\'{a}nyi unit
sphere outside $q_{\pm}$: the measure $\sigma_0$ from the polar
coordinates formula \eqref{eq:polarCoordsForm}, the measure
$\sigma$ from the modulus formula \eqref{eq:mod_formula}, and the Hausdorff measure
$\mathcal{S}^3|_{\partial B\setminus \{z=0\}}$. The latter is a
restriction of the $3$-dimensional spherical Hausdorff measure
computed with respect to the metric $d_{\infty}$ induced by the
gauge function $\|(z,t)\|_{\infty}= c \max\{|z|,\sqrt{|t|}\}$ for
a suitable constant $c>0$. Since $d_{\infty}$ is bi-Lipschitz
equivalent to the Kor\'{a}nyi distance $d$, we might as well use
the standard $3$-dimensional Hausdorff measure with respect to $d$
for our purposes, but $\mathcal{S}^3|_{\partial B}$ is related to
the horizontal perimeter measure of $B$ and we can thus use some
results from \cite{MR1871966}.

\begin{lemma}\label{l:Haus}
If we parametrize $\partial B \setminus \{z=0\}$ as in
\eqref{eq:spherical_coords}, then
\begin{displaymath}
d\mathcal{S}^3|_{\partial B\setminus \{z=0\}} = \sqrt{\cos
\alpha}\, d\alpha\, d\varphi.
\end{displaymath}
\end{lemma}

\begin{proof} Since  the boundary $\partial B$ of the Kor\'{a}nyi
unit  ball is of class  $C^1$, we know by \cite[Corollary
7.7]{MR1871966} that, if the constant $c$ in the definition of
$\|\cdot\|_{\infty}$ is chosen suitably, we have
\begin{equation}\label{eq:formSpherical}
d\mathcal{S}^3|_{\partial B} =  |Cn_B|\,d\mathcal{H}^2|_{\partial
B},
\end{equation}
where $n_B$ is the Euclidean outward unit normal to the
Kor\'{a}nyi unit sphere, $\mathcal{H}^2$ denotes the
$2$-dimensional Euclidean Hausdorff measure and
\begin{displaymath}
C(x,y,t):= \begin{pmatrix}1&0&2y\\0&1&-2x\end{pmatrix}.
\end{displaymath}
To prove the lemma, it suffices to express the right-hand side of
\eqref{eq:formSpherical} using the coordinates $(\alpha,\varphi)$
from \eqref{eq:spherical_coords}. First, since $n_B = \nabla
v/|\nabla v|$ with $v(x,y,t):=\|(x,y,t)\|^4=(x^2+y^2)^2+t^2$, we
compute for $(x,y,t)=(\sqrt{\cos \alpha}\cos \varphi,\sqrt{\cos
\alpha}\sin\varphi,\sin \alpha)\in \partial B\setminus \{z=0\}$
that
\begin{displaymath}
\nabla v(x,y,t) = \begin{pmatrix}4(x^2+y^2)x\\ 4(x^2+y^2)y\\
2t\end{pmatrix}=  \begin{pmatrix}4\cos^{3/2}\alpha
\cos \varphi\\
4\cos^{3/2}\alpha \sin \varphi\\2\sin \alpha\end{pmatrix}
\end{displaymath}
and
\begin{displaymath}
C(x,y,t) \nabla v(x,y,t) = \begin{pmatrix}4(x^2+y^2)x+4yt\\
4(x^2+y^2)y-4xt\end{pmatrix}= 4 \begin{pmatrix}\cos^{3/2}\alpha
\cos \varphi+ \cos^{1/2}\alpha \sin \alpha \sin \varphi\\
\cos^{3/2}\alpha \sin \varphi-\cos^{1/2}\alpha \sin \alpha \cos
\varphi\end{pmatrix}.
\end{displaymath}
Hence
\begin{equation}\label{eq:AreaForm1}
|Cn_B| = \frac{|C \nabla v|}{|\nabla
v|}=\frac{|\nabla_{\mathbb{H}}v|}{|\nabla v|}=\frac{\sqrt{\cos
\alpha}}{\sqrt{\cos^3 \alpha + \frac{\sin^2 \alpha}{4}}}.
\end{equation}
On the other hand, using the parametrization $
\Phi(\alpha,\varphi) = (\sqrt{\cos \alpha}
\cos \varphi,\sqrt{\cos \alpha}\sin\varphi,\sin \alpha), $ we
find that the surface measure is given by
\begin{equation}\label{eq:AreaForm2}
d\mathcal{H}^2|_{\partial B\setminus \{z=0\}}= |\Phi_{\alpha}\times
\Phi_{\varphi}|\,d\alpha d\varphi = \sqrt{\cos^3 \alpha +
\frac{\sin^2 \alpha}{4}}\,d\alpha d\varphi.
\end{equation}
Inserting \eqref{eq:AreaForm1} and \eqref{eq:AreaForm2} in formula
\eqref{eq:formSpherical} then  yields the claim.
\end{proof}

In addition to its simple expression in the coordinates
$(\alpha,\varphi)$, the measure $\mathcal{S}^3|_{\partial B}$ has
another useful feature:

\begin{lemma}\label{l:3Reg}
The measure  $\mathcal{S}^3|_{\partial B}$ is $3$-regular, that
is, there exists a constant $C\geq 1$ such that
\begin{displaymath}
C^{-1} r^3 \leq \mathcal{S}^3(B(p,r)\cap \partial B) \leq C
r^3,\quad \text{for all }p\in \partial B,\,
0<r<\mathrm{diam}(\partial B).
\end{displaymath}
\end{lemma}

The lemma follows e.g., from \cite[Theorem 6.2]{MR2229731}.
Alternatively, one can also obtain it by combining \cite[Theorem
5.1]{MR1664688},  \cite[Propositions 9 and 10]{MR1658616} and
\cite[Proposition 4.1]{MR4127898}.

The formulae for $\sigma_0$, $\sigma$, and
$\mathcal{S}^3|_{\partial B\setminus \{z=0\}}$ given in Theorem
\ref{t:polar}, Proposition \ref{p:Mod_Radial_subet},
 and Lemma \ref{l:Haus}, respectively,
show that
\begin{equation}\label{eq:measure_monotone}
\sigma(E) \leq \mathcal{S}^3(E)\leq \sigma_0(E),\quad\text{for all
 Borel sets }E\subseteq \partial B\setminus \{z=0\}.
\end{equation}
The difference between the measures becomes more pronounced for
sets $E$ concentrated near the characteristic points of $\partial
B$ as such sets are harder to reach by short radial curves from
inside the ball. Indeed, the curves $\gamma(\cdot,(z,t))$ defined
in Theorem \ref{t:polar} begin to spiral more as $z$ tends to $0$,
approaching in the limit the $t$-axis, a curve that fails to be
locally rectifiable with respect to the Kor\'{a}nyi metric.
However,since ${\cos \alpha}>0$ for $\alpha \in (-\pi/2,\pi/2)$,
the three measures are still mutually absolutely continuous:
\begin{equation}\label{eq:ac}
\sigma(E)=0\quad\Leftrightarrow\quad \sigma_0(E)=0
\quad\Leftrightarrow\quad \mathcal{S}^3(E)=0,\quad\text{for all
}E\subset \partial B\setminus\{z=0\}\text{ Borel}.
\end{equation}

\subsubsection{Nontangential regions and nontangential
maximal functions} To study how a mapping on the Kor\'{a}nyi unit
ball $B$ behaves close to the boundary $\partial B$, we use two
different tools: \emph{nontangential maximal functions} and
\emph{radial limits}. While the former are purely metric concepts,
our definition of radial limit is tailored specifically to the
Kor\'{a}nyi unit ball, as it makes use of the \emph{radial curves}
in Theorem \ref{t:polar}. The geometry of $B$ is also reflected in
the choice of the parameter $\kappa$ for which we apply the
definition of nontangential region and associated maximal
function, see Proposition~\ref{p:curve_cone}. However, the
nontangential maximal functions will not appear in our final
result, Theorem~\ref{t:NecCarleson}, where they are only used as
tools along the way. The first step in the proof of
Theorem~\ref{t:NecCarleson}, namely Proposition
\ref{p:NontangentialInequality}, works in abstract proper metric
spaces, which is why we state the following definitions in this
generality.

\begin{definition}\label{def:NTregion}\label{d:NTregion_kappa0}
Let $(X,d)$ be a  metric space, and $\Omega$ a fixed nonempty
domain in $X$. For a point $\omega \in
\partial \Omega$ and a parameter $\kappa>0$,
 we define the \emph{nontangential region in $\Omega$ with parameter $\kappa$ centered at $\omega$} as follows:
\begin{displaymath}
\Gamma_{\Omega,\kappa}(\omega):=\{x \in \Omega:\; d(x,\omega) <
(1+\kappa)d(x,\partial \Omega)\}.
\end{displaymath}
If $(X,d)=(\mathbb{H}^1,d)$, $\Omega=B$, and $\kappa$ is as in
Proposition~\ref{p:curve_cone}, we often abbreviate
\begin{displaymath}
\Gamma(\omega):=
\Gamma_{\kappa}(\omega)=\Gamma_{B,\kappa}(\omega).
\end{displaymath}
\end{definition}

If $(X,d)$ is the Euclidean plane, and $\Omega$ the open unit
disk, then $\Gamma_{\kappa,\Omega}(\omega)$ is a \emph{Stolz
region} (or \emph{nontangential approach region}) in the usual
sense. In general, we do not require the domain $\Omega$ in
Definition \ref{def:NTregion} to have the interior corkscrew
property or satisfy other geometric conditions, but in our main
application, the nontangential regions carry relevant information
thanks to the following result.

\begin{proposition}\label{p:curve_cone} Let $B$ be the Kor\'{a}nyi
unit ball in $\mathbb{H}^1$. There exists $\kappa>0$ such that for
every $\omega \in
\partial B \setminus \{z=0\}$,
\begin{displaymath}
\gamma(s,\omega) \in \Gamma_{B,\kappa}(\omega),\quad\text{for all
}s\in (0,1).
\end{displaymath}
\end{proposition}

We postpone the proof to the Appendix.

\begin{remark}
By \cite[Corollary 1]{MR1658616} and by the discussion in the
beginning of Section 6 in~\cite{MR1404326}, we have that $B$ is a
\emph{John domain}. This already shows that there is a constant
$\kappa$, and for every $\omega \in
\partial B$, a rectifiable curve in $B$ emanating from $\omega$ and
contained in $\Gamma_{B,\kappa}(\omega)$ until it hits the John
center of $B$. The purpose of Proposition \ref{p:curve_cone} is to
provide specific information about the \emph{radial} curves
$\gamma(\cdot,\omega)$.
\end{remark}

\begin{definition}\label{def:sph-cap} Let $(X,d)$ be a  metric space, and $\Omega$ a fixed
domain with nonempty boundary in $X$. For a point $x \in \Omega$
and a parameter $\kappa>0$,
 we define the \emph{shadow} associated to $ x$ and $\kappa>0$ as
\begin{displaymath}
S_{\Omega,\kappa}(x) := \partial \Omega \cap B(x,
(1+\kappa)d(x,\partial \Omega)).
\end{displaymath}
If $(X,d)=(\mathbb{H}^1,d)$, $\Omega=B$, and $\kappa$ is as in
Proposition~\ref{p:curve_cone}, we often abbreviate
\begin{displaymath}
S(q):= S_{\kappa}(q):= S_{B,\kappa}(q), \end{displaymath} and we
also call the shadow a \emph{spherical cap} in this case.
\end{definition}

The definition is tailored so that nontangential regions and
shadows are related in the following way:
\begin{displaymath}
x \in \Gamma_{\Omega,\kappa}(\omega) \quad \Leftrightarrow \quad
\omega \in S_{\Omega,\kappa}(x).
\end{displaymath}

\begin{definition}\label{d:NT} Let $(X,d)$ be a  metric space, and $\Omega$ a fixed nonempty
domain in $X$ with nonempty boundary $\partial \Omega$. If
$\kappa>0$ is such that $\Gamma_{\Omega,\kappa}(\omega)\neq
\emptyset$ for all $\omega \in \partial \Omega$, we define the
\emph{$\kappa$-nontangential maximal function}
\begin{displaymath}
N_{\Omega,\kappa}h(\omega):= \sup_{x\in
\Gamma_{\Omega,\kappa}(\omega)} h(x),\quad \omega \in
\partial \Omega
\end{displaymath}
of a function $h: \Omega \to [0,+\infty)$.

If $(X,d)=(\mathbb{H}^1,d)$, $\Omega=B$, and $\kappa$ is as in
Proposition~\ref{p:curve_cone}, we define the \emph{nontangential
maximal function of $f:B \to \mathbb{H}^1$} as
\begin{displaymath}
M(f)(\omega):= M_{\kappa}(f)(\omega):= N_{B,\kappa} \|f\|(\omega)
= \sup_{q\in \Gamma_{B,\kappa}(\omega)}\|f(q)\|.
\end{displaymath}
\end{definition}

\begin{remark}\label{r:NTMaxMeas} For $(X,d)$, $\Omega\subset X$,
$\kappa>0$ and $h:\Omega \to [0,+\infty)$ as in Definition
\ref{d:NT}, the nontangential maximal function
$N_{\Omega,\kappa}h:
\partial \Omega \to [0,+\infty]$ is lower semicontinuous.
Indeed, if $\lambda \in \mathbb{R}$ and $\omega \in \partial
\Omega$ are such that $N_{\Omega,\kappa} h(\omega)>\lambda$, then
there exists $x\in \Gamma_{\Omega,\kappa}(\omega)$ such that
$h(x)>\lambda$. Since
\begin{displaymath}
x\in \Gamma_{\Omega,\kappa}(\omega')\quad\text{for all }\omega'\in
\partial \Omega\text{ with }d(\omega,\omega')<\left[(1+\kappa)d(x,\partial
\Omega)-d(x,\omega)\right],
\end{displaymath}
we see that $N_{\Omega,\kappa} h(\omega')>\lambda$ for all
$\omega'$ in a relatively open neighborhood of $\omega$ in
$\partial \Omega$.
\end{remark}

\subsection{Quasiconformal mappings}
A homeomorphism $f:\Omega \to \Omega'$ between domains in
$\mathbb{H}^1$ is \emph{quasiconformal} if
\begin{displaymath}
H_f(q):= \limsup_{r\to 0}
\frac{\max_{d(q,q')=r}d(f(q),f(q'))}{\min_{d(q,q')=r}d(f(q),f(q'))}
\end{displaymath}
is uniformly bounded on $\Omega$. We say that $f$ is
\emph{$K$-quasiconformal} for a constant $K\geq 1$, if
\begin{displaymath}
\|H_f\|_{\infty}:= \mathrm{ess sup}_{q\in \Omega} H_f(q)\leq K,
\end{displaymath}
where the essential supremum is computed with respect to the
Lebesgue measure on $\mathbb{R}^3$.

We refer to the literature for equivalent characterizations of
quasiconformal mappings and simply recall that a
$K$-quasiconformal map $f:\Omega \to \Omega'$ between domains in
$\mathbb{H}^1$ has the following properties:
\begin{enumerate}
\item The map $f$ is absolutely continuous along $\mathrm{mod}_4$
a.e.\ curve in $\Omega$, and there exists a constant $K'$,
depending only on $K$, such that for each curve family $\Gamma$ in
$\Omega$,
\begin{equation}\label{eq:mod}
\frac{1}{K'} \mathrm{mod}_4(\Gamma)\leq
\mathrm{mod}_4(f(\Gamma))\leq K' \mathrm{mod}_4(\Gamma).
\end{equation}
 \item The components $f_1,f_2,f_3$ all
belong the \emph{horizontal Sobolev space}
$HW_{loc}^{1,4}(\Omega)$ of $L_{loc}^{4}(\Omega)$-functions with
weak $X$ and $Y$ derivatives in $L_{loc}^{4}(\Omega)$, they
satisfy the \emph{contact conditions} $Xf(q) , Yf (q)\in
H_q\mathbb{H}^1$ for almost every $q\in \Omega$, and there exists
a constant $K''$, depending only on $K$, such that
\begin{equation}\label{eq:Dist_Ineq}
|D_Hf(q)|^4 \leq K'' J_f(q),\quad\text{for almost every }q\in
\Omega,
\end{equation}
where the \emph{operator norm} is $|D_H f(q)|:= \sup_{\xi \in
H_q\mathbb{H}^1, |\xi|=1}|D_H f(q)\xi|$, and the \emph{formal
horizontal derivative} is given with respect to the frame
$\{X,Y\}$ by
\begin{displaymath}
D_H f(q):=
\begin{pmatrix}Xf_1(q)&Yf_1(q)\\Xf_2(q)&Yf_2(q)\end{pmatrix},
\end{displaymath}
and $J_f(q)= \left(\det D_H f(q)\right)^2$.
\end{enumerate}
These properties follow from
\cite{Pansu1,MR1317384,zbMATH00746069}, see also \cite{MR1778673},
and \cite[Section 9]{MR1869604} for a discussion in abstract
metric measure spaces of locally $Q$-bounded geometry. If the
quasiconformal map $f$ is also a diffeomorphism, then $J_f$ agrees
with its standard Jacobian, see \cite[Section 2.3]{MR1317384}.

\subsubsection{Radial limits}
\begin{definition}
The \emph{radial limit} of a map $f:B\subset \mathbb{H}^1 \to
\mathbb{H}^1$ is defined as
\begin{displaymath}
f^{\ast}(\omega):=\lim_{s  \to 1}f(\gamma(s,\omega))
\end{displaymath}
for all $\omega \in \partial B$, where the limit exists.
\end{definition}

\begin{lemma}\label{l:QC_RadialLimit}
If $f:B \to f(B)\subseteq \mathbb{H}^1$ is quasiconformal, then
the radial limit $f^{\ast}$ exists almost everywhere on $\partial
B\setminus \{z=0\}$ with respect to any of the measures
$\sigma_0$, $\sigma$, and $\mathcal{S}^3|_{\partial B \setminus
\{z=0\}}$. Moreover, $f^{\ast}$ is Borel measurable.
\end{lemma}
\noindent By \eqref{eq:mod_formula}, the proof is a
straightforward adaptation of \cite[p.21]{MR2900163}. Polar
coordinates were also used  in connection with radial limits at
$\infty$ for homogeneous Sobolev functions
\cite{https://doi.org/10.48550/arxiv.2203.01020}.

\begin{proof}
If $\omega \in \partial B \setminus \{z=0\}$ is such that $f\circ
\gamma(\cdot,\omega):[1/2,1) \to \mathbb{H}^1$ is rectifiable,
then it extends to a rectifiable curve on the closed interval
$[1/2,1]$ (see e.g.\ \cite[Theorem 2.1]{2015arXiv151209165A}), and
consequently, $f^{\ast}(\omega)$ is defined. In other words, the
set $A_0$ of points  $\omega \in
\partial B\setminus \{z=0\}$ for which the radial limit
$f^{\ast}(\omega)$ does \emph{not} exist is a subset of the set
$E$ of points $\omega$ in $\partial B\setminus \{z=0\}$ for which
$f\circ \gamma(\cdot,\omega):[1/2,1) \to \mathbb{H}^1$ fails to be
rectifiable.

Now let $A\subset E$ be an arbitrary Borel set and denote by
$\Gamma_A$ the family of radial curves connecting $\partial
B(0,1/2)$ to $A$. Then, since the modulus of non-rectifiable
curves is zero, see e.g.\ \cite[Proposition 5.3.3.]{MR3363168}, we
have $\mathrm{mod}_4(f(\Gamma_A))=0$ and hence
$\mathrm{mod}_4(\Gamma_A)=0$ by the quasiconformality of $f$ and
\eqref{eq:mod}. Then formula \eqref{eq:mod_formula} implies that
$\sigma(A)=0$. This holds for arbitrary Borel sets $A\subset E$,
in particular for the set $A_0$ defined above, which is indeed a
Borel set by standard arguments.
Thus we conclude that $f^{\ast}$ is defined on the
Borel set $\partial B \setminus [\{z=0\}\cup A_0]$ with
$\sigma(A_0)=0$, and hence also
$\sigma_0(A_0)=\mathcal{S}^3(A_0)=0$ by \eqref{eq:ac}. Moreover,
since each component of $f^{\ast}$ is the limit of a convergent
sequence of Borel functions, $f^{\ast}$ itself is Borel
measurable.
\end{proof}


\section{Quasiconformal maps belong to Hardy spaces}\label{s:QC_in_Hp}
In this section we prove Theorem \ref{t:QC_hardy_intro}, which
states that every $K$-quasiconformal map on the Kor\'{a}nyi unit
ball is of class $H^p$, for all $p$ less than a threshold that
depends only on the distortion $K$. We show this by establishing
the following counterpart of a result by Jerison and Weitsman
\cite[Theorem 1]{MR624919}. Obviously Theorem \ref{t:QC_hardy}
implies  Theorem \ref{t:QC_hardy_intro}.

\begin{thm}\label{t:QC_hardy}
For every $K\geq 1$, there exists a constant $p_0=p_0(K)>0$ such
that every $K$-quasiconformal map $f:B \to f(B) \subset
\mathbb{H}^1$ satisfies
\begin{displaymath}
\limsup_{r\to 1} \int_{\partial B} \|f(\gamma(r,\omega))\|^p\,
d\mathcal{S}^3(\omega) <\infty\quad\text{for all }p<p_0,
\end{displaymath}
where $\gamma(\cdot,\omega)$, $\omega \in \partial B \setminus
\{z=0\}$, are the radial curves given by Theorem \ref{t:polar}.
\end{thm}

\subsection{Consequences of modulus estimates} Readers interested
in the main argument used to prove Theorem \ref{t:QC_hardy} may
want to proceed directly to Section \ref{s:3_2}. The purpose of
the present section is to establish growth estimates for
quasiconformal maps on $B$ that will be used as a tool in the
proof of Theorem \ref{t:QC_hardy}. Analogous estimates for balls
in $\mathbb{R}^n$ have been proven using Euclidean techniques,
such as spherical symmetrization and Gr\"otzsch rings
\cite{MR1271759,MR2900163,MR538089}. Our proof relies on standard
estimates for abstract Ahlfors regular Loewner spaces. This allows
us to formulate the statement not only for the unit ball, but also
for bounded $1$-quasiconformal images thereof that contain the
origin. This generalization is possible since the Kor\'{a}nyi ball
and its conformal images are \emph{quasiextremal distance domains
(QED)} in the terminology of \cite{MR833411},
\cite[(13.33)]{MR2466579}, with a universal constant. Our use for
this property, and a challenge in proving it, is that the
Kor\'{a}nyi inversion with center at $0$ does not keep the sphere
$\partial B$ pointwise fixed. Hence simple reflection arguments as
in $\mathbb{R}^n$ are not available. On the other hand, it is then
natural to discuss the proof for more general domains than $B$,
following the reasoning used to prove \cite[Lemma 3.6]{MR2466579}
in $\mathbb{R}^n$.

Adapting the Euclidean terminology \cite{MR631089}, we say that a
domain $D \subset \mathbb{H}^1$ is an \emph{extension domain for
the Dirichlet energy space (EDE)} if there exists a bounded
extension operator
\begin{equation}\label{eq:extOp}
\mathrm{ext}: \mathcal{L}^1_4(D) \to
\mathcal{L}^1_4(\mathbb{H}^1)\quad \text{with}\quad
\|\mathrm{ext}(u)\|_{\mathcal{L}^1_4(\mathbb{H}^1)} \leq C
\|u\|_{\mathcal{L}^1_4(D)},\quad \mathrm{ext}(u)|_D = u,
\end{equation}
for the \emph{homogeneous horizontal Sobolev space}
$\mathcal{L}^1_4$, i.e.\,, for the semi-normed space of locally
integrable functions with weak $X$ and
$Y$ derivatives in $L^4$.

\begin{lemma}\label{l:QED} If $D\subset \mathbb{H}^1$ is an extension domain for
the Dirichlet energy space, then it is a quasiextremal distance
domain, quantitatively, that is, for all disjoint nonempty
continua $E$ and $F$ in $D$, it holds
\begin{equation}\label{eq:QED condition_thm}
\mathrm{mod}_4(\Gamma(E,F,\mathbb{H}^1)) \leq C \mathrm{mod}_4
(\Gamma(E,F,D)),
\end{equation}
where $C$ is the constant in \eqref{eq:extOp}.

In particular, all domains that arise as images of the Kor\'{a}nyi
ball $B$ under $1$-quasiconformal maps $T:B \to T(B) \subset
\mathbb{H}^1$ are QED with a constant $C$ that does not depend on
$T$.
\end{lemma}

The first part is based on results about $4$-capacities. The QED
property of $B$ follows from work by Lu~\cite{MR1787096}, see also
Greshnov \cite{MR1603290}, which extends a result by Jones
\cite[Theorem 2]{MR631089} to Carnot groups. The QED property of
$T(B)$ is then a consequence of a Liouville-type theorem in
$\mathbb{H}^1$ and conformal invariance of the $4$-modulus.
\begin{proof}
Let $D$ be an EDE as in the first part of the Lemma. To establish
\eqref{eq:QED condition_thm} for $D$, we follow the same reasoning
as used to prove \cite[Lemma 3.6]{MR2466579} in $\mathbb{R}^n$.
The existence of the extension operator \eqref{eq:extOp} has
immediate consequences for capacities. For a domain $U\subset
\mathbb{H}^1$, and nonempty disjoint compact sets $C_0,C_1\subset
U$, we define the \emph{$4$-capacity}
\begin{equation}\label{eq:def_cap}
\mathrm{cap}_4(C_0,C_1;U):= \inf_{u \in W} \int_{U} |\nabla_H
u(q)|^4 \, dq,
\end{equation}
where $W:=W(C_0,C_1;U)=\{u\in  C(U)\cap \mathcal{L}^1_4(U):\,
u|_{C_0}\leq 0\text{ and }u|_{C_1}\geq 1\}$. We apply this
definition first for $U=D$, $C_0=E$ and $C_1=F$ as in the
statement of the lemma. Thus, for every $\varepsilon>0$, there
exists $u\in W(E,F;D)$ such that
\begin{displaymath}
\int_D |\nabla_H u(q)| \, dq \leq \mathrm{cap}_4(E,F;D) +
\frac{\varepsilon}{2}.
\end{displaymath}
Our goal is to control the left-hand side of the inequality from
below by $\mathrm{cap}_4(E,F;\mathbb{H}^1)$. To achieve this, we
will apply the extension operator \eqref{eq:extOp} to a
modification $v$ of $u$, so as to obtain a competitor in
$W(E,F;\mathbb{H}^1)$. Namely, we choose $r>0$ small enough such
that $ v:= \frac{1+r}{1-r} (u-r) $ satisfies
\begin{equation}\label{eq:competitior_v}
\int_{D} |\nabla_H v(q)|^4 \, dq \leq \mathrm{cap}_4(E,F;D) +
\varepsilon,\quad v|_E\leq -r,\quad v|_F\geq 1+r.
\end{equation}
Clearly, $v \in \mathcal{L}^1_4(D)$ with $\nabla_H v =
\frac{1+r}{1-r}\,\nabla_H u$. By \eqref{eq:extOp}, there exists
$\mathrm{ext}(v) \in \mathcal{L}^1_4(\mathbb{H}^1)$ with
\begin{equation}\label{eq:intermediate}
\mathrm{ext}(v)|_D =v\quad \text{and}\quad \int_{\mathbb{H}^1}
|\nabla_H \mathrm{ext}(v)(q)|^4 \,dq\leq C \int_D |\nabla_H
v(q)|^4\, dq,
\end{equation}
for the constant $C$ given by \eqref{eq:extOp}. Even if $v$ is
continuous on $D$, a priori, $\mathrm{ext}(v)$ is only an element
in $\mathcal{L}^1_4(\mathbb{H}^1)$ and does not necessarily have a
continuous representative on $\mathbb{H}^1$. However, there exists
a sequence $(v_n)_{n\in \mathbb{N}} \subset
C^{\infty}(\mathbb{R}^3)$ such that
\begin{displaymath}
v_n \to \mathrm{ext}(v)\;\text{ locally in }
L^1(\mathbb{R}^3)\quad\text{and}\quad \nabla_H v_n \to \nabla_H\,
\mathrm{ext}(v)\;\text{ in }L^4(\mathbb{R}^3),
\end{displaymath}
see \cite[Section 3.1]{MR1317384}, and also \cite[Chapter
1.B]{MR657581}. Recalling that $\mathrm{ext}(v)|_{E}\leq -r$ and $\mathrm{ext}(v)|_{F}\geq 1+r$ almost everywhere, this allows us
to choose a (smooth) function $\varphi \in W(E,F;\mathbb{H}^1)$
with
\begin{displaymath}
\int_{\mathbb{H}^1}|\nabla_H \varphi(q)|^4 \, dq \leq
\int_{\mathbb{H}^1} |\nabla_H \mathrm{ext}(v)(q)|^4 \,dq+
\varepsilon.
\end{displaymath}
Combining this with \eqref{eq:competitior_v} and
\eqref{eq:intermediate}, and then letting $\varepsilon$ tend to
$0$, we deduce that
\begin{displaymath}
\mathrm{cap}_4(E,F;\mathbb{H}^1) \leq C\, \mathrm{cap}_4(E,F;D).
\end{displaymath}
To conclude the proof of the QED property of $D$, it suffices to
replace the $4$-capacities in the above estimate by the
$4$-modulus of the curve families $\Gamma(E,F;\mathbb{H}^1)$ and
$\Gamma(E,F;D)$, respectively. In $\mathbb{R}^n$, the
corresponding modulus-capacity equality is due to Hesse
\cite{MR379871}. There are several related statements for
$\mathbb{H}^1$, but the assumptions for instance in
\cite{zbMATH00746069} and Eichman's result \cite{Eichmann} cited
in \cite[Section 3.2]{MR1317384} are a bit different from ours.
Instead, we apply \cite[Theorem 1.1]{MR1833251}, which holds for
$p$-modulus and all disjoint compact non-empty sets $C_0$ and
$C_1$ in arbitrary domains $U$ of proper $\varphi$-convex metric
measure spaces with a doubling measure supporting a
$(1,p)$-Poincar\'{e} inequality with $1<p<\infty$. This class of
metric measure spaces includes $(\mathbb{H}^1,d)$ with the
Lebesgue measure for $p=4$. The caveat is that the capacities
appearing in \cite{MR1833251} are defined using upper gradients,
specifically,
\begin{equation}\label{eq:def_cap2}
\mathrm{Cont}-\mathrm{Cap}_4(C_0,C_1;U)=\inf_{g} \int_U g(q)^4\,
dq,
\end{equation}
where the infimum is taken over all non-negative Borel functions
$g$ that are upper gradients -- or weak upper gradients -- of
functions $u\in C(U)$ with the property $u|_{C_0} \leq 0$ and
$u|_{C_0}\geq 1$. By the proof of \cite[Proposition
C.12]{kleiner2021pansu},
 the definitions in \eqref{eq:def_cap} and \eqref{eq:def_cap2}
yield the same result.  Thus we have shown that every EDE
$D\subset \mathbb{H}^1$ is indeed a QED domain.

The second part of the lemma can be deduced by applying the
previous statement to the Kor\'{a}nyi unit ball $D=B$. The EDE
property of $B$ follows from \cite[Theorem C]{MR1787096} since $B$
is uniform and hence an $(\epsilon, \infty)$-domain in the sense
of \cite[Definition 1.1]{MR1787096}; see also \cite[Theorem
5]{MR1603290}, \cite[Theorem 3.2.]{MR2135732}, and the comment
below \cite[Theorem 1.1]{MR1848049}. Then every domain $T(B)$ as
in the lemma is also a QED domain, with the same constant, by
conformal invariance of $\mathrm{mod}_4$. Indeed, by a version of
Liouville's theorem \cite[Corollary 1.4]{MR1459590}, $T$ is the
restriction of a conformal self-map of the one-point
compactification $\widehat{\mathbb{H}}^1$. The claim \eqref{eq:QED
condition_thm} for $D=T(B)$ follows from the QED property of $B$
by applying $T$ to the curve families
$\Gamma(T^{-1}(E),T^{-1}(F),B)$ and
$\Gamma(T^{-1}(E),T^{-1}(F),\mathbb{H}^1)$, recalling that the
family of all nonconstant curves passing through one point in
$\mathbb{H}^1$ has vanishing $4$-modulus by \cite[Corollary
7.20]{heinonen2012lectures}, so we may ignore the points
$T(\infty)$ and $T^{-1}(\infty)$.
\end{proof}

Lemma \ref{l:QED} allows us to prove the following proposition
with constants $C$ and $\alpha$ that do not depend on the
particular domain $\Omega$ in the statement. The class of sets $\Omega$ covered by Proposition \ref{p:dist_alpha} is strictly larger than the class of Kor\'{a}nyi balls, since the $1$-quasiconformal maps on $B$  include suitable compositions with \emph{the Kor\'{a}nyi inversion}, see Section \ref{sect-Mobius}.

\begin{proposition}\label{p:dist_alpha}
For every $K\geq 1$ and $0<m<M<\infty$, there exist constants
$C,\alpha>1$ such that whenever $\Omega \subset \mathbb{H}^1$ is a
$1$-quasiconformal image of $B$ with $B(0,m) \subset \Omega
\subset B(0,M)$ and
 $g:\Omega \to g(\Omega)\subset \mathbb{H}^1\setminus \{0\} $ is a $K$-quasiconformal map, then
\begin{equation}\label{eq:dist_est_Finer}
C^{-1}\, d(q,\partial \Omega)^{\alpha} \leq
\frac{\|g(q)\|}{\|g(0)\|} \leq C d(q,\partial
\Omega)^{-\alpha},\quad \text{for all }q\in \Omega.
\end{equation}
\end{proposition}

\begin{proof}[Proof of Proposition~\ref{p:dist_alpha}] If $\Omega$ is as assumed in the proposition, then
\begin{equation}\label{eq:dist_bound_sphere}
d(q,\partial \Omega) \leq M, \quad \text{for all }q\in \Omega.
\end{equation}
Indeed, since $q\in \Omega$ belongs to $B(0,M)$, there exists a
horizontal line segment $\ell$ of $d$-length at most $M$
connecting $q$ to $\mathbb{H}^1 \setminus B(0,M)\subseteq
\mathbb{H}^1 \setminus \Omega$. As $\ell$ is a connected set, it
must intersect $\partial \Omega$ in at least one point, which
proves that $d(q,\partial \Omega) \leq M$ as claimed.

Fix now an arbitrary point $q\in \Omega$. We first show the upper
bound in \eqref{eq:dist_est_Finer}. There are two cases to
consider. If $\|g(0)\|>\frac{1}{2} \|g(q)\|$, then we find by
\eqref{eq:dist_bound_sphere} that
\begin{displaymath}
\frac{\|g(q)\| }{\|g(0)\|}< 2   \leq  2   M^{\alpha} d(q,\partial
\Omega)^{-\alpha}\leq C d(q,\partial \Omega)^{-\alpha}
\end{displaymath}
for every $\alpha >0$ and $C\geq 2M^{\alpha}$, establishing
 the upper bound in \eqref{eq:dist_est_Finer} in that case.

\medskip If instead $\|g(0)\|\leq \frac{1}{2} \|g(q)\|$, we will use
the modulus of suitable curve families $\Gamma_n$ to prove the
corresponding estimate. To define $\Gamma_n$, let first $E'$ be
the line segment connecting $g(0)$ to $0$, and let $F'$ be the
half ray on the line through $0$ and $g(q)$ that emanates from
$g(q)$ and does not contain $0$. By convexity of Kor\'{a}nyi
balls,
\begin{equation}\label{eq:ContinuaPosition}
E'\subset \overline{B}(0,\|g(0)\|)\quad\text{and}\quad F' \subset
\mathbb{H}^1 \setminus B(0,\|g(q)\|).
\end{equation}
Moreover,  $E'$ and $F'$ are disjoint since $\|g(0)\|<\|g(q)\|$ by
assumption. Essentially, we would like to work with the family of
curves connecting $E'\cap g(\Omega)$ and $F'\cap g(\Omega)$ inside
$g(\Omega)$, but this would lead to technical problems since the
two sets are neither compact nor necessarily connected. To address
this, we consider suitable sequences of continua $E_n'\subset E'$,
$F_n'\subset F'$. By assumption, $\Omega$ is the image of $B$
under a $1$-quasiconformal map $T$, and we define
$\Omega_n:=T(B(0,1-\frac{1}{n}))$. Then $g(0)\in g(\Omega_n)$, and
for $n$ large enough, also $g(q)\in g(\Omega_n)$. On the other
hand $0\notin \overline{g(\Omega_n)}$ by our assumption on $g$,
and also the unbounded ray $F'$ must contain points outside the
compact set $\overline{g(\Omega_n)}$. It follows that $E'$ and
$F'$ must intersect $\partial g(\Omega_n)$. These considerations
imply that, for $n$ large enough depending on $q$, the sets
\begin{displaymath}
E_n'=\text{connected component of }E'\cap
\overline{g(\Omega_n)}\text{ containing }g(0),
\end{displaymath}
\begin{displaymath}
F_n'=\text{connected component of }F'\cap
\overline{g(\Omega_n)}\text{ containing }g(q)
\end{displaymath}
are disjoint continua as desired. If
\begin{displaymath}
\Gamma'_n:= \Gamma(E'_n,F'_n,g(\Omega))
\end{displaymath}
denotes the family of curves in $g(\Omega)$ that connect $E_n'$
and $F_n'$, then every element in $\Gamma_n'$ must have a subcurve
connecting $\partial B(0,\|g(0)\|)$ and $\partial B(0,\|g(q)\|)$
by \eqref{eq:ContinuaPosition}. This yields
\begin{displaymath}
\mathrm{mod}_4(\Gamma'_n)\leq \mathrm{mod}_4(\Gamma(\partial
B(0,\|g(0)\|),\partial B(0,\|g(q)\|),\mathbb{H}^1)),
\end{displaymath}
and formula \eqref{eq:ring} for the modulus of  Kor\'{a}nyi annuli
implies
\begin{equation}\label{eq:mod12}
\mathrm{mod}_4(\Gamma'_n)\leq \pi^2 \left(\ln
\frac{\|g(q)\|}{\|g(0)\|}\right)^{-3}.
\end{equation}
On the other hand, if $\Gamma_n$ denotes the family
$g^{-1}(\Gamma'_n)$, we know by \eqref{eq:mod} that
\begin{equation}\label{eq:mod22}
\mathrm{mod}_4(\Gamma_n) \leq K' \mathrm{mod}_4(\Gamma'_n),
\end{equation}
with $K'$ depending only on $K$. To find a lower bound for
$\mathrm{mod}_4(\Gamma_n)$, we observe that $\Gamma_n$ consists of
all curves in $\Omega$ that connect the two continua $E_n:=
g^{-1}(E'_n)$ and $F_n:=g^{-1}(F'_n)$. By \eqref{eq:mod22} and the
QED property stated in Lemma \ref{l:QED}, there exists a universal
constant $0<c\leq 1$, independent of the choice of $\Om$ in
Proposition~\ref{p:dist_alpha}, and independent of $E_n,F_n$, such
that
\begin{equation}\label{eq:mod_interm_est}
\mathrm{mod}_4(\Gamma'_n) \geq \frac{c}{K'}
\mathrm{mod}_4(\Gamma(E_n,F_n,\mathbb{H}^1)).
\end{equation}
The right-hand side can be bounded from below using the fact that
$(\mathbb{H},d)$ equipped with the Haar measure is a $4$-regular
$4$-Loewner space in the sense of \cite{MR1654771}, see for
instance \cite{heinonen2012lectures,MR1683160} and references
therein. This means that
\begin{equation}\label{eq:mod_last}
\mathrm{mod}_4(\Gamma(E_n,F_n,\mathbb{H}^1))\geq
\psi(\Delta(E_n,F_n)),
\end{equation}
where $\psi$ denotes a decreasing homeomorphism as in
\cite[(3.9)]{MR1654771} and
\begin{displaymath}
\Delta(E_n,F_n):=
\frac{\mathrm{dist}(E_n,F_n)}{\min\{\mathrm{diam}E_n,\mathrm{diam}F_n\}}
\end{displaymath}
is the relative distance of $E_n$ and $F_n$. Hence, by
\eqref{eq:mod12}, \eqref{eq:mod_interm_est}, and
\eqref{eq:mod_last}, we obtain for large enough $n$,
\begin{equation}\label{eq:ModPutTogether}
\pi^2 \left(\ln \frac{\|g(q)\|}{\|g(0)\|}\right)^{-3}
 \geq \frac{c}{K'} \psi(\Delta(E_n,F_n)).
\end{equation}
Since $0\in E_n\subset \Omega_n$ and $q\in F_n\subset \Omega_n$,
and the endpoints of $E_n$ and $F_n$ lie in $\partial \Omega_n$,
we find
\begin{equation}\label{eq:Loewner_prep}
\mathrm{diam} E_n \geq d(0,\partial \Omega_n),\quad \mathrm{diam}
F_n \geq d(q,\partial \Omega_n)\quad \text{and }\quad
\mathrm{dist}(E_n,F_n)\leq \|q\|\leq M.
\end{equation}
Recall that $\Omega$ is the image of $B=B(0,1)$ under a
$1$-quasiconformal map $T$, and $\Omega_n =T(B(0,1-\frac{1}{n}))$,
where $T$ is the restriction of a conformal self-map of
$\widehat{\mathbb{H}}^1$, which we continue to denote by $T$.
Since the image $T(B) \subset B(0,M)$ is bounded, $T$ is uniformly
continuous on $\overline{B}$. This can be used to prove for every
$q\in \Omega$ that
\begin{equation}\label{eq:dist_est_bdry}
 2 d(q,\partial \Omega_n) \geq
d(q,\partial \Omega)
 \end{equation}
for all $n$ large enough, depending on $q$ and $T$. Since
$B(0,m)\subset \Omega$ by assumption, we have $d(0,\partial
\Omega)\geq m$, and hence we may assume that
\begin{displaymath}
d(0,\partial \Omega_n)\overset{\eqref{eq:dist_est_bdry}}{\geq}
\frac{m}{2}\overset{\eqref{eq:dist_bound_sphere}}{\geq} \frac{m \,
d(q,\partial \Omega)}{2 M}
\end{displaymath}
for $n$ large enough. If we apply \eqref{eq:dist_est_bdry} also to
the lower bound for $\mathrm{diam}F_n$ in \eqref{eq:Loewner_prep},
then the monotone decrease of the Loewner function yields
\begin{equation}
\psi(\Delta(E_n,F_n))\geq \psi\left(\frac{2 M^2}{m\, d(q,\partial
\Omega)}\right).
\end{equation}
By \cite[Theorem 3.6]{MR1654771}, we may assume that that $\psi(t)
\sim (\ln t)^{-3}$ for large enough $t$, i.e.,
\begin{equation}\label{eq:mod3a}
\psi(\Delta(E_n,F_n))\gtrsim \left(\ln\left(\frac{2
M^2}{m\,d(q,\partial \Omega))}\right)\right)^{-3}
\end{equation}
if $d(q,\partial \Omega)$ is small enough, where the notation
``$\gtrsim$'' means that the inequality holds up to a positive
multiplicative constant that does not depend on $q$, $E_n$ or
$F_n$. Hence, if $d(q,\partial \Omega)$ is, say, such that
\begin{equation}\label{eq:DIfficult_q}
\frac{m d(q,\partial \Omega)}{2 M^2}
 \leq \frac{1}{t_0},
\end{equation}
for some $t_0$ that depends only on the geometry of
$\mathbb{H}^1$, then \eqref{eq:mod3a} holds for all large enough
$n$. On the other hand, if \eqref{eq:DIfficult_q} fails, then we
will
 simply use the estimate
\begin{equation}\label{eq:mod3b}
\psi(\Delta(E_n,F_n))\geq \psi( t_0)>0.
\end{equation}
This suffices to treat that case since we always have that
$d(q,\partial \Omega)\leq M$ by \eqref{eq:dist_bound_sphere}, so
if $d(q,\partial \Omega)$ is also bounded from below by a positive
constant in terms of $m$ and $M$ (and the absolute constant $t_0$), then actually $d(q,\partial
\Omega)$ is comparable to a constant depending on $m$ and $M$.
Hence in that case
\begin{displaymath}
\pi^2 \left(\ln \frac{\|g(q)\|}{\|g(0)\|}\right)^{-3} \geq
\frac{c}{K'} \psi(t_0)
\end{displaymath}
implies that $q$ satisfies the second inequality in
\eqref{eq:dist_est_Finer} for any $\alpha$ with large enough
constant $C$, depending only on $K$, $m$, and $M$. Hence it
remains to discuss the case of $q$ as in \eqref{eq:DIfficult_q}.
By \eqref{eq:ModPutTogether} and \eqref{eq:mod3a} applied to large
enough $n$, we obtain
\begin{equation}\label{eq:mod3}
\pi^2 \left(\ln \frac{\|g(q)\|}{\|g(0)\|}\right)^{-3} \gtrsim
\left(\ln\left(\frac{2M^2}{m\,d(q,\partial
\Omega)}\right)\right)^{-3}.
\end{equation}
It follows that
\begin{displaymath}
\frac{\|g(q)\|}{\|g(0)\|}\lesssim \left(\frac{2 M^2}{md(q,\partial
\Omega)}\right)^{C(K)^3},
\end{displaymath}
which concludes the proof of the upper bound in
\eqref{eq:dist_est_Finer}.

\medskip

In order to show the lower estimate in the
assertion~\eqref{eq:dist_est_Finer}, we follow a similar approach
as above, and hence we only sketch the proof. First assume that
$\|g(q)\|>\frac{1}{2} \|g(0)\|$. By \eqref{eq:dist_bound_sphere}
we immediately obtain that
\begin{displaymath}
\|g(0)\| < 2 \|g(q)\|  \leq 2 M \|g(q)\| d(q,\partial
\Omega)^{-1}.
\end{displaymath}
So in that case the lower estimate in \eqref{eq:dist_est_Finer}
holds for every $\alpha>0$ and with sufficiently large $C$,
depending on $M$. If $\|g(q)\|\leq \frac{1}{2} \|g(0)\|$, then we
define sets $E'$ and $F'$ as above with $g(q)$ instead of $g(0)$
in the definition of $E'$ and with the opposite change in the
definition of $F'$. Then the counterpart for \eqref{eq:mod12}
reads:
\begin{equation*}
\mathrm{mod}_4(\Gamma')\leq \mathrm{mod}_4(\Gamma(\partial
B(0,\|g(q)\|),\partial B(0,\|g(0)\|),\mathbb{H}^1))= \pi^2
\left(\ln \frac{\|g(0)\|}{\|g(q)\|}\right)^{-3},
\end{equation*}
and the rest of the proof follows as before.
\end{proof}

\begin{remark}\label{rem:growth-non-zero}
By applying Proposition~\ref{p:dist_alpha} to the case $\Om=B$
(with $m=M=1$), we get
 the growth estimates~\eqref{eq:dist_est_Finer} for quasiconformal
 maps on $B$
 omitting the origin, hence generalizing Lemma 2.2 in~\cite{MR2900163}.
\end{remark}

Having proven Proposition~\ref{p:dist_alpha} for quasiconformal
mappings that omit the origin, we next deduce information for maps
with $f(0)=0$.

\begin{proposition}\label{c:dist_alpha}
For every $K\geq 1$, there is $\alpha>0$ such that whenever
$f:B\subset \mathbb{H}^1 \to f(B)\subset \mathbb{H}^1 $ is a
$K$-quasiconformal map with $f(0)=0$, then
\begin{displaymath}
\|f(q)\| \leq C_1(f)+ C_2(f)\,d(q,\partial B)^{-\alpha},\quad
\text{for all }q\in B,
\end{displaymath}
for positive and finite constants $C_1(f)$ and $C_2(f)$, which do
not depend on $q$.
\end{proposition}
The proof below will in fact yield
\begin{displaymath}
\frac{d(f(0),f(q))}{d(f(0),\partial f(B))}\leq 1 + C d(q,\partial
B)^{-\alpha},\quad \text{for all }q\in B.
\end{displaymath}
The crucial feature is that this holds for \emph{all} points $q\in
B$, arbitrarily close to the boundary, with a constant $C$
depending only on $K$ thus controlling the rate at which
$\|f(q)\|$ can grow as $q$ approaches $\partial B$.

\begin{proof}
If $f$ satisfies the assumptions of the proposition, then $f(B)$
is a \emph{strict} subset of $\mathbb{H}^1$. Indeed, suppose
towards a contradiction that $f(B)=\mathbb{H}^1$. This yields a
quasiconformal map $f^{-1}:\mathbb{H}^1 \to B$, which is in fact
\emph{quasisymmetric}, see for instance \cite[Lemma
5.2]{MR3029176}. Since the quasisymmetric image of a complete
space is complete, this would imply that the open ball $B$ is
complete, a contradiction. Hence we can pick a point $\tau\in
\mathbb{H}^1\setminus \{0\}$ that is omitted by $f$. The constants
$C_1(f)$ and $C_2(f)$ will depend on the choice of $\tau$. Now the
map $g: B \to g(B)$, defined by $g(q):= \tau^{-1}\cdot f(q)$, is
quasiconformal with the same constant $K$ as $f$, and it fulfills
the assumptions of Proposition \ref{p:dist_alpha} for $\Omega =B$.
Then,
\begin{align*}
\|f(q)\|= d(\tau^{-1},g(q))\leq \|\tau^{-1}\| + \|g(q)\|&\leq
\|\tau^{-1}\| + C\, \|g(0)\|\, d(q,\partial B)^{-\alpha} \\&=
\|\tau\| + C\, \|\tau\|\, d(q,\partial B)^{-\alpha}
\end{align*}
for all $q\in B$, and constants $C$ and $\alpha$ depending only on
the distortion $K$ of $f$.
\end{proof}

\subsection{Proof that quasiconformal maps belong to Hardy
classes}\label{s:3_2} To prove Theorem \ref{t:QC_hardy}, we have
to consider the radial curves $\gamma(\cdot,\omega)$ from Theorem
\ref{t:polar} more closely. We will use in particular that they
are horizontal curves, that is
\begin{displaymath}
\dot{\gamma}(s,\omega)\in H_{\gamma(s,\omega)}\mathbb{H}^1=
\mathrm{span}\{X_{\gamma(s,\omega)},Y_{\gamma(s,\omega)}\},
\end{displaymath}
and that
\begin{equation}\label{eq:length_el}
\left|\frac{\partial \gamma}{\partial s}(s,(z,t))\right|=
\frac{1}{|z|}.
\end{equation}
where $|\cdot|$ on the left-hand side of \eqref{eq:length_el}
denotes the norm on $H_{\gamma(s,(z,t))}\mathbb{H}^1$ that makes
$X_{\gamma(s,(z,t))},Y_{\gamma(s,(z,t))}$ orthonormal, see the
formula below (4.4) in \cite{zbMATH04023589}, or \cite[Lemma 3.3
and Example 3.6]{MR1942237}. This allows us to prove the
following:

\begin{lemma}\label{l:deriv_r}
If  $f:B \to f(B) \subset \mathbb{H}^1$ is quasiconformal, then
for almost every $\omega =(z,t)\in \partial B\setminus \{z=0\}$
(with respect to any of the measures $\sigma_0$, $\sigma$, or
$\mathcal{S}^3$) and for almost every $s\in (0,1)$, we have
\begin{displaymath}
\left|\frac{\partial}{\partial s}
\|f(\gamma(s,\omega))\|\right|\leq
\frac{|f_I(\gamma(s,\omega))|}{\|f(\gamma(s,\omega))\|} |D_H
f(\gamma(s,\omega))| \frac{1}{|z|},
\end{displaymath}
where $|D_H f|$ is defined as below \eqref{eq:Dist_Ineq}.
\end{lemma}

\begin{proof}
Taking into account Proposition \ref{p:Mod_Radial_subet}, it
follows from well-known properties of quasiconformal mappings
\cite{MR1317384}, and of rectifiable curves in the Heisenberg
group \cite{MR3417082}, that $f\circ \gamma(\cdot,\omega)$ is a
horizontal curve for $\sigma$ almost every $\omega \in
\partial B\setminus \{z=0\}$.  On the other hand, being quasiconformal, the map $f$
is differentiable in the sense of Pansu \cite{Pansu1} at Lebesgue
almost every point in $B$. By the polar coordinates formula stated
in Theorem \ref{t:polar}, this means that for $\sigma$ almost
every $\omega$ in the Kor\'{a}nyi sphere, for almost every $s\in
(0,1)$, the point $\gamma(s,\omega)$ is a Pansu differentiability
point of $f$. We fix now $\omega\in \partial B \setminus \{z=0\}$
such that $\lambda:=f\circ \gamma(\cdot,\omega)$ is horizontal,
and we further fix $s\in (0,1)$ such that the tangent vectors
$\partial_s \gamma(\omega,s)$ and $\dot{\lambda}(s)$ exist and are
 horizontal, and such that $f$ is Pansu differentiable at
$\gamma(s,\omega)$ with Pansu differential given by its
\emph{formal Pansu differential} as in \cite[Theorem
5.1]{MR1778673}. The horizontality of $\lambda$ means that
\begin{equation}\label{eq:lambda_horizontal}
\dot \lambda_3 = 2 (\dot \lambda_1 \lambda_2 - \dot \lambda_2
\lambda_1),\quad \text{almost everywhere on }(0,1).
\end{equation}
Then, almost everywhere,
\begin{align*}
\frac{\partial}{\partial s} \sqrt[4]{(\lambda_1^2
+\lambda_2^2)^2+\lambda_3^2}&= \frac{1}{4}
\frac{4(\lambda_1^2+\lambda_2^2)(\lambda_1 \dot
\lambda_1+\lambda_2 \dot \lambda_2)+2\dot\lambda_3 \lambda_3}
{\|\lambda\|^3}\\
&\overset{\eqref{eq:lambda_horizontal}}{=
}\frac{(\lambda_1^2+\lambda_2^2)(\lambda_1 \dot\lambda_1+\lambda_2
\dot \lambda_2)+ \dot \lambda_1 \lambda_2 \lambda_3 -
\dot\lambda_2 \lambda_1\lambda_3}
{\|\lambda\|^3}\\
&= \frac{\dot\lambda_1((\lambda_1^2+\lambda_2^2)\lambda_1 +
\lambda_2\lambda_3)+\dot\lambda_2((\lambda_1^2+\lambda_2^2)\lambda_2-\lambda_1\lambda_3)}
{\|\lambda\|^3}.
\end{align*}
Hence
\begin{equation}\label{eq:tangent_bound}
\left|\frac{\partial}{\partial s} \|\lambda\| \right| \leq
\frac{|\dot \lambda_I|}{\|\lambda\|^3} |\lambda_I| \|\lambda\|^2 =
|\dot \lambda_I| \frac{|\lambda_I|}{\|\lambda\|},
\end{equation}
where we have denoted $\lambda_I = \lambda_1 + \mathrm{i}\lambda_2$. The lemma follows upon observing that
\begin{displaymath}
\dot \lambda(s) = D_H f(\gamma(s,\omega))
\partial_s\gamma(s,\omega) \in H_{\lambda(s)} \mathbb{H}^1,
\end{displaymath}
which is based on the fact that the restriction of the Pansu
derivative of $f$ to $H_{\gamma(s,\omega)}\mathbb{H}^1$ coincides
with the formal horizontal derivative $D_H f(\gamma(s,\omega))$
defined below \eqref{eq:Dist_Ineq}, and a chain rule holds for
Pansu derivatives, see \cite{MR2115223}. Hence we obtain
\begin{displaymath}
|\dot \lambda_I (s)| \leq |D_H f(\gamma(s,\omega))|
\left|\frac{\partial}{\partial s}\gamma(s,\omega)\right|,
\end{displaymath}
which yields the claim in Lemma \ref{l:deriv_r} by
\eqref{eq:tangent_bound} and \eqref{eq:length_el}.
\end{proof}

 We now prove the main result of this
section. With Proposition \ref{c:dist_alpha} and Lemma
\ref{l:deriv_r} in place, our argument follows the proof by
Jerison and Weitsman of \cite[Theorem 1]{MR624919}.
\begin{proof}[Proof of Theorem \ref{t:QC_hardy}]
Using left translations and Heisenberg dilations, we may assume
without loss of generality that $f(0)=0$ and that there exists
$\varepsilon>0$ so that
\begin{equation}\label{eq:epsilon_ass}
 \| f|_{B\setminus B(0,1/2)}\|>
\varepsilon. \end{equation} Points $\omega \in
\partial B$ will be denoted by $\omega=(z,t)$. Since $f$ is quasiconformal, it is absolutely continuous along
$\mathrm{mod}_4$ almost every curve. By the modulus formula
 \eqref{eq:mod_formula}, this implies that $f$ is absolutely continuous along the
radial curve segment $\gamma(\cdot,\omega):[1/2,1] \to
\mathbb{H}^1$ for $\mathcal{S}^3$ almost every $\omega$, recalling
that $\mathcal{S}^3|_{\partial B \setminus \{z=0\}}$ is absolutely
continuous with respect to $\sigma$.

Let us now fix an exponent $p>0$, to be determined later, and
$\frac{1}{2}<r<1$. We apply Lemma \ref{l:deriv_r} to obtain
\begin{align*}
&\int_{\partial B} \|f(\gamma(r,\omega))\|^p\,d\mathcal{S}^3(\omega)-
\int_{\partial B} \|f(\gamma(\tfrac{1}{2},\omega))\|^p\,d\mathcal{S}^3(\omega)\\
&= \int_{\partial B} \int_{\frac{1}{2}}^r \frac{\partial}{\partial s} \|f(\gamma(s,\omega))\|^p \, ds\,d\mathcal{S}^3(\omega)\\
&\leq \int_{\partial B} \int_{\frac{1}{2}}^r \left|\frac{\partial}{\partial s} \|f(\gamma(s,\omega))\|^p\right| \, ds\,d\mathcal{S}^3(\omega) \\
&\leq p \int_{1/2}^r \int_{\partial B}
\|f(\gamma(s,\omega))\|^{p-2}|f_I(\gamma(s,\omega)| \,|D_H
f(\gamma(s,\omega))| \frac{1}{|z|}d\mathcal{S}^3(\omega)\,ds.
\end{align*}
The second integral on the left-hand side is a finite positive
number $C=C(f,p)$ that does not depend on $r$. Thus, by the
distortion inequality \eqref{eq:Dist_Ineq} for quasiconformal
maps,
\begin{align*}
&\int_{\partial B} \|f(\gamma(r,\omega)\|^p d\mathcal{S}^3(\omega)\\
& \leq (K'')^{1/4}p \int_{1/2}^r\int_{\partial
B}\|f(\gamma(s,\omega))\|^{p-2}|f_I(\gamma(s,\omega)|\,
J_f(\gamma(s,\omega))^{1/4} \frac{1}{|z|}d\mathcal{S}^3(\omega)\,ds +C\\
& = (K'')^{1/4}p \int_{1/2}^r\int_{\partial B} g(s,\omega)\cdot
h(s,\omega)d\mathcal{S}^3(\omega)\,ds +C=: I(r) + C,
\end{align*}
where, for $\omega=(z,t)$, we have
\begin{align*}
g(s,\omega)&:=\|f(\gamma(s,\omega))\|^{-(p+1)}
J_f(\gamma(s,\omega))^{1/4} \frac{1}{|z|^{1/4}}s^{3/4}\\
h(s,\omega)&:=
\|f(\gamma(s,\omega))\|^{2p-1}|f_I(\gamma(s,\omega)|\frac{1}{|z|^{3/4}}
s^{-3/4}.
\end{align*}
We estimate $I$ by applying H\"older's inequality with exponents
$4$ and $4/3$. This yields $ I(r) \leq (K'')^{1/4} p I_1(r)^{1/4}
I_2(r)^{3/4}, $ where by the area formula for quasiconformal maps
\cite[Theorem 5.4]{MR1778673}, and  \eqref{eq:epsilon_ass},
\begin{align*}
I_1(r)&:= \int_{1/2}^r \int_{\partial B} \|f(\gamma(s,\omega))\|^{-4(p+1)} J_f(\gamma(s,\omega)) s^3 \frac{1}{|z|}\,d\mathcal{S}^3(\omega)\, ds\\
&= \int_{1/2}^r \int_{\partial B} \|f(\gamma(s,\omega))\|^{-4(p+1)} J_f(\gamma(s,\omega)) s^3\, d\sigma_0(\omega) \,ds\\
&= \int_{B(0,r)\setminus B(0,1/2)} \|f(q)\|^{-4(p+1)} J_f(q)\, dq\\
&\leq \int_{\mathbb{H}^1\setminus B(0,\varepsilon)}
\|q\|^{-4(p+1)}\, dq<\infty,
\end{align*}
 and
\begin{align*}
I_2(r)&:= \int_{1/2}^r \int_{\partial B} \|f(\gamma(s,\omega))\|^{(2p-1)4/3}|f_I(\gamma(s,\omega))|^{4/3} s^{-1}\, d\sigma_0(\omega) \,ds\\
&\leq \int_{1/2}^r \int_{\partial B} \|f(\gamma(s,\omega))\|^{2p 4/3} s^{-1} \,d\sigma_0(\omega)\, ds\\
&\leq C \int_{1/2}^r (1-s)^{-2p\alpha4/3}\,ds+C.
\end{align*}
Here we used that $d\mathcal{S}^3 = |z| d\sigma_0$ by Lemma
\ref{l:Haus} for $\sigma_0$ as in the polar coordinates formula of
Theorem \ref{t:polar}. The estimate in the last line is a
consequence of Proposition \ref{c:dist_alpha}, which yields that
\begin{displaymath}
\|f(q)\| \leq C_1(f) + C_2(f) d(q,\partial B)^{-\alpha}\leq C_1(f)
+ C_2(f) \left(1-\|q\|\right)^{-\alpha} ,\quad\text{for all }q\in
B,
\end{displaymath}
for some positive and finite constants $C_1(f)$ and $C_2(f)$, and
$\alpha=\alpha(K)>0$ that only depends on the distortion $K$ of $f$.  If $2p\alpha4/3<1$, then $\sup_{1/2<r<1}
I_2(r)$
 is clearly finite. This proves the theorem with
$
p_0(K):= \tfrac{3}{8 \alpha(K)}.
$
\end{proof}

\section{Carleson measures}\label{s:Carleson}

The aim of this section is to prove Proposition \ref{l:5.6}, which
states that every quasiconformal map $f:B \to f(B) \subset
\mathbb{H}^1\setminus \{0\}$ gives rise to certain Carleson
measures on $B$. This will be an important tool in the proof of
Theorem \ref{t:main2}. A straightforward corollary of Proposition
\ref{l:5.6} says that $|\nabla_H \log \|f(q)\||\,dq$ defines a
Carleson measure for every quasiconformal map $f$ as above. This
extends to $\mathbb{H}^1$ a result that was proven in
$\mathbb{R}^n$ first by Jones \cite[Lemma 4.2 and p.65]{MR554817}
in order to show that $\log |f|$ belongs to $BMO(S^{n-1})$,
quantitatively, for each $K$-quasiconformal map $f$ on the unit
ball in $\mathbb{R}^n$ that omits $0$, where $f$ on $S^{n-1}$ is
understood as the radial limit. Jones' result about Carleson
measures was also obtained, with a different method, by Astala and
Koskela \cite[Lemma 5.6]{MR2900163}, and we follow roughly their
approach. However, the class of M\"obius self-maps of $B$ is not
rich enough to perform the standard normalization arguments done
in \cite{MR554817,MR2900163}. The generality in which we stated
Proposition \ref{p:dist_alpha} allows us to work with
$1$-quasiconformal maps that do not necessarily preserve $B$.

\subsection{The canonical M\"obius transformation}\label{sect-Mobius}
M\"obius transformations are a commonly used tool to simplify
proofs concerning Hardy spaces and Carleson measures on the unit
ball $B^n$ in  $\mathbb{R}^n$. M\"obius self-maps of $B^n$ are
discussed in great detail in Ahlfors' monograph \cite[p.24
ff.]{MR725161}. The computations there use specific properties of
the Euclidean metric and M\"obius transformations in
$\mathbb{R}^n$. The relevant maps in $\mathbb{H}^1$ arise as
restrictions of conformal maps of the compactified Heisenberg
group, they are compositions of \emph{left translations},
\emph{Heisenberg dilations}, \emph{rotations about the vertical
$t$-axis}, and the \emph{Kor\'{a}nyi inversion}. While these share
many properties with M\"obius transformations in $\mathbb{R}^n$,
see e.g. \cite{CTU, MR3276133}, there are important differences.
To give some examples, the conformal group in our case is not
transitive on the set of triples of distinct points
\cite{MR925991}, rotations are not transitive on $\partial B$, and
unlike $x\mapsto \frac{x}{|x|^2}$, the Kor\'anyi inversion does
not keep $\partial B$ pointwise fixed. Nonetheless,
$1$-quasiconformal
 maps of, but not necessarily onto, $B$  play an important role in our proof of Theorem
 \ref{t:main2}. Here we discuss the preliminaries.

The \emph{Kor\'{a}nyi inversion} in the Kor\'{a}nyi unit sphere
centered at the origin is defined as follows:
$I(y)=-\frac{1}{\|y\|^4}\left(y_z(|y_z|^2+iy_t), y_t\right)$,
where $y=(y_z, y_t) \in \mathbb{H}^1 \setminus \{0\}$. It is the
restriction of a conformal self-map of the compactification
$\widehat{\mathbb{H}}^1$, with $I(0)=\infty$ and $I(\infty)=0$,
see \cite{MR1317384}. This inversion
 was introduced by Kor\'{a}nyi \cite[(1.8)]{MR680658} to define
a Kelvin transform for functions on the Heisenberg group. The
inversion has the crucial property that
\begin{equation}\label{eq:InversionFormula}
d(I(y),I(y')) = \frac{d(y,y')}{\|y\| \, \|y'\|},\quad y,y'\in
\mathbb{H}^1 \setminus \{0\},
\end{equation}
see, e.g., \cite[p.19]{capogna2007introduction}. The Jacobian of
$I$ at a point $y\in B$, $y\not=0$, can be expressed as follows:
$J_I(y)=(Xf_1 Yf_2-Xf_2 Yf_1)^2(y)= \frac{1}{\|y\|^8}$, see
\cite[(3.5)]{MR1102963}. We will use the inversion $I$ to define
certain \emph{canonical} $1$-quasiconformal mappings. We start
with the most general definition, and later add more requirements
on the parameters, as we prove finer properties.

\begin{proposition}\label{p:GeneralMob}
For $x\in \mathbb{H}^1$, $a\in \mathbb{H}^1\setminus \{x\}$, and
$\rho>0$, the map
\begin{displaymath}
T:=T_{x,a,\rho}:\hat{\mathbb{H}}^1 \to \hat{\mathbb{H}}^1,\quad
T(y):= \delta_{\rho}\left(\left[I(a^{-1}\cdot x)\right]^{-1} \cdot
\left[I(a^{-1}\cdot y)\right] \right)
\end{displaymath}
has the following properties:
\begin{equation}\label{eq1}
T(x)=0,\quad T(a)=\infty,\quad T(\infty) =
\delta_{\rho}\left([I(a^{-1}\cdot x)]^{-1}\right),
\end{equation}
\begin{equation}\label{eq2}
T|_{\mathbb{H}^1 \setminus \{a\}}: \mathbb{H}^1 \setminus \{a\}
\to \mathbb{H}^1 \setminus \{\delta_{\rho}\left([I(a^{-1}\cdot
x)]^{-1}\right)\}\quad \text{is $1$-quasiconformal},
\end{equation}
for all $y,y'\in \mathbb{H}^1 \setminus \{a\}$, it holds that
\begin{equation}\label{eq3}
d(T(y),T(y')) = \rho \frac{d(y,y')}{d(a,y) d(a,y')},
\end{equation}
\begin{equation}\label{eq4}
\|T(y)\| = \rho \frac{d(x,y)}{d(a,y) d(a,x)},
\end{equation}
\begin{equation}\label{eq5}
J_T(y) = \frac{\rho^4}{d(a,y)^8},
\end{equation}
and for all $r>0$, one has
\begin{equation}\label{eq6}
T(\partial B(a,r)) = \partial
B\left(\delta_{\rho}\left([I(a^{-1}\cdot
x)]^{-1}\right),\frac{\rho}{r}\right).
\end{equation}
\end{proposition}

\begin{proof}
Property \eqref{eq1} is immediate from the definition of $T$,
recalling that $I(0)=\infty$ and $I(\infty)=0$. Then $T$ maps
$\mathbb{H}^1 \setminus \{a\}$ homeomorphically onto $\mathbb{H}^1
\setminus \{\delta_{\rho}\left([I(a^{-1}\cdot x)]^{-1}\right)\}$
and it is $1$-quasiconformal on $\mathbb{H}^1 \setminus \{a\}$ as
a composition of left-translations, dilations, and the
$1$-quasiconformal inversion $I$. This proves \eqref{eq2}. Since
the metric $d$ is invariant under left translations, scales by
factor $\rho$ under the dilation $\delta_{\rho}$, and interacts
with the inversion $I$ as stated in \eqref{eq:InversionFormula},
we can deduce \eqref{eq3} immediately from the definition of $T$.
Next, \eqref{eq4} follows from \eqref{eq3} applied to $y'=x$
(recalling that $T(x)=0$ as stated in \eqref{eq1}). The Jacobian
of $T$ can be computed by the chain rule. Denoting left
translations by $L_{q_0}(q):= q_0 \cdot q$, we find
\begin{align*}
J_T(y) &= J_{\delta_{\rho}}\left([I(a^{-1}\cdot x)]^{-1}\cdot
[I(a^{-1}\cdot y)]\right) \, JL_{[I(a^{-1}\cdot
x)]^{-1}}(I(a^{-1}\cdot
y))\, JI(a^{-1}\cdot y)\, JL_{a^{-1}}(y)\\
&= \rho^4 \cdot 1 \cdot \frac{1}{d(a,y)^8}\cdot 1.
\end{align*}
This yields \eqref{eq5}. Finally, \eqref{eq6} follows since
$I(\partial B(0,r)) = \partial B(0,1/r)$ for all $r>0$.
\end{proof}

\begin{cor}\label{c:TB_in_ball} There exists a constant $C\geq 1$ such that for all
$x\in B$, $a\in \mathbb{H}^1 \setminus \overline{B}$ and $\rho>0$,
the map $T=T_{x,a,\rho}$ from Proposition \ref{p:GeneralMob}
satisfies
\begin{displaymath}
T(B) \subset B\left(0, \frac{C \rho}{d(a,\partial B)}+
\frac{\rho}{d(a,x)}\right).
\end{displaymath}
In particular, if
\begin{equation}\label{eq:ImplicitConstants}
\rho \lesssim d(a,\partial B) \quad\text{and} \quad \rho \lesssim
d(a,x),
\end{equation}
then $T(B)\subset B(0,M)$ for a radius  $M$ that depends on $x$,
$a$, and $\rho$ only through the implicit multiplicative constants
in the inequalities in \eqref{eq:ImplicitConstants}.
\end{cor}

\begin{proof}
Formula \eqref{eq4} in Proposition \ref{p:GeneralMob} yields by
the triangle inequality that
\begin{displaymath}
\|T(y)\| = \rho \frac{d(x,y)}{d(a,y) d(a,x)}\leq
\frac{\rho}{d(a,y)} + \frac{\rho}{d(a,x)},\quad y\in \mathbb{H}^1
\setminus \{0\}.
\end{displaymath}
Now if $y\in B$, then $C d(a,y) \geq d(a,\partial B)$ for a
universal constant $C$. Indeed, the inequality holds true with
constant $1$ if $d$ is replaced by the sub-Riemannian distance,
and then the constant $C$ can be found by comparing the two
distances.
\end{proof}

Corollary \ref{c:TB_in_ball} concerned the size of balls $B(0,M)$
that \emph{contain} $T_{x,a,\rho}(B)$. Similarly, we next study
the size of a ball $B(0,m)$ that can be \emph{included} in
$T_{x,a,\rho}(B)$.

\begin{cor}\label{c:ball_in_TB}
Assume that $x\in B$, $a\in \mathbb{H}^1 \setminus \overline{B}$
and $\rho>0$ are such that
\begin{equation}\label{eq:Cond_x_a_rho}
d(x,\partial B) \gtrsim \rho ,\quad d(a,x)\lesssim \rho.
\end{equation}
Then $ B(0,m) \subset T_{x,a,\rho}(B)$ for a constant $m>0$ that
depends on $x$, $a$, and $\rho$ only through the implicit
multiplicative constants in the inequalities in
\eqref{eq:Cond_x_a_rho}.
\end{cor}

\begin{proof}
By Proposition \ref{p:GeneralMob}, the map  $T=T_{x,a,\rho}$
satisfies for  all $w'\in \partial B$ that
\begin{align*}
\|T(w')\|=d(T(x),T(w'))= \frac{\rho d(x,w')}{d(a,x)\,
d(a,w')}\overset{\eqref{eq:Cond_x_a_rho}}{\gtrsim}
\frac{d(x,w')}{d(a,w')}&\gtrsim \frac{d(x,w')}{d(a,x)+d(x,w')}\\&
\overset{\eqref{eq:Cond_x_a_rho}}{\gtrsim}
\frac{d(x,w')}{\rho+d(x,w')}\\&
\overset{\eqref{eq:Cond_x_a_rho}}{\gtrsim}
\frac{d(x,w')}{d(x,w')+d(x,w')} \gtrsim 1.
\end{align*}
\end{proof}

We now discuss the behavior of $T_{x,a,\rho}$ on $B(\omega,r)\cap
B$ for $\omega \in \partial B$ and  $r>0$, under certain
conditions on these parameters.

\begin{proposition}\label{p:T_w_r}
 Let $\omega \in \partial B$, $x\in B$, and $\rho>0$.
Assume that $a\in \mathbb{H}^1 \setminus \overline{B}$
and $r>0$  are such that
\begin{equation}\label{eq:Cond_x_a_rho_C}
d(a,\omega) \lesssim r, \quad d(a,B)\geq C r
\end{equation}
for a constant $C>1$. Then, the map $T= T_{x,a,\rho}$ from
Proposition \ref{p:GeneralMob} satisfies
\begin{equation}\label{eq:ComparabilityJ}
\frac{d(T(y),\partial T (B))}{d(y,\partial B)} \sim_C
\frac{\rho}{d(y,a)^2},\quad \text{for all }y\in B(\omega,r)\cap B.
\end{equation}
\end{proposition}

\begin{proof}
By \eqref{eq3} in Proposition \ref{p:GeneralMob}, we know that
\begin{equation}\label{eq:GenFormq_q'}
\frac{\rho}{d(y,a)^2} = \frac{d(T(y),T(y')) \, d(a,y')}{d(a,y)\,
d(y,y')},\quad y,y'\in \mathbb{H}^1 \setminus \{a\}.
\end{equation}
First, we apply \eqref{eq:GenFormq_q'} for $y\in B$ and $y'\in
\partial B$ with the property that
\begin{displaymath}
d(T(y),T(y')) = d\left(T(y),T(\partial B)\right).
\end{displaymath}
For this pair of points, \eqref{eq:GenFormq_q'} yields
\begin{align*}
\frac{\rho}{d(y,a)^2} = \frac{d(T(y),T(\partial B)) \,
d(a,y')}{d(a,y)\, d(y,y')} &\leq d(T(y),T(\partial B))
\left[\frac{1}{d(y,y')} + \frac{1}{d(a,y)} \right]\\
&\lesssim \frac{d(T(y),\partial T B)}{d(y,\partial B)}.
\end{align*}
Here the last inequality holds since $y\in B$, $a\in \mathbb{H}^1
\setminus \overline{B}$, and $y'\in \partial B$, hence
$d(y,\partial B)\leq d(y,y')$ and, as in the proof of Corollary
\ref{c:TB_in_ball},  $d(y,\partial B)\lesssim d(y,a)$. Thus,
\begin{displaymath}
\frac{\rho}{d(y,a)^2} \lesssim \frac{d(T(y),\partial T
B)}{d(y,\partial B)},\quad y\in B.
\end{displaymath}
\medskip Second, to prove the converse inequality in
\eqref{eq:ComparabilityJ}, we apply \eqref{eq:GenFormq_q'} to
$y\in B(\omega,r)\cap B$ and $y'\in \partial B$  with the property
that
\begin{displaymath}
d(y,y') = d\left(y,\partial B\right).
\end{displaymath}
For this pair of points, \eqref{eq:GenFormq_q'} yields
\begin{align*}
\frac{\rho}{d(y,a)^2} = \frac{d(T(y),T(y')) \, d(a,y')}{d(a,y)\,
d(y,\partial B)} &\geq  \frac{d(T(y),T(\partial T B)) \,
d(a,y')}{d(a,y)\, d(y,\partial B)} .
\end{align*}
Thus it suffices to show that $d(a,y') \gtrsim_C d(a,y)$. To this
end, we use the assumptions stated in \eqref{eq:Cond_x_a_rho_C},
which yield, since $C>1$ and $d(y,\partial B)\leq d(y,\omega)<r$,
that
\begin{displaymath}
d(a,y') \geq d(a,y)-d(y,y') = d(a,y)- d(y,\partial
B)\overset{\eqref{eq:Cond_x_a_rho_C}}{\geq} (C-1) r,
\end{displaymath}
and
\begin{displaymath}
d(a,y) \leq d(a,\omega) +
d(\omega,y)\overset{\eqref{eq:Cond_x_a_rho_C}}{\lesssim} r.
\end{displaymath}
Combining these estimates, we deduce that  $d(a,y') \gtrsim_C
d(a,y)$ and conclude the proof.
\end{proof}

\begin{lemma}\label{l:f=gT}  Assume that
$\omega \in \partial B$, $x\in B$, $a\in \mathbb{H}^1 \setminus
\overline{B}$, $\rho>0$ and $r>0$ satisfy the conditions in
Proposition \ref{p:T_w_r},  and additionally, $\rho \sim r$. Let
$f:B \to f(B) \subset \mathbb{H}^1\setminus \{0\}$ be a
$K$-quasiconformal map, and let $T=T_{x,a,\rho}$ be the
$1$-quasiconformal map defined in Proposition~\ref{p:GeneralMob}.
 Then
\begin{displaymath}
g|_{T(B(\omega,r) \cap B)}:= f \circ T^{-1}|_{T(B(\omega,r) \cap
B)}
\end{displaymath}
satisfies
\begin{align}\label{eq:Df_Dg}
\frac{|D_H f (y)|^p }{\|f(y)\|^{p}}\,d(y,\partial B)^{p-1}
\sim_{p,K} r^3 \frac{|D_H g
(T(y))|^p}{\|g(T(y))\|^{p}}\,d(T(y),\partial T( B))^{p-1} \,
J_T(y)
\end{align}
for $0<p<4$ and almost every $y\in B(\omega,r) \cap B(0,1)$.
\end{lemma}

\begin{proof}
The maps $f$ and $g$ are  $K$-quasiconformal on $B(\omega,r)\cap
B$ and $T(B(\omega,r)\cap B)$, respectively. Hence their Pansu
derivatives exist and agree with the formal Pansu derivatives
almost everywhere. Since $T$ is conformal and $T$-preimages of
null sets are null sets, the chain rule \cite[Proposition
3.2.5]{MR2115223} and elementary linear algebra imply that the
operator norms satisfy
\begin{equation}\label{eq:f=gT}
|D_H f|\sim_K |D_Hg(T(\cdot))|\, J_T^{1/4},\quad \text{almost
everywhere}.\end{equation} By Proposition \ref{p:GeneralMob},  the
assumptions $\rho \sim r \sim d(a,y)$ for $y\in B(\omega,r)\cap
B$, and by Proposition \ref{p:T_w_r}, we have
\begin{equation}
\label{eq:J_T_equiv} J_T(y)^{1/4} \overset{\eqref{eq5}}{\sim}
\frac{\rho}{d(y,a)^2} \overset{\eqref{eq:ComparabilityJ}} {\sim}
\frac{d(T(y),\partial T(B))}{d(y,\partial B)} \sim
\frac{1}{r},\quad\text{for all } y\in B(\omega,r) \cap B.
\end{equation}
Therefore we can proceed as follows:
\begin{align*}
|D_H f|^p \overset{\eqref{eq:f=gT}}{\sim}_{p,K}
|D_Hg(T(\cdot))|^p\, J_T^{p/4}& \sim  |D_Hg(T(\cdot))|^p\,
J_T^{(p-1)/4} J_T^{-3/4} J_T
\\&\overset{\eqref{eq:J_T_equiv}}{\sim} |D_Hg(T(\cdot))|^p\, \left(\frac{d(T(\cdot),\partial
T(B))}{d(\cdot,\partial B)}\right)^{p-1}\, r^3 \,J_T
\end{align*}
almost everywhere on $B(\omega,r) \cap B$. This yields
\eqref{eq:Df_Dg}.
\end{proof}

\begin{remark}\label{r:Choice_x_a_r} Given $\omega \in \partial B$
and $0<r\lesssim 1$, it is possible to choose $x\in B$, $a\in
\mathbb{H}^1 \setminus \overline{B}$, and $\rho>0$ such that all
the assumptions of the results stated in this section are
satisfied. To see this, we use the fact that $B$ satisfies
interior and exterior corkscrew conditions \cite{MR1658616} (or
apply Proposition \ref{p:curve_cone} to  find interior corkscrew
points). This means that there exist constants $M_0 \geq 1$ and
$r_0>0$ such that the following holds:
\begin{itemize}
\item For every $\omega \in \partial B$ and $0<r<r_0$, there
exists $A_{i,r}(\omega) \in B$ such that
\begin{displaymath}
\frac{r}{M_0} \leq d(A_{i,r}(\omega),\partial B) \leq
d(A_{i,r}(\omega),\omega)\leq r.
\end{displaymath}
\item For every $\omega \in \partial B$ and $0<r<r_0$, there
exists $A_{o,r}(\omega) \in \mathbb{H}^1 \setminus \overline{B}$
such that
\begin{displaymath}
\frac{r}{M_0} \leq d(A_{o,r}(\omega),\partial B) \leq
d(A_{o,r}(\omega),\omega)\leq r.
\end{displaymath}
\end{itemize}
We consider $r<r_0/(M_0 N)=:r_{\ast}$ for a fixed constant $N>1$
to be determined. We claim that $N$ can be chosen such that
\begin{displaymath}
\rho:= r,\quad x:=A_{i,r}(\omega)\in B,\quad a:= A_{o,M_0 N
r}(\omega)\in \mathbb{H}^1 \setminus \overline{B}
\end{displaymath}
have the desired properties. Indeed, we have
\begin{enumerate}
\item $d(x,\partial B)\geq r/M_0$, \item $r/M_0\leq
d(x,\omega)\leq r$, \item $Nr \leq d(a,\partial B) \leq
d(a,\omega)\leq M_0 N r$, \item $(N-1)r\leq
d(a,\omega)-d(\omega,x)\leq d(x,a) \leq
d(x,\omega)+d(\omega,a)\leq (1 +M_0N) r$.
\end{enumerate}
Finally, arguing as in the proof of Corollary \ref{c:TB_in_ball},
we find a universal constant $\eta<1$ such that
\begin{displaymath}
d(a,q) \geq  \eta d(\partial B,a) \geq \eta N r =: Cr,\quad q\in
B.
\end{displaymath}
Thus we choose $N>1$ large enough, depending on $\eta$, such that
$C:= \eta N>1$.
\end{remark}

\begin{remark}
 The maps $T_{x,a,\rho}$ are in general
 \emph{not} self-maps of $B$.
 For a simple example, consider $x=(x_1,0,0)$ and $a=(a_1,0,0)$ with
 $0<x_1<1<a_1$. Then $T=T_{x,a,\rho}$ does not map $B$ onto itself
 for any choice of $\rho>0$. This can be seen by computing
 $\|T(y)\|$ for $y\in \{(1,0,0),(-1,0,0),(0,2^{-\frac14},2^{-\frac12})\}$.
 This is in
contrast to the situation in $\R^n$, where the relevant maps
preserve the unit ball thanks to the
identity~\cite[(33)]{MR725161}.
 \end{remark}

\subsection{Carleson measures related to quasiconformal maps}
Carleson characterized in  \cite[Theorem 1]{MR141789} the measures
$\mu$ on the unit disk $\mathbb{D}$ in $\mathbb{C}$ for which
$\|f\|_{L^p(\mu)}\leq C(\mu) \|f\|_{H^p}$ holds for all
holomorphic functions on $\mathbb{D}$ that belong to the Hardy
space $H^p$, $p\geq 1$. These are measures $\mu$ with the property
that $\mu(\mathbb{D}\cap B(\omega,r)) \lesssim r$ for $\omega \in
\partial \mathbb{D}$ and $r>0$. Such measures, and various
generalizations thereof, are now known as \emph{Carleson
measures}. Carleson's result was reproved by H\"ormander~\cite{MR234002} and, in extended form, by
Duren \cite{MR241650}. Later Nolder \cite{MR1307591}, and Astala and Koskela \cite[4.5.
Corollary]{MR2900163}, generalized Carleson's and Duren's result
to Hardy spaces of quasiconformal mappings on the unit ball in
$\mathbb{R}^n$. We will prove an analogous result for $\Hei$ in Section~\ref{s:CarlesonChar}.
This motivates the following definition (cf. Definition~\ref{def: Carleson-msp} for metric spaces).

\begin{definition}\label{def: Carleson-m}
Let $1\leq \alpha<\infty$.  We say that a (positive) Borel measure $\mu$ on $B$ is an
\emph{$\alpha$-Carleson measure on  the Kor\'anyi unit ball $B$}, if there exists a
constant $C>0$ such that
\begin{equation}\label{eq:CarlesonCOnst}
\mu(B \cap B(\omega,r))\leq C r^{3\alpha},\quad \text{for all
}\omega\in \partial B\text{ and }r>0.
\end{equation}
The $\alpha$-\emph{Carleson measure constant} of $\mu$ is defined
by
\begin{displaymath}
\gamma_{\alpha}(\mu):= \inf\{C>0\text{ such that
\eqref{eq:CarlesonCOnst} holds for all $\omega \in\partial B$ and
$r>0$}\}
\end{displaymath}
We also call $1$-Carleson measures simply \emph{Carleson measures}.
\end{definition}

The following is an extension to $\mathbb{H}^1$ of a result
originally due to Jones \cite{MR554817}. Our argument is inspired
by a proof of Jones' result that was given later by Astala and
Koskela in \cite[Lemma 5.6]{MR2900163}, but we avoid the use of
M\"obius self-maps of $B$.

\begin{proposition}\label{l:5.6}
 Fix $0<p<4$ and let $f$ be a quasiconformal map on $B\subset \Hei$ with $f(q)\not=0$ for all $q\in B$.
 Then the following measure $\mu$ is a Carleson measure in $B$,
 with an upper bound for the Carleson measure constant that depends only on $K$ and $p$:
 \begin{equation}\label{eq:form_Carles_p}
  d\mu=\frac{|D_Hf(q)|^p}{\|f(q)\|^p}d(q,\partial B)^{p-1} dq.
 \end{equation}
\end{proposition}
\begin{proof}
Let  $f$ be an arbitrary map satisfying the assumptions of the
proposition, and fix $0<p<4$. We aim to prove that the associated
measure $\mu$ in \eqref{eq:form_Carles_p} is a Carleson measure.
In the course of the proof below, we will show that $\mu(B(0,1))$
can be bounded by a finite constant depending only on $p$ and $K$,
see \eqref{eq:CarlesonLargeScale}. Taking this for granted, it
suffices to verify the Carleson condition for small $r>0$, say
$r<r_{\ast}$, where $r_{\ast}$ is the absolute constant in Remark
\ref{r:Choice_x_a_r} that only depends on the geometry of
$\mathbb{H}^1$. Indeed, for $r>r_{\ast}$ and $\omega \in
\partial B$, we trivially have
\begin{displaymath}
\mu\left(B\cap B(\omega,r)\right)\leq \mu(B) \lesssim_{K,p}
(1/r_{\ast})^3 r^3.
\end{displaymath}

\medskip

Thus let us fix $\omega \in \partial B$ and $0<r<r_{\ast}$. Then
choose $x\in B$, $a\in \mathbb{H}^1 \setminus \overline{B}$ and
$\rho>0$, depending on $\omega$ and $r$, as in Remark
\ref{r:Choice_x_a_r}, and consider the associated
$1$-quasiconformal map $T=T_{x,a,\rho}$ defined in Proposition
\ref{p:GeneralMob}. Writing $g:= f \circ T^{-1}|_{TB}$, Lemma
\ref{l:f=gT} yields:
\begin{align}\label{l.5.6-ineq2}
 &\int_{B\cap B(\om, r)}\frac{|D_Hf(q)|^p}{\|f(q)\|^p}\,d(q,\partial B)^{p-1} dq \nonumber \\
 &\phantom{aa}\sim_K r^{3} \int_{B\cap B(\om, r)}\frac{|D_H g(T(q))|^p}{\|g(T(q))\|^p}\,d(T(q),\partial T (B))^{p-1} J_{T}(q) dq.
  \end{align}
By the change of variables $q\mapsto T(q)$ and H\"older's
inequality with exponents $4/p$ and $4/(4-p)$, the last line
in~\eqref{l.5.6-ineq2} becomes:
\begin{align}\label{l.5.6.est3}
& r^{3} \int_{T(B\cap B(\om, r))}\frac{|D_H
g(q))|^p}{\|g(q)\|^p}\,d(q,\partial T(B))^{p-1}
dq\\
&\phantom{aa}\leq r^3 \left(\int_{T(B\cap B(\om,
r))}\frac{|D_Hg(q)|^4}{\|g(q)\|^4}d(q,\partial
T(B))^{\ep\frac{4}{p}} dq \right)^{\frac{p}{4}}
  \left(\int_{T(B\cap B(\om, r))} d(q,\partial T(B))^{\frac{4(p-1-\ep)}{4-p}}
  dq\right)^{\frac{4-p}{4}}.\notag
 \end{align}
Here, $\varepsilon>0$ is a suitably chosen constant (depending
only on $p$) that makes the last integral converge. To see that
such a constant $\varepsilon$ exists, recall from Proposition
\ref{p:T_w_r}, formula \eqref{eq5} and the choice of parameters
according to Remark \ref{r:Choice_x_a_r}, that
\begin{equation}\label{eq:T on_B_w_r}
d(T(q),\partial TB) \sim \frac{1}{r}
d(q,\partial B),\quad J_T(q) \sim \frac{1}{r^4},\quad q\in
B(\omega,r) \cap B,
\end{equation}
with implicit constants independent of $\omega$ and $r$. Hence,
for any $\eta<1$, we find
\begin{align}\label{eq:sum}
\int_{T(B\cap B(\omega,r))} d(q,\partial T(B))^{-\eta}\,dq&=
\int_{B\cap B(\omega,r)}
d(T(q),\partial T(B))^{-\eta} J_T(q)\,dq\notag\\
&\overset{\eqref{eq:T on_B_w_r}}{\sim} \frac{r^{\eta}}{r^4}
\int_{B\cap B(\omega,r)} d(q,\partial B)^{-\eta} \,dq\notag\\
&\sim   \frac{r^{\eta}}{r^4}\sum_{j=0}^{\infty} r^{-\eta} 2^{\eta
j} |\{q\in B\cap B(\omega,r):\,r2^{-j-1}\leq
d(q,\partial B)< r 2^{-j}\}| \notag\\
&\lesssim \frac{1}{r^4} \sum_{j=0}^{\infty}  2^{\eta j} |\{q\in
B\cap B(\omega,r):\, d(q,\partial B)< r 2^{-j}\}|.
\end{align}
It remains to find a good upper bound for the Lebesgue measure of
the sets
\begin{displaymath}
A_j(\omega,r):= \{q\in B\cap B(\omega,r):\, d(q,\partial B)< r
2^{-j}\},\quad j\in \mathbb{N}_0.
\end{displaymath}
First, we observe that
\begin{equation}\label{eq:cover}
A_j(\omega,r) \subset \bigcup_{\widetilde{\omega}\in
B(\omega,2r)\cap
\partial B} B(\widetilde{\omega}, r 2^{-j}).
\end{equation}
Indeed, if $q\in A_j(\omega,r)$, then $d(q,\omega_q)< r 2^{-j}$
for some $\omega_q\in \partial B$ with $d(q,\omega_q)=
d(q,\partial B)$. Since
\begin{displaymath}
d(\omega_q,\omega)\leq d(\omega_q,q)+d(q,\omega) <r2^{-j} + r,
\end{displaymath}
we see that $\omega_q \in B(\omega,2r)\cap
\partial B$, so \eqref{eq:cover} holds. Then we apply the
$5r$-covering lemma to find a disjoint subfamily
\begin{displaymath}
\{B(\widetilde{\omega}_i, r 2^{-j}):\,
i=1,2,\ldots,I_j(\omega,r)\}
\end{displaymath}
such that the $5$-times enlarged balls still cover
$A_j(\omega,r)$. To conclude the argument, we control the
cardinality $I_j(\omega,r)$ from above by observing that
\begin{displaymath}
I_j(\omega,r)\, r^3 2^{-3j}\lesssim
\mathcal{S}^3\left(\bigcup_{i=1}^{I_j(\omega,r)}
B(\widetilde{\omega}_i, r 2^{-j})\cap \partial B \right) \leq
\mathcal{S}^3 (B(\omega, 3r)\cap \partial B) \lesssim r^3
\end{displaymath}
by the $3$-regularity of $\mathcal{S}^3|_{\partial B}$, recall
Lemma \ref{l:3Reg}. Thus $I_j(\omega,r)\lesssim 2^{3j}$, and hence
\begin{equation}\label{eq:A_j_est}
|A_j(\omega,r)|\lesssim I_j(\omega,r)\, r^4 2^{-4j} \lesssim r^4
2^{-j}.
\end{equation}
Inserting this estimate in \eqref{eq:sum}, we conclude that for
every $\eta<1$ (positive or negative), it holds
\begin{displaymath}
\int_{T(B\cap B(\omega,r))} d(q,\partial T(B))^{-\eta}\,dq
\lesssim \sum_{j=0}^{\infty}  2^{(\eta-1) j}\lesssim_{\eta} 1.
\end{displaymath}

Applying these considerations for $\eta =
4(\varepsilon+1-p)/(4-p)$, we see that if we choose $\varepsilon
=\varepsilon(p) < \frac{3}{4}p$, the second integral in
\eqref{l.5.6.est3} can be bounded from above by a finite constant
depending only on $p$.

Combining the above estimates, we have so far found that
\begin{align*}
\int_{B\cap B(\om, r)}\frac{|D_Hf(q)|^p}{\|f(q)\|^p}\,d(q,\partial
B)^{p-1} dq &\lesssim_{K,p} r^3 \left(\int_{T(B\cap B(\om,
r))}\frac{|D_Hg(q)|^4}{\|g(q)\|^4}d(q,\partial
T(B))^{\ep\frac{4}{p}} dq \right)^{\frac{p}{4}}\\
&\lesssim_{K,p}  r^3
\left(\int_{T(B)}\frac{|D_Hg(q)|^4}{\|g(q)\|^4}d(q,\partial
T(B))^{\ep\frac{4}{p}} dq \right)^{\frac{p}{4}}.
\end{align*}

 In order to estimate the last integral, we proceed as follows,
 using the fact that $g$ is quasiconformal,  with constant $K$,
 and the change-of-variables formula holds,
 \begin{align}
 \int_{T(B)}\frac{|D_Hg(q)|^4}{\|g(q)\|^4}d(q,\partial T(B))^{\ep\frac{4}{p}} dq
 & \lesssim_K \int_{T(B)}\frac{d(q,\partial T(B))^{\ep\frac{4}{p}}}{\|g(q)\|^4}J_g(q)
 dq\notag\\
 &\sim_K
 \int_{g(T(B))}\frac{d(g^{-1}(q),\partial T(B))^{\ep\frac{4}{p}}}{\|q\|^4} dq \notag \\
 & \sim_K I_1 +   I_2,
\label{eq:l5.6}
 \end{align}
 where
 \begin{displaymath}
 I_1:= \int_{g(T(B))\cap \{q: \|q\|<\|g(0)\|\}}\frac{d(g^{-1}(q),\partial T(B))^{\ep\frac{4}{p}}}{\|q\|^4} dq
 \end{displaymath}
 and
 \begin{displaymath}
 I_2:= \int_{g(T(B)\cap \{q: \|q\|\geq \|g(0)\|\}}\frac{d(g^{-1}(q),\partial T(B))^{\ep\frac{4}{p}}}{\|q\|^4}
 dq.
 \end{displaymath}
The integrals $I_1$ and $I_2$ can be bounded from above using the
first and second inequality in Proposition~\ref{p:dist_alpha},
respectively:
\begin{displaymath}
I_1 \lesssim_{p,K} \int_{\|q\|< \|g(0)\|} \|q\|^{-4}
\frac{\|q\|^{\frac{\varepsilon 4}{p \alpha}
}}{\|g(0)\|^{\frac{\varepsilon 4}{p \alpha} }}\, dq \lesssim_{p,K}
\frac{1}{\|g(0)\|^{\frac{\varepsilon 4}{p \alpha} }} \int_0^{
\|g(0)\|} s^{\frac{\varepsilon 4}{p \alpha} -1} \, ds \sim_{p,K} 1
\end{displaymath}
and
\begin{displaymath}
I_2 \lesssim_{p,K} \int_{\|q\|\geq \|g(0)\|} \|q\|^{-4}
\frac{\|g(0)\|^{\frac{\varepsilon 4}{p \alpha}
}}{\|q\|^{\frac{\varepsilon 4}{p \alpha} }}\, dq \lesssim_{p,K}
\|g(0)\|^{\frac{\varepsilon 4}{p \alpha} } \int_{
\|g(0)\|}^{\infty} s^{-\frac{\varepsilon 4}{p \alpha} -1} \, ds
\sim_{p,K} 1.
\end{displaymath}
In conclusion, we have shown that
\begin{equation}\label{eq:CarlesonSmallScale}
\int_{B\cap B(\om, r)}\frac{|D_Hf(q)|^p}{\|f(q)\|^p}\,d(q,\partial
B)^{p-1} dq \lesssim_{K,p} r^3,\quad \text{for all }\omega \in
\partial B, \, 0<r<r_{\ast}.
\end{equation}

Finally, the argument by H\"older's inequality and
change-of-variables $q\mapsto g(q)$ that we applied above to
``$g$'' and ``$T(B)$'' works also for ``$f$'' and ``$B$'' to show
that
\begin{equation}\label{eq:CarlesonLargeScale}
\int_{B}\frac{|D_Hf(q)|^p}{\|f(q)\|^p}\,d(q,\partial B)^{p-1} dq
\lesssim_{K,p} 1.
\end{equation}
More precisely, we have
\begin{align*}
\int_{B}\frac{|D_H f(q))|^p}{\|f(q)\|^p}\,d(q,\partial B))^{p-1}
dq\leq \left(\int_{B}\frac{|D_Hf(q)|^4}{\|f(q)\|^4}d(q,\partial
B)^{\ep\frac{4}{p}} dq \right)^{\frac{p}{4}}
  \left(\int_{B} d(q,\partial B)^{\frac{4(p-1-\ep)}{4-p}}
  dq\right)^{\frac{4-p}{4}},
 \end{align*}
where $\varepsilon$ is as before. The first integral can be
bounded by a finite constant depending only on $K$ and $p$ by the
same argument as we used for \eqref{eq:l5.6}, with ``$g$'' and
``$T(B)$'' replaced by ``$f$'' and ``$B$'' The second integral is
finite constant depending only on $p$, thanks to
\eqref{eq:A_j_est} (for $r\sim 1$). As remarked at the beginning
of the proof, \eqref{eq:CarlesonSmallScale} and
\eqref{eq:CarlesonLargeScale} together suffice to show the
Carleson measure condition for all scales $r>0$.
\end{proof}

The following corollary of Proposition \ref{l:5.6} extends
\cite[Lemma 4.2]{MR554817} to $\Hei$.

\begin{cor}\label{cor:form_Carles_p}
If $f$ is a quasiconformal map on $B\subset \Hei$ with
$f(q)\not=0$ for all $q\in B$, then
 \begin{equation*}
 \left|\nabla_H \log \|f(q)\| \right| \, dq
 \end{equation*}
 defines a Carleson measure on $B$, with an upper bound for the Carleson measure constant
 depending only on $K$.
\end{cor}

\begin{proof}
Since for Lebesgue almost every $q\in B$, we have
\begin{displaymath}
 \left|\nabla_H \log \|f\| (q)\right| = \frac{|\nabla_H
 (\|\cdot \|\circ f)(q)|}{\|f(q)\|}\lesssim \frac{|D_Hf(q)|}{\|f(q)\|}
\end{displaymath}
this is an immediate consequence of Proposition \ref{l:5.6} for
$p=1$. The above inequality can be verified by a direct
computation using the contact equations for $f$, or observing that
$|\nabla_H \|\cdot\|| \leq1$ and using the chain rule for Pansu
derivatives.
 \end{proof}

\section{Characterization of the $H^p$ property for quasiconformal mappings} \label{s:ZinsmeisterChar}
The purpose of this section is to prove Theorem \ref{t:main2},
which characterizes membership in $H^p$ for a quasiconformal
mapping in terms of its radial limit and nontangential maximal
function, respectively. This is motivated by a result for
quasiconformal mappings on the Euclidean unit ball in
$\mathbb{R}^n$, $n\geq 2$, originally due to Zinsmeister
\cite{MR860655}, and later obtained with different methods by
Astala and Koskela in \cite[Theorem 4.1]{MR2900163}. Our proof
combines elements from both approaches with arguments tailored to
the Heisenberg geometry.

\subsection{Conditions implying $p$-integrability of the radial
limit} By Lemma \ref{l:QC_RadialLimit} and \eqref{eq:ac}, the
radial limit $f^{\ast}$ of a quasiconformal map $f$ on the
Kor\'anyi unit ball $B$ exist for $\mathcal{S}^3$-almost every
boundary point, and it is a Borel function. We will use this fact
throughout the section to prove the straightforward implications
$(1)\Rightarrow (3)$, $(2)\Rightarrow (3)$, and $(2)\Rightarrow
(1)$ in Theorem \ref{t:main2}.

\begin{lemma}\label{l:FromH^pToMaximalFunction}
Fix $0<p<\infty$ and $f:B \subset \mathbb{H}^1 \to f(B) \subset
\mathbb{H}^1$ be a quasiconformal map. Then
\begin{displaymath}
\int_{\partial B} \|f^{\ast}\|^p \, d\mathcal{S}^3\leq
\|f\|_{H^p}^p.
\end{displaymath}
In particular, $f\in H^p$ implies $\|f^{\ast}\|\in
L^p(\mathcal{S}^3|_{\partial B})$.
\end{lemma}

\begin{proof} Since $f$ is quasiconformal and $f^{\ast}$
exists $\mathcal{S}^3|_{\partial B}$ almost everywhere on
$\partial B$, the claim follows by Fatou's lemma:
\begin{align*}
\int_{\partial B} \|f^{\ast}(\omega)\|^p \, d\mathcal{S}^3(\omega)
&= \int_{\partial B} \|\lim_{r\nearrow1}f(\gamma(r,\omega))\|^p
\, d\mathcal{S}^3(\omega)\\
&\leq \liminf_{r\nearrow 1}  \int_{\partial B}
\|f(\gamma(r,\omega))\|^p \,
d\mathcal{S}^3(\omega)\overset{\text{Def.}\ref{d:Hp_Heis}}{\leq}
\|f\|_{H^p}^p.
\end{align*}
\end{proof}

In what follows we will frequently refer to
Proposition~\ref{p:curve_cone}. It asserts that there exists
$\kappa>0$ such that for every $\omega \in \partial B \setminus
\{z=0\}$, we have
\begin{displaymath}
\gamma(s,\omega)\in\Gamma(\omega):=
\Gamma_{\kappa}(\omega),\quad\text{for all }s\in (0,1),
\end{displaymath}
see the Appendix for the proof. The nontangential maximal function
$M(f)$ is defined with respect to that parameter $\kappa$. We also
recall from Remark~\ref{r:NTMaxMeas} that $M(f)$ is Borel
measurable. These observations immediately yield the next two
propositions.

\begin{proposition}\label{p:Zinsmeister(3)_to_(1)} Let $f:B
\subset \mathbb{H}^1 \to f(B) \subset \mathbb{H}^1$ be a
quasiconformal map. Then
\begin{displaymath}
\|f^{\ast}(\omega)\|\leq M(f)(\omega),\quad \mathcal{S}^3\text{
a.e. }\omega \in \partial B.
\end{displaymath}
In particular, for $0<p<\infty$, the condition $M(f)\in
L^p(\mathcal{S}^3|_{\partial B})$ implies that $\|f^{\ast}\|\in
L^p(\mathcal{S}^3|_{\partial B})$.
\end{proposition}

\begin{proposition}
Let $0<p<\infty$ and let $f: B \to f(B) \subseteq \mathbb{H}^1$ be
a quasiconformal map. Then
\begin{displaymath}
\|f\|_{H^p}^p \leq \int_{\partial B} M(f)^p\, d\mathcal{S}^3.
\end{displaymath}
In particular, $M(f) \in L^p\left(\mathcal{S}^3|_{\partial
B}\right)$ implies that $f\in H^p$.
\end{proposition}

\subsection{Conditions implied by the $p$-integrability of the radial limit}
In this section, we prove the main implication in Theorem
\ref{t:main2}, namely (3)\,$\Rightarrow$\,(2):

\begin{proposition}\label{p:FromRadialLimitToNTMax}
Assume that $K\geq 1$, $0<p<\infty$, and that $f:B \to f(B)
\subset \mathbb{H}^1$ is $K$-quasiconformal. Then
\begin{displaymath}
\|M(f)\|_{L^p(\mathcal{S}^3|_{\partial B})} \leq C(K,p) \big\|
\|f^{\ast}\| \big\|_{L^p(\mathcal{S}^3|_{\partial B})}
\end{displaymath}
for a constant $C(K,p)$ that depends only on $K$ and $p$.
\end{proposition}

Proposition \ref{p:FromRadialLimitToNTMax} is a counterpart for
Zinsmeister's result \cite[Proposition 1]{MR860655} in Euclidean
spaces, which was re-proven with a different argument by Astala
and Koskela in \cite[Corollary 4.3 and Theorem 4.1]{MR2900163}. We
establish our result by combining the use of a specific Carleson
measure inspired by \cite{MR860655} with a Hardy-Littlewood
maximal function argument as in \cite{MR2900163}. The core of the
proof is then to relate the nontangential maximal function of $f$
and the Hardy-Littlewood maximal function of its radial limit.
This is achieved with Lemma \ref{l:ZLemma6}, where the geometry of
the Heisenberg group enters the picture due to the nonisotropic
nature of the Kor\'{a}nyi ball. In contrast to~\cite{MR860655}, we
do not involve $\log \|f\|$ in our discussion, although
Corollary~\ref{cor:form_Carles_p} shows that for a quasiconformal
$f$ on $B\subset \mathbb{H}^1$, omitting the origin,  $|\nabla_H
\log \|f\||$ always defines a Carleson measure on $B$. In
\cite{MR860655}, an analogous result is used to observe that if a
Sobolev mapping $f$ is such that $f\not=0$ in $B$ and $|\nabla
\log \|f\||$ defines a Carleson measure, then the radial limit
$f^{\ast}$ exists (see~\cite[pg. 128]{MR860655}), whereas for us
$f$ is quasiconformal and $f^{\ast}$ exists by
Lemma~\ref{l:QC_RadialLimit}. The Carleson measure induced by
$|\nabla \log \|f\||$ appears in \cite{MR860655} also in a more
subtle way, through condition \cite[(10)]{MR860655}. There is a
similar element in our argument, where we apply the Carleson
measure defined by $|D_H f|/\|f\|$ to establish Lemma
\ref{l:ZLemma5}, but our proof necessarily looks different due to
the presence of characteristic points in the Kor\'{a}nyi sphere
and the rigidity of M\"obius self-maps of the Kor\'{a}nyi ball. We
now turn to the details.

For a Borel function $h:\partial B \to [0,+\infty]$, we define the
\emph{non-centred Hardy-Littlewood maximal function}
\begin{displaymath}
\mathcal{M}_{\partial B}h(\omega):= \sup_{B(\omega',r)\ni \omega}
\frac{1}{\mathcal{S}^3(B(\omega',r)\cap \partial B)} \int_{B(\omega',r)\cap
\partial B} h \, d\mathcal{S}^3,\quad \text{for all }\omega\in
\partial B.
\end{displaymath}
Proposition \ref{p:FromRadialLimitToNTMax} will be established
through a series of intermediate results, but the core of the
argument is the  chain of inequalities
\begin{equation}\label{eq:chain}
\int_{\partial B} M(f)^p\;d\mathcal{S}^3 = \int_{\partial B}
\left(\sup_{\Gamma(\cdot)}
\|f\|^q\right)^{\frac{p}{q}}\;d\mathcal{S}^3 \overset{\text{Lem.
\ref{l:ZLemma6}}}{\lesssim} \int_{\partial B}
(\mathcal{M}_{\partial
B}(\|f^{\ast}\|^q))^{\frac{p}{q}}\,d\mathcal{S}^3 \lesssim
\int_{\partial B} \|f^{\ast}\|^p\,d\mathcal{S}^3,
\end{equation}
for $0<q<p$ and every quasiconformal map $f: B \to f(B) \subset
\mathbb{H}^1$, with implicit constants depending on $p,q$, and
$K$.  Exactly as in the proof of \cite[Theorem 4.1]{MR2900163},
the last inequality holds since the operator
$\mathcal{M}_{\partial B}$ is of strong type $(s,s)$ for all
$s>1$, see e.g., \cite{MR447954,heinonen2012lectures} and recall
that $(\partial B, d|_{\partial B},\mathcal{S}^3|_{\partial B})$
is a doubling metric measure space. By fixing a suitable constant
$q$, depending on $p$, Proposition  \ref{p:FromRadialLimitToNTMax}
follows from \eqref{eq:chain}, so we concentrate on Lemma
\ref{l:ZLemma6}.

\begin{lemma}\label{l:ZLemma6} For $K\geq 1$ and $0<q<\infty$,
there exists a constant $C$ such that for every quasiconformal map
$f: B \to f(B) \subset \mathbb{H}^1$, its radial limit $f^{\ast}$
satisfies
\begin{displaymath}
\sup_{\Gamma(\omega)} \|f\|^q  \leq C \mathcal{M}_{\partial
B}(\|f^{\ast}\|^q)(\omega),\quad \text{for all }\omega \in
\partial B.
\end{displaymath}
\end{lemma}

Lemma \ref{l:ZLemma6} follows from the subsequent result:

\begin{lemma}\label{l:AKCOr4.3}
Let $0<q<\infty$ and $K\geq 1$. Then there exists a constant $C$,
depending on $q$ and $K$, such that for every quasiconformal map
$f:B \to f(B) \subset \mathbb{H}^1$, we have
\begin{equation}\label{eq:MeanValue}
\|f(x)\|^q \leq C \frac{1}{\mathcal{S}^3(S(x))} \int_{S(x)}
\|f^{\ast}(\omega)\|^q\, d\mathcal{S}^3(\omega),\quad \text{for
all }x\in B.
\end{equation}
Here $S(x)= B\left(x,(1+\kappa)d(x,\partial B)\right)\cap \partial
B$ with $\kappa$ given by Proposition~\ref{p:curve_cone}.
\end{lemma}

\begin{proof}[Proof of Lemma \ref{l:ZLemma6} using Lemma
\ref{l:AKCOr4.3}] Under the assumptions of Lemma \ref{l:AKCOr4.3},
we have for every $\omega_0 \in \partial B$ that
\begin{equation}\label{eq:FirstMax}
\sup_{x\in \Gamma(\omega_0)} \|f(x)\|^q \leq C \sup_{x\in
\Gamma(\omega_0)} \frac{1}{\mathcal{S}^3(S(x))} \int_{S(x)}
\|f^{\ast}(\omega)\|^q\, d\mathcal{S}^3(\omega).
\end{equation}
Now for every $x \in B$, there exists $\omega_x\in \partial B$
such that $d(x,\omega_x)= d(x,\partial B)$ and it is easy to see
that then
\begin{equation}\label{eq:ball_cap_incl}
B\left(\omega_x,\kappa\, d(x,\partial B)\right)\cap \partial B
\subset S(x) \subset B\left(\omega_x, (2+\kappa)\,d(x,\partial
B)\right) \cap
\partial B.
\end{equation}
Using the $3$-regularity of $\mathcal{S}^3|_{\partial B}$, we can
deduce from \eqref{eq:FirstMax} that
\begin{align*}
&\sup_{x\in \Gamma(\omega_0)} \|f(x)\|^q \\&\leq C \sup_{x\in
\Gamma(\omega_0)} \frac{1}{\mathcal{S}^3(B\left(\omega_x,
(2+\kappa)\,d(x,\partial B)\right) \cap
\partial B)} \int_{B\left(\omega_x, (2+\kappa)\,d(x,\partial
B)\right) \cap
\partial B}
\|f^{\ast}(\omega)\|^q\, d\mathcal{S}^3(\omega),
\end{align*}
where  $C$ now also depends on $\kappa$ and the $3$-regularity
constant of  $\mathcal{S}^3|_{\partial B}$, which we consider as
universal constants. Since $\omega_0 \in
B(\omega_x,(2+\kappa)d(x,\partial B)\cap\partial B)$ if $x\in
\Gamma(\omega_0)$, the right-hand side of the above inequality can
be bounded from above by $\mathcal{M}_{\partial
B}\|f^{\ast}\|^q(\omega_0)$ on $\partial B$ as claimed.
\end{proof}

Thus Proposition \ref{p:FromRadialLimitToNTMax} will follow if we
manage to prove Lemma \ref{l:AKCOr4.3}.

\subsubsection{Points and shadows: proof of Lemma~\ref{l:AKCOr4.3}} The main
ingredient for Lemma~\ref{l:AKCOr4.3} is a statement for a
quasiconformal map $f$ on $B$, saying that if $0\notin f(B)$, then
$\|f^{\ast}(\omega)\|$ cannot be too small compared to $\|f(x)\|$
for too many  $\omega \in S(x)$.

\begin{lemma}\label{l:ZLemma5} For every $K\geq 1$, there exists
$N(K)$ and a function $\Psi_K$ with $\lim_{N \to \infty}
\Psi_K(N)=0$ such that the following holds. If $f:B \to f(B)
\subset \mathbb{H}^1\setminus \{0\}$ is $K$-quasiconformal, then,
for all $x\in B$ and all $N\geq N(K)$, we have
\begin{displaymath}
\mathcal{S}^3 \left(\left\{\omega \in S(x):\, \|f^{\ast}(\omega)\|
\leq \frac{\|f(x)\|}{N} \right\} \right) \leq \Psi_K(N)
\mathcal{S}^3(S(x)).
\end{displaymath}
\end{lemma}

This lemma is easy to prove by standard modulus techniques if $x$
is deep inside $B$ and the shadow $S(x)$ is large, see Lemma
\ref{l:CentredLemma4} for the case $x=0$. In Euclidean spaces, the
general case can be reduced to this one by a suitable M\"obius
self-map of the unit ball that sends a small spherical cap to a
large one in a canonical, metrically controlled way, see the proof
of \cite[Lemma 4.2]{MR2900163}. M\"obius automorphisms of the
Kor\'{a}nyi unit ball $B$ are not flexible enough for this
approach. In our setting, we therefore give a separate proof of
the statement in Lemma \ref{l:ZLemma5} for $x$ close to $\partial
B$ with the help of a Carleson measure provided by Proposition
\ref{l:5.6}. To do so, we again make use of the growth estimate in
Proposition \ref{p:dist_alpha} and the specific M\"obius
transformations introduced in Section \ref{sect-Mobius}. Yet this
time we start from a point $x\in B$ (close to $\partial B$), and
assign to it the data
\begin{displaymath}
\omega_x \in \partial B,\quad a_x \in \mathbb{H}^1 \setminus
\overline{B},\quad \rho_x>0
\end{displaymath}
and associated M\"obius transformation
\begin{equation}\label{eq:Tx}
T_x:= T_{x,a_x,\rho_x}
\end{equation}
as described in the following. Without loss of generality, we may
assume that the constant $r_{\ast}$ in Remark \ref{r:Choice_x_a_r}
satisfies $r_{\ast}<1$. Then there exists $R_{\ast} \in
[1-r_{\ast},1)$ such that
\begin{displaymath}
d(y,\partial B) <r_{\ast}\quad \text{for all }\|y\|>R_{\ast}.
\end{displaymath}
Now, for $x\in B$ with $\|x\|>R_{\ast}$, we define:
\begin{itemize}
\item[(d1)] $\rho_x:=d(x,\partial B)$,\item[(d2)] $\omega_x$ a
point in $\partial B$ such that $d(x,\omega_x)= d(x,\partial B)$,
\item[(d3)] $a_x:= A_{0,M_0 N \rho_x}(\omega_x)$ (outer corkscrew
point with ``$N$'' as in Remark \ref{r:Choice_x_a_r}).
\end{itemize}

To prove Lemma \ref{l:ZLemma5} for $\|x\|>R_{\ast}$, we will study
the behavior of $f\circ T_x^{-1}$ on
\begin{displaymath}
T_x\left(\left[\bigcup_{\omega \in S(x)} B(\omega,\rho_x)\right]
\cap B\right),
\end{displaymath}
recalling that $T_x(x)=0$. The precise statement, given in Lemma
\ref{l:ZCorollary4}, therefore requires us to control the time for
which the radial segment $\gamma(\cdot,\omega)$ stays inside
$B(\omega,\rho_x)$.

\begin{lemdef}\label{l:curve_cone} Let $\kappa$ be as in Proposition
\ref{p:curve_cone}. For all $\omega \in \partial B \setminus
\{z=0\}$ and $\rho \in (0,1)$, there exists $s_{\omega,\rho}\in
(0,1)$ such that
\begin{enumerate}
\item $\gamma(s,\omega)\in B(\omega,\rho)\cap
\Gamma_{\kappa}(\omega)$ for all $s\in [s_{\omega,\rho},1)$, \item
$\gamma(s_{\omega,\rho},\omega) \in \partial B(\omega,\rho)$,
\item $s_{\omega,\rho}\geq 1-\rho$.
\end{enumerate}
Moreover, the choice $\omega \mapsto s_{\omega,\rho}$ can be made
Borel measurable on $\partial B \setminus \{z=0\}$. If $x\in B
\setminus \{0\}$ and $\rho=\rho_x:= d(x,\partial B)$, we denote
$s_{\omega,x}:= s_{\omega,\rho_x}$.
\end{lemdef}

\begin{proof} Let $\omega \in \partial B \setminus \{0\}$. By Proposition
\ref{p:curve_cone} we know that $ \gamma(s,\omega) \in
\Gamma_{\kappa}(\omega)$ for all $s\in (0,1)$. Since the radial
curves are continuous with
\begin{displaymath}
d(\gamma(0,\omega),\omega)=1\quad \text{and}\quad
d(\gamma(1,\omega),\omega)=0,
\end{displaymath}
we can take $s_{\omega,\rho}$ to be the largest number in $(0,1)$
such that
\begin{displaymath}
\gamma(s_{\omega,\rho},\omega)\in \partial
B(\omega,\rho).\end{displaymath} By maximality, this satisfies
also
\begin{displaymath}
\gamma(s,\omega)\in B(\omega,\rho),\quad\text{for all }s\in
[s_{\omega,\rho},1).\end{displaymath} Moreover, it is clear that
\begin{displaymath}
s_{\omega,\rho}=\|\gamma(s_{\omega,\rho},\omega)\|\geq \|\omega\|-
d(\gamma(s_{\omega,\rho},\omega),\omega) \geq 1- \rho.
\end{displaymath}
To prove the Borel measurability of $\omega \mapsto
s_{\omega,\rho}$, we use that the function $(s,\omega)\mapsto
d(\gamma(s,\omega),\omega)$ is continuous on $(0,1)\times
(\partial B\setminus \{z=0\})$ and it extends to a continuous
function $h:[0,1]\times \partial B  \to [0,1]$, cf.\,
\eqref{eq:dis_gamma_omega}. Then, for a given $0<\rho<1$, the
function
\begin{displaymath}
\omega \mapsto s(\omega):=\max\{s\in[0,1]:\, h(s,\omega)=\rho\}
\end{displaymath}
is upper semicontinuous. Indeed, for any $\lambda\leq 1$, if
$s(\omega)<\lambda$, then $h(s,\omega)<\rho$ for all $s \in
[\lambda,1]$ while $h(s(\omega),\omega)=\rho$, and thus there
exists a small relatively open neighborhood $U$ of $\omega$ in
$\partial B$ such that for all $\omega'\in U$ and $s\geq \lambda$,
we also have $h(s,\omega')<\rho$, and hence $s(\omega')<\lambda$.
This shows that $\omega \mapsto s(\omega)$ is upper semicontinuous
as claimed, and in particular,
\[
\omega \mapsto
s_{\omega,\rho}:=\max\{s\in(0,1):\,d(\gamma(s,\omega),\omega)=\rho\}
\]
is a Borel function.
\end{proof}

If a point $x\in B$ is close to $\partial B$, then so is
$\gamma(s_{\omega,x},\omega)$. The map $T_x$ allows us to
normalize the situation in such a way that we obtain a point at a
uniformly bounded distance away from the boundary of the new
domain, independently of the choice of $x$ and $\omega\in S(x)$.

\begin{lemma}[Normalization]\label{l:ClaimMobRad} There exists a constant $c>0$ such that, for all $x\in B$ with
$\|x\|>R_{\ast}$ and all $\omega \in S(x)\setminus \{z=0\}$, we
have
\begin{displaymath}
d\left(T_x(\gamma(s_{\omega,x},\omega)),\partial T_x B\right) \geq
c.
\end{displaymath}
\end{lemma}

\begin{proof} Let $x\in B$ and $\omega\in S(x)$ be as assumed in the lemma. By definition of $T_x$, it holds for
every $\widetilde{\omega}\in
\partial B$ that
\begin{displaymath}
d\left(T_x(\gamma(s_{\omega,x},\omega)),T_x(\widetilde{\omega})\right)
 = \rho_x \frac{d(\gamma(s_{\omega,x},\omega),\widetilde{\omega})}{d(\gamma(s_{\omega,x},\omega),a_x)\,d(\widetilde{\omega},a_x)}
\end{displaymath}
Since $\omega\in S(x)$, we can apply Lemma \ref{l:curve_cone} (2)
 and Remark \ref{r:Choice_x_a_r} to deduce that
 \begin{align*}
 d\left(T_x(\gamma(s_{\omega,x},\omega)),T_x(\widetilde{\omega})\right)
 \gtrsim \rho_x \frac{d(\gamma(s_{\omega,x},\omega),\widetilde{\omega})}{\rho_x \,
 [d(\gamma(s_{\omega,x},\omega),\widetilde{\omega})+d(\gamma(s_{\omega,x},\omega),a_x)]}
 \gtrsim \frac{\rho_x}{\rho_x} =1.
 \end{align*}
In the last step we also used that
\[
 d(\gamma(s_{\omega,x},\omega),a_x) \lesssim \rho_x \simeq  d(\gamma(s_{\omega,x},\omega), \partial B)\lesssim  d(\gamma(s_{\omega,x},\omega), \widetilde{\omega}).
\]
\end{proof}

Lemma \ref{l:curve_cone} and Lemma \ref{l:ClaimMobRad} are useful
to deduce consequences of a small radial limit:

\begin{lemma}\label{l:ZCorollary4}
There exists a constant $C>0$ and, for every $K\geq 1$, a number
$N(K)>1$ such that the following holds. Whenever $x\in B$ is such
that $\|x\|>R_{\ast}$ and $f:B\to f(B) \subset
\mathbb{H}^1\setminus \{0\}$ is $K$-quasiconformal, then for all
$N\geq N(K)$ and $\mathcal{S}^3$ almost every $\omega \in S(x)$,
we have
\begin{displaymath}
\|f^{\ast}(\omega)\|\leq \frac{\|f(x)\|}{N}\quad \Rightarrow \quad
\int_{s_{\omega,x}}^1 \frac{|D_H (f\circ
T_x^{-1})(T_x(\gamma(s,\omega)))|}{\|f(\gamma(s,\omega))\|} s^3 \,
ds \geq C \rho_x |(\omega_1,\omega_2)| \log N
\end{displaymath}
where $T_x$ is the M\"obius transformation defined in
\eqref{l:curve_cone}, and $s_{\omega,x}$ is given by Lemma
\ref{l:curve_cone}.
\end{lemma}
We postpone the proof to Section~\ref{ss:ProofOfLemmaZCor4}. The
next result serves as a substitute for Lemma \ref{l:ZCorollary4}
in case $x\in B$ is far away from the boundary. This is a
Heisenberg version of \cite[Lemma 4.2]{MR2900163} in the special
case when the distinguished point is the origin. This case is
particularly simple since
\begin{displaymath}
S_{\kappa}(0) = \partial B \cap B(0,(1+\kappa)d(0,\partial B)) =
\partial B.
\end{displaymath}
Lemma  \ref{l:ZCorollary4} and Lemma \ref{l:CentredLemma4} look
different at first, but their similarity will become clear latest
in Remark \ref{r:Reform_Lemma4}. The integral in Lemma
\ref{l:ZCorollary4} is a new element in our proof, but the
inspiration for using it came from the proof of Zinsmeister's
\cite[Lemma 4]{MR860655}, where M\"obius self-maps of $B$ appear
through the definition of $\|u\|_{\ast}$ at the bottom of
\cite[p.127]{MR860655}.

\begin{lemma}\label{l:CentredLemma4}
If $f:B \to f(B) \subset \mathbb{H}^1\setminus \{0\}$ is
$K$-quasiconformal, then
\begin{displaymath}
\mathcal{S}^3\left(\left\{\omega \in \partial B:\;
\|f^{\ast}(\omega)\|\leq \tfrac{\|f(0)\|}{N} \right\}\right) \leq
C(K) \left(\ln N\right)^{-\frac{3}{4}},
\end{displaymath}
where the constant $C(K)$ depends only on $K$.
\end{lemma}

\begin{proof} We first prove a similar inequality for the
measure $\sigma$ on $\partial B \setminus \{0\}$. Namely, for
every $K \geq 1$, we show that there exists a constant $C(K)$ such
that if $f:B\subset \mathbb{H}^1 \to f(B) \subset \mathbb{H}^1
\setminus \{0\}$ is a $K$-quasiconformal map, then for all $N>1$,
one has
\begin{equation}\label{eq:statement:q_0}
\sigma\left(\left\{ \omega \in \partial B\setminus \{z=0\}:\;
\|f^{\ast}(\omega)\|< \|f(0)\|/N\right\}\right)\leq C(K)
\sigma(\partial B\setminus \{z=0\})(\ln N)^{-3}.
\end{equation}
As in the proof of \cite[Lemma 4.2]{MR2900163}, we set
$$
d(f(0), \partial f(B)):=c.
$$
  We apply Corollary 3.5 in~\cite{2017arXiv170702832A} with $g=f$, $U=B$, $U'=f(B)$,
  $\beta:=5$ and a ball $B=B(0,r_0)$, where $r_0=r_0(K)$ is such that $3kB\subset B$ ($k$ depends on $K$ only)
  to obtain that for all $q\in B(0,r_0)$ we have
\[
 d(f(0), f(q)) \leq \diam f(B(0,r_0))\leq d(f(0), \partial f(B(0,r_0)))\leq \frac{c}{2}.
\]
Since $c:=d(f(0), \partial f(B))\leq \|f(0)\|$, it holds for any
$q'\in f(B(0 ,r_0))$ that
\[
d(q',0)=\|q'\|\geq \|f(0)\|-d(q', f(0))\geq
\|f(0)\|-\frac{c}{2}\geq \frac12 \|f(0)\|.
\]
Thus $ f(B(0 ,r_0))\cap B(0, \frac12 \|f(0)\|)=\emptyset$. Then,
for any given $N>1$ we denote \begin{displaymath}E:=E_N:=\left\{
\omega \in
\partial B\setminus \{z=0\}:\; \|f^{\ast}(\omega)\|<
\|f(0)\|/N\right\},\end{displaymath} and we define a family of
radial curves $\Gamma_E=\Gamma(\partial B(0,r_0), E, B)$, see the
discussion following Theorem~\ref{t:polar}.  By
\eqref{eq:mod_formula} we obtain that $\mathrm{mod}_4(\Gamma_E) =
\pi^2 \left(\ln \frac{1}{r_0}\right)^{-3} \sigma(E)$. If $N>3$,
then (radial) curves in $f(\Gamma_E)$ have one endpoint in
$\partial f(B(0, r_0))$ and another in $B(0, \|f(0)\|/N)$ and so,
in particular, they connect the complement of $B(0, \|f(0)\|/2)$
to $B(0, \|f(0)\|/N)$. Hence, by \eqref{eq:ring},
\[
 \mathrm{mod}_4(f(\Gamma_E))\lesssim_K \left(\ln \frac{N}{2}\right)^{-3}\leq C (\ln N)^{-3}.
\]
Therefore, $\sigma(E)\lesssim_K \frac{(\ln
\frac{1}{r_0})^3}{\pi^2} (\ln N)^{-3}$. If $1<N\leq 3$, then the
estimate is trivial:
\begin{displaymath}
\sigma(E) \leq \left(\frac{\ln 3}{\ln N}\right)^3 \sigma (\partial
B\setminus \{z=0\}).
\end{displaymath}
Thus $ \sigma(E)\lesssim_K (\ln N)^{-3},
$
and the proof of \eqref{eq:statement:q_0} is complete.

\medskip

The statement of Lemma \ref{l:CentredLemma4} can be reduced to
this estimate. Since
\begin{displaymath}
d\sigma= \cos^2\alpha\, d\sigma_0\quad\text{and}\quad
d\mathcal{S}^3= \sqrt{\cos \alpha} \,d\sigma_0,
\end{displaymath}
H\"older's inequality with exponents $p=4$ and $q=4/3$ yields
\begin{displaymath}
\mathcal{S}^3(E) = \int_{E} \sqrt{\cos \alpha}\,d\sigma_0 \leq
\left[\int_{E} \cos^2 \alpha
\,d\sigma_0\right]^{\frac{1}{4}}\,\left[\int_E
\,d\sigma_0\right]^{\frac{3}{4}}= \sigma(E)^{\frac{1}{4}}\,
\sigma_0(E)^{\frac{3}{4}}.
\end{displaymath}
Now we bound $\sigma(E)$ from above with the help of
\eqref{eq:statement:q_0}, and thus conclude that
\begin{displaymath}
\mathcal{S}^3(E) \leq C(K)^{\frac{1}{4}}\, \sigma(\partial
B\setminus \{z=0\})^{\frac{1}{4}} \, \left(\ln
N\right)^{-\frac{3}{4}}\,\sigma_0(\partial B\setminus
\{z=0\})^{\frac{3}{4}}.
\end{displaymath}
\end{proof}

Lemmas \ref{l:ZCorollary4} and \ref{l:CentredLemma4} allow to
control the set of points $\omega$ in $S(x)$ where
$\|f^{\ast}(\omega)\|$ is small compared to $\|f(x)\|$, as made
precise by Lemma \ref{l:ZLemma5}.

\begin{proof}[Proof of Lemma \ref{l:ZLemma5}]
The proof is split in two cases: $\|x\|\leq R_{\ast}$ and
$\|x\|>R_{\ast}$, where the first case is handled with Lemma
\ref{l:CentredLemma4} and Proposition \ref{p:dist_alpha}, while
Lemma \ref{l:ZCorollary4} and the Carleson measure $|D_H
f(q)|/\|f(q)\|\, dq$ are used to treat the second case.

\medskip
Let us assume first that $\|x\|\leq R_{\ast}$. Since $0 \notin
f(B)$, Proposition \ref{p:dist_alpha} for $\Omega=B$ and $g=f$
implies that there exist constants $C_K,\alpha_K>1$ such that
\begin{displaymath}
\frac{\|f(x)\|}{\|f(0)\|}\leq C_K\, d(x,\partial B)^{-\alpha_K}
\leq C_K\, (1-R_{\ast})^{-\alpha_K}.
\end{displaymath}
Then Lemma \ref{l:CentredLemma4} with $\Psi(N)= \left(\ln
N\right)^{-\frac{3}{4}}$ implies that
\begin{align*}
\mathcal{S}^3 \left(\left\{\omega \in S(x):\,
\|f^{\ast}(\omega)\|\leq \frac{\|f(x)\|}{N}\right\}\right) & \leq
\mathcal{S}^3 \left(\left\{\omega \in \partial B:\;
\|f^{\ast}(\omega)\|\leq C_K \,(1-R_{\ast})^{-\alpha_K}\,
\frac{\|f(0)\|}{N}\right\}\right)\\
&\leq C(K) \Psi\left(\frac{N}{C_K \,(1-R_{\ast})^{-\alpha_K}}\right)\\
&\leq C(R^{\ast}) \,C(K)\, \Psi\left(\frac{N}{C_K
\,(1-R_{\ast})^{-\alpha_K}}\right)\,\mathcal{S}^3(S(x)),
\end{align*}
where we have used in the last inequality that
$\mathcal{S}^3(S(x))\gtrsim_{R_{\ast}} 1$ for $\|x\|\leq
R_{\ast}$. This is the case by the inclusion
\eqref{eq:ball_cap_incl}, the inequality $d(x,\partial B)\geq
1-R_{\ast}$, and the $3$-regularity of $\mathcal{S}^3|_{\partial
B}$ stated in Lemma~\ref{l:3Reg}.
 Thus the estimate in Lemma \ref{l:ZLemma5} holds for all $x\in \overline{B(0,R_{\ast})}$
 with any function $\Psi_K$ satisfying $\lim_{N\to
 \infty}\Psi_K(N)=0$ and
\begin{displaymath}
\Psi_K(N)\geq C(R^{\ast}) \,C(K)\,\Psi\left(\frac{ N}{C_K
\,(1-R_{\ast})^{-\alpha_K}}\right).
\end{displaymath}

\medskip In the second part of the proof, we assume that $\|x\|>
R_{\ast}$. Choosing $N(K)$ as in Lemma \ref{l:ZCorollary4}, we
find for $N\geq N(K)$ that
\begin{align*}
\mathcal{S}^3&\left(\left\{\omega \in S(x):\, \|f^{\ast}(\omega)\|
< \frac{\|f(x)\|}{N} \right\}\right) \\&\leq \mathcal{S}^3
\left(\left\{\omega \in S(x):\ \frac{1}{\rho_x\, C\, \log N\,
|(\omega_1,\omega_2)|}\int_{s_{\omega,x}}^1 \frac{|D_H(f\circ
T_x^{-1})(T_x(\gamma(s,\omega)))|}{\|f(\gamma(s,\omega))\|} \,s^3
\,ds \geq 1 \right\}\right)\\
&\leq \frac{1}{\rho_x \,C \,\log N}\int_{S(x)}
\frac{1}{|(\omega_1,\omega_2)|} \int_{s_{\omega,x}}^1
\frac{|D_H(f\circ
T_x^{-1})(T_x(\gamma(s,\omega)))|}{\|f(\gamma(s,\omega))\|}\,s^3
\,ds \,d\mathcal{S}^3(\omega)\\
&= \frac{1}{\rho_x \,C \,\log N} \int_0^1 \int_{S(x)}
\chi_{[s_{\omega,x},1]}(s) \frac{|D_H(f\circ
T_x^{-1})(T_x(\gamma(s,\omega)))|}{\|f(\gamma(s,\omega))\|} \,
d\sigma_0(\omega) \,s^3 ds,
\end{align*}
where we have used $d\mathcal{S}^3|_{\partial B}(\omega) =
|(\omega_1,\omega_2)| d\sigma_0(\omega)$ and the Borel
measurability of $\om \mapsto s_{\om,x}$, recall
Lemma~\ref{l:curve_cone}.

Let us take a closer look at the domain of the double integration.
By the choice of $s_{\omega,x}$ in Lemma \ref{l:curve_cone}, we
know that for all $\omega \in \partial B$, we have
\begin{equation}\label{eq:condI}
d(\gamma(s,\omega),\omega) \leq \rho_x,\quad \text{for all }s\in
[s_{\omega,x},1].
\end{equation}
If we assume that $\omega$ is contained in the spherical cap
$S(x)$, then by definition we also have
\begin{equation}\label{eq:condII}
d(\omega,x) \leq (1+\kappa)\rho_x.
\end{equation}
Finally, by definitions (d1)-(d2) below \eqref{eq:Tx}, we have
\begin{equation}\label{eq:condIII}
d(x,\omega_x)=\rho_x.
\end{equation}
Recalling that $S(x):=S_{\kappa}(x)$, a combination of
\eqref{eq:condI}, \eqref{eq:condII}, and \eqref{eq:condIII} shows
that
\begin{displaymath}
\{\gamma(s,\omega):\, s\in [s_{\omega,x},1],\,\omega\in S(x)\}
\subset B \cap B(\omega_x, (3+\kappa)\rho_x).
\end{displaymath}
Thus we can continue with the previous estimate as follows
\begin{align*}
\mathcal{S}^3\left(\left\{ \omega \in S(x):\,
\|f^{\ast}(\omega)\|< \frac{\|f(x)\|}{N}\right\}\right) & \leq
\frac{1}{\rho_x \, C\, \log N}\int_{B \cap  B(\omega_x,
(3+\kappa)\rho_x)} \frac{|D_H (f\circ
T_x^{-1})(T_x(q))|}{\|f(q)\|}\, dq.
\end{align*}
We now derive an upper bound for the operator norm of the
horizontal derivative appearing in that integral. The chain rule
for Pansu derivatives, the contact and the Lusin property of
quasiconformal maps yield for Lebesgue almost every $q\in B $ that
\begin{displaymath}
D_H(f\circ T_x^{-1})(T_x(q)) = D_H f(q) \, D_H T_x^{-1}(T_x(q)).
\end{displaymath}
Then, for almost every $q\in B \cap  B(\omega_x, (3+\kappa)\rho_x)$, we find by a similar computations as in the proof of Lemma~\ref{l:f=gT} that
\begin{align*}
|D_H(f\circ T_x^{-1})(T_x(q))|  \leq |D_H f(q)| \, |D_H
T_x^{-1}(T_x(q))|
&= |D_H f(q)|\, J_{T_x}(q)^{-\frac{1}{4}}\\
&= |D_H f(q)|\, \frac{d(a_x,q)^2}{\rho_x}\\
&\lesssim_{\kappa} |D_H f(q)|\,\rho_x.
\end{align*}
Here the first equation holds since $T_x^{-1}$ is
$1$-quasiconformal and the next equation is due to \eqref{eq5},
the formula for the Jacobian of $T_x$. Finally, the last
inequality holds because $q\in B \cap  B(\omega_x,
(3+\kappa)\rho_x)$ and $d(a_x,\omega_x)\lesssim \rho_x$ by the
choices we made below \eqref{eq:Tx} in (d1)-(d3). Inserting the
obtained estimate for $|D_H(f\circ T_x^{-1})(T_x(q))|$ in our
chain of inequalities, we find that
\begin{displaymath}
\mathcal{S}^3 \left(\left\{\omega \in S(x):\,
\|f^{\ast}(\omega)\|< \frac{\|f(x)\|}{N} \right\}\right) \leq
\frac{1}{C\, \log N} \int_{B\cap
B(\omega_x,(3+\kappa)\rho_x)}\frac{|D_H f(q)|}{\|f(q)\|} \,dq.
\end{displaymath}
Now we apply Proposition~\ref{l:5.6} for $p=1$ to deduce that
 \begin{displaymath}
 d\mu(q)= \frac{|D_H f(q)|}{\|f(q)\|}\,dq
\end{displaymath}
is a Carleson measure with Carleson measure constant depending
only on $K$. Hence,
\begin{align*}
\mathcal{S}^3 \left(\left\{\omega \in S(x):\,
\|f^{\ast}(\omega)\|< \frac{\|f(x)\|}{N} \right\}\right)
\lesssim_K \frac{1}{\log N} \rho_x^3\lesssim \frac{1}{\log
N}\mathcal{S}^3(S(x)),
\end{align*}
where in the last step we used the $3$-regularity of
$\mathcal{S}^3|_{\partial B}$, recall Lemma \ref{l:3Reg}, and
\eqref{eq:ball_cap_incl}. This concludes the proof of Lemma
\ref{l:ZLemma5} in the second case, that is, if $\|x\|>R_{\ast}$.
\end{proof}

With these preparations in hand, we are now ready to prove Lemma
\ref{l:AKCOr4.3}, following the proof of \cite[Corollary
4.3]{MR2900163} by Astala and Koskela.

\begin{proof}[Proof of Lemma \ref{l:AKCOr4.3}]
If $0\notin f(B)$, then we can directly apply Lemma
\ref{l:ZLemma5} by choosing $N$ large enough, depending only on
$K$, such that
\begin{displaymath}
\mathcal{S}^3 \left(\left\{\omega \in S(x):\, \|f^{\ast}(\omega)\|
\leq \frac{\|f(x)\|}{N} \right\} \right) < \frac{1}{2}
\mathcal{S}^3(S(x)),\quad \text{for all }x\in B.
\end{displaymath}
This yields \eqref{eq:MeanValue} in that case, cf., the proof of
Corollary 4.3 in~\cite{MR2900163}. If, on the other hand, $0\in
f(B)$, then we can choose a point $y\in \mathbb{H}^1$ such that
\begin{displaymath}
\|y\|\leq \|f^{\ast}(\omega)\|,\quad\text{for almost all } \omega
\in \partial B.
\end{displaymath}
Then the map $g:= L_{y^{-1}}\circ f$ is $K$-quasiconformal on $B$
with $0\notin g(B)$. Applying the inequality from the previous
case yields
\begin{align*}
\|f(x)\|^q &\leq 2^q \|g(x)\|^q + 2^q \|y\|^q\\
&\leq 2^q \|g(x)\|^q + \frac{2^q}{\mathcal{S}^3(S(x))}
\int_{S(x)} \|f^{\ast}(\omega)\|^q\, d\mathcal{S}^3(\omega)\\
&\leq \frac{2^q C} {\mathcal{S}^3(S(x))} \int_{S(x)}
\|g^{\ast}(\omega)\|^q\, d\mathcal{S}^3(\omega) +
\frac{2^q}{\mathcal{S}^3(S(x))} \int_{S(x)}
\|f^{\ast}(\omega)\|^q\, d\mathcal{S}^3(\omega)\\
&\leq \frac{2^q C} {\mathcal{S}^3(S(x))} \int_{S(x)}
(2\|f^{\ast}(\omega)\|)^q\, d\mathcal{S}^3(\omega) +
\frac{2^q}{\mathcal{S}^3(S(x))} \int_{S(x)}
\|f^{\ast}(\omega)\|^q\, d\mathcal{S}^3(\omega),
\end{align*}
which concludes the proof also in that case.
\end{proof}

\subsubsection{Proof of Lemma \ref{l:ZCorollary4}}\label{ss:ProofOfLemmaZCor4}

We will deduce Lemma \ref{l:ZCorollary4} for quasiconformal maps
$f$ on $B$ from a related statement for quasiconformal maps  on
certain conformal images of $B$, similarly to the reasoning in the
proof of Proposition~\ref{l:5.6}.

\begin{lemma}\label{l:ZLemma4}
There exists a constant $C>0$, and for every $K\geq 1$ a number
$N(K)>1$, such that for all $x\in B$ with $\|x\|>R_{\ast}$ with
associated M\"obius transformation $T_x$ as in \eqref{eq:Tx}, and
for all $K$-quasiconformal $g:T_x(B) \to g(T_x(B)) \subset
\mathbb{H}^1 \setminus \{0\}$, we have for all $N\geq N(K)$ and
$\mathcal{S}^3$ almost all $\omega=(\omega_1,\omega_2,\omega_3)
\in S(x)$ that
\begin{displaymath}
\|\lim_{s \to 1} g(T_x(\gamma(s,\omega)))\|\leq
\frac{\|g(T_x(x))\|}{N} \quad \Rightarrow \quad G_x(\omega) \geq C
\rho_x |(\omega_1,\omega_2)| \log N,
\end{displaymath}
where
\begin{displaymath}
G_x(\omega)= \int_{s_{\omega,x}}^1 \frac{|D_H
g(T_x(\gamma(s,\omega)))|}{\|g\|(T_x(\gamma(s,\omega)))}\,s^3
\,ds.
\end{displaymath}
\end{lemma}

\begin{remark}\label{r:Reform_Lemma4}
Lemma \ref{l:ZLemma4} implies
\begin{displaymath}
\mathcal{S}^3 \left(\omega \in S(x):\ \| \lim_{s\to 1}
g(T_x(\gamma(s,\omega))) \| < \frac{\|g(0)\|}{N}\right) \leq
\frac{1}{C\,\rho_x \,|(\omega_1,\omega_2)|\,\log N}
\|G_x\|_{L^1(\mathcal{S}^3|_{\partial B})}
\end{displaymath}
\end{remark}

Lemma \ref{l:ZCorollary4}  follows immediately by applying Lemma
\ref{l:ZLemma4} to $g:= f \circ T_x^{-1}$. So Lemma
\ref{l:ZLemma4} is the last missing piece. We prove it by applying
Proposition \ref{p:dist_alpha}, which is possible thanks to the
normalization provided by $T_x$. This is inspired by ideas from
\cite{MR860655}, but at the same time geometric properties of
radial curves and M\"obius transformations in $\mathbb{H}^1$ play
an important role in our argument via the following auxiliary
result.

\begin{lemma}\label{l:green} For a point $x\in B$ with
$\|x\|>R_{\ast}$, let $T_x$ be the associated M\"obius
transformation defined in \eqref{eq:Tx}, and let $g:T_x(B) \to
g(T_x(B))\subset \mathbb{H}^1 \setminus \{0\}$ be quasiconformal.
Then, for $\mathcal{S}^3$ almost every  $\omega \in S(x)$ and
almost every $s\in (0,1)$, we have
\begin{displaymath}
|D_H g(T_x(\gamma(s,\omega)))|\geq \left| \frac{\partial}{\partial
s} \|g(T_x(\gamma(s,\omega)))\|\right|
\frac{|(\omega_1,\omega_2)|}{J_{T_x}(\gamma(s,\omega))^{1/4}}.
\end{displaymath}
\end{lemma}

\begin{proof}
Since $g\circ T_x$ is quasiconformal, the curve $s\mapsto
g(T_x(\gamma(s,\omega)))$ is horizontal for $\mathcal{S}^3$ almost
every $\omega \in S(x)$. Fix such
$\omega=(\omega_1,\omega_2,\omega_3)$. Applying Lemma
\ref{l:deriv_r} to $f=g\circ T_x$ and using the chain rule we find
\begin{align*}
\left|\frac{\partial}{\partial s}
\|g(T_x(\gamma(s,\omega)))\|\right| &\leq \frac{|(g\circ
T_x)_I(\gamma(s,\omega)))|}{\|g(T_x(\gamma(s,\omega)))\|} |D_H
g(T_x(\gamma(s,\omega)))| |D_H T_x(\gamma(s,\omega))|
\frac{1}{|(\omega_1,\omega_2)|}.
\end{align*}
We conclude that
\begin{align*}
|D_H g(T_x(\gamma(s,\omega)))|&\geq \left|\frac{\partial}{\partial
s}\|g(T_x(\gamma(s,\omega)))\|\right|
\frac{\|g(T_x(\gamma(s,\omega)))\|}{|(g\circ
T_x)_I(\gamma(s,\omega)))|}\frac{|(\omega_1,\omega_2)|}{|D_H
T_x(\gamma(s,\omega))|} \\
&\geq  \left|\frac{\partial}{\partial
s}\|g(T_x(\gamma(s,\omega)))\|\right|
\frac{|(\omega_1,\omega_2)|}{J_{T_x}(\gamma(s,\omega))^{1/4}}.
\end{align*}
\end{proof}

With these preparations in hand, we can conclude the proof of Lemma~\ref{l:ZLemma4}, and thus, in particular the whole proof of Proposition~\ref{p:FromRadialLimitToNTMax}.

\begin{proof}[Proof of Lemma \ref{l:ZLemma4}]
Fix $x\in B$ with $\|x\|>R_{\ast}$, and a constant $N$ to be
determined (not to be confused with the universal constant $N$
from Remark~\ref{r:Choice_x_a_r}). To simplify notation, we denote
\begin{displaymath}
F_{x,N}:= \left\{\omega\in S(x)\text{ where $(g\circ
T_x)^{\ast}(\omega)$ exists and }\|\lim_{s\to 1}
g(T_x(\gamma(s,\omega)))\|<\frac{\|g(0)\|}{N} \right\}.
\end{displaymath}
Our goal is to prove for $\mathcal{S}^3$ almost every $\omega \in
F_{x,N}$ that
\begin{equation}\label{eq:goalGx}
G_x(\omega) \geq C \,\rho_x\, |(\omega_1,\omega_2)| \, \log N
\end{equation}
for some universal constant $C>0$. Indeed, for such $\omega$, we
find by the definition of $s_{\omega,x}$ (recall Lemma
\ref{l:curve_cone} (3)) and by the choice of $R_{\ast}$ below
\eqref{eq:Tx} that
\begin{displaymath}
G_x(\omega) \geq (1-r_{\ast})^3 \int_{s_{\omega,x}}^1
\frac{|D_Hg(T_x(\gamma(s,\omega)))|}{\|g\|(T_x(\gamma(s,\omega))}
\,ds.
\end{displaymath}
Applying the bound for the horizontal derivative given in Lemma
\ref{l:green}, we observe that
\begin{align*}
 \int_{s_{\omega,x}}^1
\frac{|D_Hg(T_x(\gamma(s,\omega)))|}{\|g\|(T_x(\gamma(s,\omega))}
\,ds& \geq   \int_{s_{\omega,x}}^1 \left|\frac{\partial_s
\|g(T_x(\gamma(s,\omega)))\|}{\|g(T_x(\gamma(s,\omega)))\|}
\right|
\frac{|(\omega_1,\omega_2)|}{J_{T_x}(\gamma(s,\omega))^{1/4}}\,
ds\\
&\geq \int_{s_{\omega,x}}^1 \left|\frac{\partial}{\partial s}\log
\|g(T_x(\gamma(s,\omega)))\| \right|
\frac{|(\omega_1,\omega_2)|}{J_{T_x}(\gamma(s,\omega))^{1/4}}\,
ds.
\end{align*}
By the formula for the Jacobian $J_{T_x}$ stated in \eqref{eq5},
the last expression equals
\begin{displaymath}
 \int_{s_{\omega,x}}^1
\left|\frac{\partial}{\partial s} \log
\|g(T_x(\gamma(s,\omega)))\|\right|
\frac{|(\omega_1,\omega_2)|\,d(a_x,\gamma(s,\omega))^2}{\rho_x}\,ds.
\end{displaymath}
By Remark \ref{r:Choice_x_a_r}, we have
\begin{displaymath}
d(a_x,\gamma(s,\omega))^2 \geq C \rho_x,
\end{displaymath}
and hence we obtain from the above chain of inequalities that
\begin{align*}
G_x(\omega)&\geq C^2 \,(1-r_{\ast})^3 \, \rho_x\,
|(\omega_1,\omega_2)|\, \left|\int_{s_{\omega,x}}^1
\frac{\partial}{\partial s} \log
\|g(T_x(\gamma(s,\omega)))\| \, ds\right|\\
&\geq C^2 \,(1-r_{\ast})^3 \,\rho_x\,|(\omega_1,\omega_2)|
\left|\log \frac{\|(g\circ
T_x)(\omega)\|}{\|g(T_x(\gamma(s_{\omega,x},\omega)))\|}\right|.
\end{align*}
To control the logarithm term, we first apply Proposition
\ref{p:dist_alpha} to obtain
\begin{equation}\label{eq:k}
\|g(0)\|\leq \frac{C_K
\|g(T_x(\gamma(s_{\omega,x},\omega)))\|}{d(T_x(\gamma(s_{\omega,x},\omega)),\partial
T_x B)^{\alpha_K}}
\end{equation}
for some constants $C_K$ and $\alpha_K$, which depend only on the
distortion $K$. To justify this application of Proposition
\ref{p:dist_alpha}, we observe that there exist universal
constants $0<m<M<\infty$ such that
\begin{equation}\label{eq:Ball_Incl}
B(0,m) \subset T_x(B) \subset B(0,M)
\end{equation}
holds for all $x\in B$ with $\|x\|\geq R_{\ast}$ and $T_x=T_{x,
a_x,\rho_x}$ defined in Proposition \ref{p:GeneralMob} for $a=a_x$
and $\rho=\rho_x$ as in (d1)-(d3) below \eqref{eq:Tx}. Indeed, the
second inclusion in \eqref{eq:Ball_Incl} can be arranged by
Corollary \ref{c:TB_in_ball} since the choice of parameters
$\rho_x$, $\omega_x$, and $a_x$ in (d1)-(d3) implies that
conditions \eqref{eq:ImplicitConstants} are satisfied for
$\rho=\rho_x$ and $a=a_x$:
\begin{displaymath}
d(a_x,\partial B) \overset{(d3)}{=} d(A_{o,M_0 N
\rho_x}(\omega_x),\partial B) \geq \rho_x \quad \text{and}\quad
d(a_x,x)\overset{(d1)-(d3)}{>} d(x,\partial B) \overset{(d1)}{=}
\rho_x.
\end{displaymath}
On the other hand, the first inclusion in \eqref{eq:Ball_Incl} can
be arranged by Corollary \ref{c:ball_in_TB} since (d1)-(d3) ensure
that conditions \eqref{eq:Cond_x_a_rho} are satisfied for our
choice of $x$, $\rho=\rho_x$, and $a=a_x$:
\begin{displaymath}
d(x,\partial B) \overset{(d1)}{=} \rho_x,\quad d(a_x,x)\leq
d(a_x,\omega_x)+ d(\omega_x,x) \overset{(d2),(d3)}{\lesssim}
\rho_x.
\end{displaymath}
Then we use the normalization provided by Lemma
\ref{l:ClaimMobRad} to continue the estimate in \eqref{eq:k}:
\begin{equation}\label{eq:denom}
\|g(0)\|\lesssim_K \|g(T_x(\gamma(s_{\omega,x},\omega)))\|.
\end{equation}
On the other hand, since $\omega \in F_{x,N}$, we know that
\begin{equation}\label{eq:nom}
\|(g\circ T_x)^{\ast}(\omega)\|< \frac{\|g(0)\|}{N}.
\end{equation}
Combining \eqref{eq:denom} and \eqref{eq:nom}, we find for some
constant $C(K)$, which depends only on the distortion $K$, that
\begin{displaymath}
\frac{\|(g\circ
T_x)^{\ast}(\omega)\|}{\|g(T_x(\gamma(s_{\omega,x},\omega))))\|}<
\frac{\|g(0)\|}{N} \frac{C(K)}{\|g(0)\|}
\end{displaymath}
and hence
\begin{displaymath}
\log \frac{\|(g\circ
T_x)^{\ast}(\omega)\|}{\|g(T_x(\gamma(s_{\omega,x},\omega))))\|} <
\log \frac{C(K)}{N} < \log \frac{1}{N^{1/2}} = - \frac{1}{2}\log N
(< 0)
\end{displaymath}
if we have initially chosen
\begin{displaymath}
N> C(K)^2=:N(K),
\end{displaymath}
 so that $C(K)/N < 1/N^{1/2}$. Thus for this
choice of $N$, the previous estimates yield
\begin{align*}
G_x(\omega)&\geq  C^2 \,(1-r_{\ast})^3
\,\rho_x\,|(\omega_1,\omega_2)| \left|\log \frac{\|(g\circ
T_x)(\omega)\|}{\|g(T_x(\gamma(s_{\omega,x},\omega)))\|}\right|\\
&\geq  C^2 \,(1-r_{\ast})^3 \,\rho_x\,|(\omega_1,\omega_2)|
\frac{1}{2} \log N
\end{align*}
for almost every $\omega = (\omega_1,\omega_2,\omega_3)\in
F_{x,N}$.
\end{proof}

\section{Carleson measures and radial limits of quasiconformal maps on $B$}\label{s:CarlesonChar}
We apply the results from the previous section, notably
Proposition \ref{p:FromRadialLimitToNTMax}, to characterize
Carleson measures on $B$. While this proposition concerns the
nontangential maximal function, the inequalities concerning this
maximal function will only be used as intermediate results and not
appear in the main result. Below we take the strategy to first
discuss a general type result in metric spaces and then apply it
in the Heisenberg setting.

\subsection{Carleson measures and nontangential maximal functions in metric spaces}

It is well known that several variants of Carleson's embedding
theorem on the Euclidean unit ball and half-space can be proven
using nontangential maximal functions as an intermediate tool, see
\cite[VII. 4.4]{MR0290095}, \cite[I. Exercise 19]{MR2450237} and
\cite[Corollary 4.5]{MR2900163}. The first step in these arguments
works in rather general metric spaces, as we now show. Later we
will apply this abstract result in the context of quasiconformal
maps on the Kor\'{a}nyi unit ball.

Recall that a non-empty domain $\Om\subset X$ of a metric space
$(X,d)$ has \emph{$s$-regular boundary} for some $s>0$, if its
boundary is Ahlfors $s$-regular with respect to the Hausdorff
measure on $X$ restricted to $\partial \Om$, i.e., there exists a
constant $C\geq 1$ such that
\begin{displaymath}
C^{-1}\, r^s \leq \mathcal{H}^s(B(x,r)\cap \partial \Omega)\leq C
\,r^s,\quad \text{for all }x\in \partial \Omega\text{ and
}0<r<\mathrm{diam}(\partial \Omega).
\end{displaymath}

\begin{definition}\label{def: Carleson-msp}
 Fix $1\leq \alpha<\infty$ and $s>0$. Let $(X,d)$ be a metric space and $\Omega
\subset X$ a domain with nonempty $s$-regular boundary. We say
that a (positive) Borel measure $\mu$ on $\Omega$ is an
\emph{$\alpha$-Carleson measure on  $\Omega$} if there exists a
constant $C>0$ such that
\begin{equation}\label{eq:CarlesonCOnstAbstr}
\mu(\Omega \cap B(\omega,r))\leq C r^{\alpha\,s},\quad \text{for
all }\omega\in \partial\Omega\text{ and }r>0.
\end{equation}
The $\alpha$-\emph{Carleson measure constant} of $\mu$ is defined
by
\begin{displaymath}
\gamma_{\alpha}(\mu):= \inf\{C>0\text{ such that
\eqref{eq:CarlesonCOnstAbstr} holds for all $\omega \in\partial
\Omega$ and $r>0$}\}
\end{displaymath}
We also call $1$-Carleson measures simply \emph{Carleson
measures}.
\end{definition}

Recall Definition~\ref{d:NT} and Remark~\ref{r:NTMaxMeas}.

\begin{proposition}\label{p:NontangentialInequality} Fix $s>0$ and $1\leq \alpha <\infty$. Let $(X,d)$
be a proper metric space and let $\Omega \subset X$ be a bounded
domain with nonempty $s$-regular boundary and let $\kappa>0$ be
such that the nontangential region $\Gamma_{\kappa}(\omega)$ is
nonempty for all $\omega \in
\partial \Omega$. Assume that $\mu$ is an $\alpha$-Carleson
measure on $\Omega$. Then the $\kappa$-nontangential maximal
function $N_{\kappa}h$ of an arbitrary Borel function $h:\Omega
\to [0,+\infty)$ satisfies
\begin{equation}\label{eq:IntIneqFromCarlesonNT}
\int_{\Omega} h^{\alpha p}\,d\mu \leq C \left(\int_{\partial
\Omega}
(N_{\Omega,\kappa}h)^p\,d\mathcal{H}^s\right)^{\alpha},\quad\text{for
all }0<p<\infty
\end{equation}
where $C$ depends on $p$, $\alpha$, $s$, $\kappa$,
$\gamma_{\alpha}(\mu)$, and the $s$-regularity constant of
$\partial \Omega$. If $\alpha=1$, then  $C$ can be chosen
independently of $p$.
\end{proposition}

Our proof of Proposition \ref{p:NontangentialInequality} is
inspired by the first part of the proof of \cite[Corollary
4.5]{MR2900163} in the context of the Euclidean unit ball. We
generalize this approach using a Whitney decomposition in abstract
doubling metric spaces, in the spirit of \cite[Theorem
3.2]{MR447954}. For this purpose, it will be more natural to work
directly with covering balls centered at points in $\partial
\Omega$, rather than analogs of the spherical caps $S(x_k)$,
$k=1,2,\ldots$, in \cite{MR2900163}.

\begin{proof} Let $p$,
$\alpha$, $s$, $\kappa$, $\gamma_{\alpha}(\mu)$, and $h$ be as in
the statement of the proposition. For simplicity, we abbreviate
throughout the proof $N_{\kappa} h:= N_{\Omega,\kappa}h$ and
$S(x):=S_{\kappa}(x)$. To prove \eqref{eq:IntIneqFromCarlesonNT},
 we define the superlevel sets
\begin{displaymath}
E(\lambda):=\{x\in \Omega:\, h(x)>\lambda\}\quad\text{and}\quad
U(\lambda):= \{\omega \in
\partial \Omega:\, N_{\kappa}h(\omega)>\lambda\},\quad \lambda >0.
\end{displaymath}
It suffices to show that there exists a constant $C$, depending
only on $\alpha$, $s$, $\kappa$, $\gamma_{\alpha}(\mu)$, and the
$s$-regularity constant of $\partial \Omega$, such that
\begin{equation}\label{eq:superlevel}
\mu(E(\lambda)) \leq C
\mathcal{H}^s(U(\lambda))^{\alpha}\quad\text{for all }\lambda>0.
\end{equation}

 If $\alpha=1$, then \eqref{eq:IntIneqFromCarlesonNT},
with $C$ independent of $p$, follows immediately from
\eqref{eq:superlevel} by a standard application of Cavalieri's
principle. If $\alpha>1$, then by a similar reasoning, one
concludes as follows:
\begin{align*}
\int_{\partial \Omega} h^{\alpha p}\, d\mu = \alpha p
\int_0^{\infty} \lambda^{\alpha p -1}\mu(E(\lambda))\, d\lambda
&\leq C \int_0^{\infty} \lambda^{\alpha p
-1}\mathcal{H}^s(U(\lambda))^{\alpha}\, d\lambda\\&\lesssim
\left(\sum_{j=-\infty}^{\infty}
\mathcal{H}^s(U(2^j))2^{jp}\right)^{\alpha} \lesssim
\left(\int_{\partial \Omega} (N_{\kappa}h)^p
\,d\mathcal{H}^s\right)^{\alpha},
\end{align*}
where the implicit constants now depend additionally also on
$\alpha$ and $p$.

It remains to prove the superlevel set estimate
\eqref{eq:superlevel}, which we will do by a Whitney-type
decomposition of  $U(\lambda)$ if the latter is a strict subset of
$\partial \Omega$ (otherwise the claim is trivial). Recall that
$U(\lambda)$ is a relatively open subset in $\partial \Omega$, and
$(\partial \Omega,d|_{\partial \Omega})$ is metrically doubling
since it is $s$-regular. Thus we can for instance apply the
general result \cite[Proposition 4.1.15]{MR3363168} to the metric
space $(\partial \Omega, d|_{\partial \Omega})$ and the open set
$U(\lambda)$ to find a countable collection
$\mathcal{W}_{\lambda}=\{B(\omega_i,r_i):\, i=1,2,\ldots\}$ of
balls with $\omega_i \in U(\lambda)$ such that
\begin{equation}\label{eq:Whitney1} U(\lambda) = \bigcup_{i=1,2,\ldots}
B(\omega_i,r_i) \cap \partial \Omega,
\end{equation}
\begin{equation}\label{eq:Whitney2}
\sum_i \chi_{B(\omega_i,2 r_i)\cap \partial \Omega}\leq 2 N^5,
\end{equation}
where $r_i = (1/8) d(\omega_i,\partial \Omega\setminus
U(\lambda))$ and $N$ depends only on $s$ and  the $s$-regularity
constant of $\partial \Omega$. (For our purposes, it would in fact
suffice to obtain \eqref{eq:Whitney2} with ``$B(\omega_i, r_i)$''
instead of ``$B(\omega_i,2 r_i)$''.) To prove the superlevel set
estimate, we want to show that $E(\lambda)$ is included in the
union of the balls $B(\omega_i,Cr_i)$, for a suitable geometric
constant $C=C(\kappa)$. If $x$ is an arbitrary point in
$E(\lambda)$, then
\begin{displaymath}
N_{\kappa}h(\omega)>\lambda,\quad \text{for all }\omega\in S(x)=
B\left(x,(1+\kappa)d(x,\partial \Omega)\right)\cap
\partial \Omega,
\end{displaymath}
and hence
\begin{equation}\label{eq: S(q)_in_U(lambda)}
S(x) \subset U(\lambda) \overset{\eqref{eq:Whitney1} }{=}
\bigcup_i B(\omega_i,r_i) \cap \partial \Omega,\quad \text{for all
}x\in E(\lambda).
\end{equation}
Next, for $x\in E(\lambda)$, let $\omega_x \in \partial B$ be such
that
\begin{equation}\label{eq:flower}
d(x,\omega_x)= d(x,\partial \Omega).
\end{equation}
Such a point may not be unique, but there exists at least one
since $\partial \Omega$ is compact. By definition, $\omega_x\in
S(x)$, and therefore \eqref{eq: S(q)_in_U(lambda)} implies that
there exists $i_x\in \{1,2,\,\ldots\}$ such that $\omega_x \in
B(\omega_{i_x}, r_{i_x})$. Since $S(x) \subset U(\lambda)$, we
moreover know that
\begin{equation}\label{eq:Carleson_star}
d(\omega_x,\partial \Omega \setminus U(\lambda)) \geq
d(\omega_x,\partial \Omega \setminus S(x)).
\end{equation}
Combining this information, we find that
\begin{align}\label{eq:triangle}
r_{i_x} = \frac{1}{8} d(\omega_{i_x},\partial \Omega\setminus
U(\lambda)) &\geq \frac{1}{8} \left[d(\omega_x,\partial \Omega
\setminus U(\lambda)) - d(\omega_x,\omega_{i_x}) \right]\notag\\
&\overset{\eqref{eq:Carleson_star}}{\geq} \frac{1}{8}
d(\omega_x,\partial \Omega\setminus S(x)) - \frac{1}{8} r_{i_x}.
\end{align}
Since $d(\omega_x,x)= d(x,\partial \Omega)$, it is easy to see
that
\begin{displaymath}
B\left(\omega_x,\kappa d(x,\partial \Om)\right)\cap \partial \Omega
\subset S(x).
\end{displaymath}
Hence  the above chain of inequalities implies that
\begin{displaymath}
9 \,r_{i_x} \overset{\eqref{eq:triangle}}{\geq}d(\omega_x,
\partial \Omega \setminus S(x)) \geq d(\omega_x,\partial \Omega\setminus
B\left(\omega_x,\kappa d(x,\partial \Omega)\right)).
\end{displaymath}
As the right-hand side of the above inequality is bounded from
below by $\kappa d(x,\partial \Omega)=\kappa d(x,\omega_x)$, we
obtain that
\begin{equation}\label{eq:Carleson_4star}
x\in B\left(\omega_{i_x},\left(\tfrac{9}{\kappa}+1\right) r_{i_x}
\right).
\end{equation}
Since $x$ was chosen arbitrarily from $E(\lambda)$, we have thus
shown that $E(\lambda)$ is covered by the countable family of
balls $B(\omega_i,C r_i)$, $i=1,2,\ldots$, where $C=C(\kappa)=
\frac{9}{\kappa}+1$. Using that $\mu$ is an $\alpha$-Carleson
measure by assumption, the fact that $\mathcal{H}^s|_{\partial \Omega}$ is
$s$-regular, and the multiplicity of the Whitney balls is
controlled by \eqref{eq:Whitney2}, we deduce that
\begin{align*}
\mu(E(\lambda)) \leq \mu\left(\bigcup_i B(\omega_i,Cr_i) \cap
\Omega\right)&\leq \sum_i \mu(B(\omega_i, C r_i)\cap \Omega)
\\&\leq \gamma_{\alpha}(\mu)
\,\left(\tfrac{9}{\kappa}+1\right)^{s\alpha}\, \sum_i
r_i^{s\alpha}
\\
&\leq \gamma_{\alpha}(\mu)\,
\left(\tfrac{9}{\kappa}+1\right)^{s\alpha}\,\left( \sum_i
r_i^s\right)^{\alpha}\\
&\lesssim \left( \sum_i \mathcal{H}^s(B(\omega_i,r_i)\cap
\partial \Omega)\right)^{\alpha}\overset{\eqref{eq:Whitney2}}{\lesssim}
\mathcal{H}^s(U(\lambda))^{\alpha},
\end{align*}
as desired. This concludes the proof of Proposition
\ref{p:NontangentialInequality}, as explained above.
\end{proof}

\subsection{Carleson measures and nontangential maximal functions in $\Hei$}
We now apply Proposition \ref{p:NontangentialInequality} in
$\Hei$ and observe that in this setting  it can be strengthened to
give an integral inequality involving the radial limit of a
quasiconformal map. That such inequalities actually
\emph{characterize} Carleson measures on the Kor\'{a}nyi ball $B$
follows by an application of the Heisenberg radial stretch map
\cite{MR3076803}. Moreover, we provide consequences of this
characterization in Section~\ref{ss:appl-necC}, where we present a
relation between Hardy spaces and Bergman-type spaces for
quasiconformal maps on $B$, see Theorem~\ref{t: Ap-spaces}.

\begin{thm}\label{t:NecCarleson}
Let $1\leq \alpha<\infty$ and assume that $\mu$ is an
$\alpha$-Carleson measure on $B$. If $f:B \to f(B) \subset
\mathbb{H}^1$ is $K$-quasiconformal, then
\begin{equation}\label{eq:IntIneqFromCarleson}
\int_B \|f(q)\|^{\alpha p}\,d\mu(q) \leq C_p \left(\int_{\partial
B}
\|f^{\ast}(\omega)\|^p\,d\mathcal{S}^3(\omega)\right)^{\alpha},\quad\text{for
all }0<p<\infty,
\end{equation}
where $C_p$ depends only on $p$, $\alpha$, $K$, and a Carleson
measure constant $\gamma_{\alpha}(\mu)$.

Conversely, for every $K\geq 1$, there is $p(K) < 3$ such that if
$p>p(K)$ is fixed and $\mu$ is a Borel measure for which
\eqref{eq:IntIneqFromCarleson} holds for all $K$-quasiconformal
maps, then $\mu$ is an $\alpha$-Carleson measure.
\end{thm}

By Lemma \ref{l:FromH^pToMaximalFunction}, the first part of
Theorem \ref{t:NecCarleson} immediately yields a necessary
condition for $\alpha$-Carleson measures $\mu$ on $B$ in terms of
$L^{\alpha p}(\mu)$ integral inequalities for quasiconformal maps
in $H^p$ and $\|\cdot\|_{H^p}$.

\begin{cor}\label{c:NecCarlesonHp}
Let $1\leq \alpha<\infty$ and assume that $\mu$ is an
$\alpha$-Carleson measure on $B$. If $f:B \to f(B) \subset
\mathbb{H}^1$ is $K$-quasiconformal, then, for all $0<p<\infty$,
\begin{displaymath}
\left(\int_B \|f(q)\|^{\alpha p}\, d\mu(q)\right)^{\frac{1}{\alpha
p}}\leq C_p \|f\|_{H^p},
\end{displaymath}
where $C_p$ depends only on $p$, $\alpha$, $K$, and
$\gamma_{\alpha}(\mu)$.
\end{cor}

\begin{proof}[Proof of Theorem \ref{t:NecCarleson}]
In order to show the sufficiency part of the theorem, we apply
Proposition \ref{p:NontangentialInequality} to
$(X,d)=(\mathbb{H}^1,d)$, $\Omega = B$, and $h=\|f\|$, where $f:B
\to f(B) \subset \mathbb{H}^1$ is $K$-quasiconformal, and we
combine the result with
Proposition~\ref{p:FromRadialLimitToNTMax}.

For the proof of the necessity part of Theorem \ref{t:NecCarleson}, let $K\geq 1$ and assume that $p>p(K)$ for a constant $0<p(K)<3$ to be determined. We suppose that \eqref{eq:IntIneqFromCarleson}
holds for all $K$-quasiconformal mappings, and we will apply this
condition to a particular such map in order to deduce that $\mu$
has to be an $\alpha$-Carleson measure. The choice of the map in
question is inspired by the proof of \cite[Corollary
4.5]{MR2900163} and it involves a quasiconformal \emph{Heisenberg
radial stretch map}, see
\cite{MR1246889,MR3076803,MR3343051,MR3845546} for different
contexts in which such stretch maps have arisen. The only relevant
information for us is that there exists a $K$-quasiconformal map
$f_K:\mathbb{H}^1 \to \mathbb{H}^1$ with  $f_K(\partial B(0,r))=
\partial B(0,r^{\beta(K)})$ for some $\beta(K)\geq 1$ (this follows by
considering the inverse of the map $f_k$ studied in \cite[Section
4.1]{MR3076803}).

Let us verify the $\alpha$-Carleson measure condition for $\mu$ at
an arbitrary point $\omega_0 \in \partial B$. Since
\eqref{eq:IntIneqFromCarleson} holds for all $K$-quasiconformal
maps on $B$, it holds in particular for left translations, which
shows that $\mu(B)<\infty$. For this reason, it suffices to verify
the $\alpha$-Carleson measure condition of $\mu$ for $0<r<r_0$,
where $r_0\in (0,1)$ is as in Remark \ref{r:Choice_x_a_r}
concerning the corkscrew property of $B$. Thus let us fix
$\omega_0 \in
\partial B$ and $0<r<r_0$. We will apply condition
\eqref{eq:IntIneqFromCarleson} to the map
\begin{displaymath}
f:=f_{\omega_0,r}: \mathbb{H}^1\setminus \{a\} \to
\mathbb{H}^1\setminus \{0\},\quad f(y):= f_K(I(a^{-1}\cdot y)),
\end{displaymath}
for suitably chosen $a=a(\omega_0,r)\in \mathbb{H}^1 \setminus
\overline{B}$, where $f_K$ is the $K$-quasiconformal radial
stretch map discussed above and
\begin{displaymath}
I(y)=-\frac{1}{\|y\|^{4}}\left(y_z(|y_z|^2+iy_t),
y_t\right)
\end{displaymath}
 is the $1$-quasiconformal inversion at $\partial B$. Since
left translations are $1$-quasiconformal as well, it follows that
$f|_B$ is $K$-quasiconformal and \eqref{eq:IntIneqFromCarleson} is
applicable. The point $a$ can be chosen  using the exterior
corkscrew condition of $B$ such that
\begin{equation}\label{eq:choice_a}
\frac{r}{M_0} \leq d(a,\partial B) \leq d(a,\omega_0) \leq r,
\end{equation}
recall Remark \ref{r:Choice_x_a_r}. It follows from the formula
for the inversion that $\|I(y)\|= 1/ \|y\|$ for all $y\in
\mathbb{H}^1 \setminus \{0\}$, and hence
\begin{displaymath}
\|f(y)\| = \|I(a^{-1}\cdot y)\|^{\beta(K)} =
\frac{1}{d(y,a)^{\beta(K)}},\quad y\in \mathbb{H}^1\setminus
\{a\}.
\end{displaymath}
For all $y\in B(\omega_0,r)\cap B$, we know by \eqref{eq:choice_a}
that $d(y,a) \leq 2 r$, and hence
\begin{displaymath}
1 = d(y,a)^{\beta(K)} \,\|f(y)\| \leq 2^{\beta(K)} \, r^{\beta(K)}
\|f(y)\|,\quad \text{for all }y\in B(\omega_0,r)\cap B.
\end{displaymath}
Hence, by \eqref{eq:IntIneqFromCarleson},
\begin{align*}
\mu(B(\omega_0,r)\cap B) &\leq 2^{\alpha p \beta(K)} r^{\alpha p
\beta(K)} \,\int_B \|f(y)\|^{\alpha p}\, d\mu(y)\\
&\lesssim_{\alpha,p,K} r^{\alpha p \beta(K)} \left(\int_{\partial
B}\|f^{\ast}(\omega)\|^p\,d\mathcal{S}^3(\omega)\right)^{\alpha}.
\end{align*}
Thus the $\alpha$-Carleson measure property of $\mu$ will follow
if we manage to show that
\begin{equation}\label{eq:UpperBoundSurfaceInt}
\int_{\partial B}\|f^{\ast}(\omega)\|^p\,d\mathcal{S}^3(\omega)
\lesssim_{p,K}  r^{3-\beta(K) p}.
\end{equation}
To upper bound the integral, we decompose the domain of
integration as follows:
\begin{displaymath}
\partial B \subset \bigcup_{j=0}^{\infty} A_j,\quad \text{where
}A_j:= \left\{\omega \in \partial B:\, \frac{r}{M_0}2^j\leq
d(\omega,a)< \frac{r}{M_0}2^{j+1} \right\}.
\end{displaymath}
The number of nonempty $A_j$ depends on $\omega_0$ and $r$, so we
need estimates that do not depend on this number. Indeed,
\begin{align*}
\int_{\partial B}\|f^{\ast}(\omega)\|^p\,d\mathcal{S}^3(\omega) =
\int_{\partial B}d(\omega,a)^{-\beta(K)
p}\,d\mathcal{S}^3(\omega)\lesssim \sum_{j=0}^{\infty}
r^{-\beta(K)p} 2^{-j \beta(K) p} \,\mathcal{S}^3(A_j).
\end{align*}
From \eqref{eq:choice_a}, it follows that
\begin{displaymath}
A_j \subset B\left(\omega_0,\left(\frac{1}{M_0}+1\right) 2^{j+1}r
\right)\cap \partial B,\quad j=0,1,\ldots,
\end{displaymath}
and hence, by the $3$-regularity of $\mathcal{S}^3|_{\partial B}$,
the integral can be further estimated as follows:
\begin{align*}
\int_{\partial B}\|f^{\ast}(\omega)\|^p\,d\mathcal{S}^3(\omega) &
\lesssim \sum_{j=0}^{\infty} r^{-\beta(K)p} 2^{-j \beta(K) p}
\,\mathcal{S}^3\left( B\Big(\omega_0,\left(\frac{1}{M_0}+1\right)
2^{j+1}r \Big)\cap \partial B\right)\\
&\lesssim  r^{3-\beta(K)p}\sum_{j=0}^{\infty} 2^{j(3-\beta(K) p)}.
\end{align*}
This yields \eqref{eq:UpperBoundSurfaceInt} provided that
$p>3/\beta(K)$, so that the above series is a finite constant
depending on $K$ and $p$. This concludes the proof with $p(K):=
3/\beta(K)$.
\end{proof}

\subsubsection{Applications of Theorem~\ref{t:NecCarleson}}\label{ss:appl-necC}
We apply the characterization of Carleson measures on $B\subset
\Hei$ to relate Hardy spaces and quasiconformal mappings
integrable on the unit ball (a counterpart of the Bergman spaces),
thus generalizing Theorem 9.1 in~\cite{MR2900163}. For a map
$f:B\to \Hei$ and $0<p<\infty$ we define
 \[
  \|f\|_{A^p}:=\left(\int_{B} \|f(q)\|^p\,dq\right)^{\frac1p},
 \]
 and write $f\in A^p$ if  $\|f\|_{A^p}<\infty$.

\begin{thm}\label{t: Ap-spaces}
Let $f$ be a quasiconformal mapping $f:B\to f(B) \subset \Hei$.
\begin{itemize}
\item[(1)] If $f\in H^p$, then $f\in A^{\frac43 p}$.
\item[(2)] If $f\in A^p$, then $f\in H^{p'}$ for all $0<p'<\frac34 p$.
\end{itemize}
\end{thm}

The proof of the theorem employs a quantity that is inspired by
\cite[Theorem 3.3]{MR2900163}. Given $f: B \subset \mathbb{H}^1
\to \mathbb{H}^1$, we define
\begin{displaymath}
M(r,f):= \sup_{q\in \Sigma_r} \|f(q)\|,\quad 0\leq r<1,
\end{displaymath}
where $\Sigma_r:= \{q\in B:\, d(q,\partial B)=1-r\}$. The set
$\Sigma_r$ is different from $\partial B(0,r)$ (it contains, e.g.,
the point $(0,0,1-(1-r)^2)=(0,0,r(2-r))$), and its definition is
tailored to play along well with the measure
$\mathcal{S}^3|_{\partial B}$.

\begin{proposition}\label{p:max-Hp}
Let $0<p<\infty$ and $K\geq 1$. Then every $K$-quasiconformal map
$f: B \to f(B) \subset \mathbb{H}^1$ with $f(0)=0$ satisfies
\begin{equation}\label{ineq:max-Hp}
\int_{\partial B} \|f^{\ast}(\omega)\|^{p}
\,d\mathcal{S}^3(\omega) \leq C \int_0^1 (1-r)^2 M(r,f)^{p}\, dr,
\end{equation}
for a constant $C$ that depends only on $p$ and $K$.

If $f: B \to f(B) \subset \mathbb{H}^1$ is an arbitrary
quasiconformal map, then
\begin{equation}\label{eq:Impl_M}
 \int_0^1 (1-r)^2 M(r,f)^{p}\, dr <\infty \quad \text{implies} \quad
 f\in H^p.
\end{equation}
\end{proposition}

We first show how Proposition \ref{p:max-Hp} implies Theorem
\ref{t: Ap-spaces}.

\begin{proof}[Proof of Theorem~\ref{t: Ap-spaces} based on Proposition \ref{p:max-Hp} ]
Let $f\in H^p$ be quasiconformal on $B\subset \Hei$. Since the
Lebesgue measure on $\Hei$ is $4$-Ahlfors regular, we obtain
by~\eqref{eq:CarlesonCOnstAbstr} for $s=3$ that its restriction to
$B$ is $\alpha$-Carleson for $\alpha=\frac43$ and hence $f\in
A^{\frac43 p}$ by Corollary~\ref{c:NecCarlesonHp}.

In order to show the second assertion of the theorem, we may
without loss of generality assume that $0\notin f(B)$, using
compositions with suitable left translations if necessary. Now for
every quasiconformal map $f: B \to f(B) \subset \mathbb{H}^1
\setminus \{0\}$, there exist constants $0<\lambda<1$ and $C>1$
such that
\begin{displaymath}
\|f(q)\|\leq   C \|f(y)\|,\quad y \in B(q, \lambda d(q,\partial
B)),\quad q\in B.
\end{displaymath}
This Harnack property follows for instance from \cite[Proposition
3.12]{2017arXiv170702832A}, observing that $d(f(y),\partial
f(B))\leq \|f(y)\|$ since $0\notin f(B)$. Hence it holds that
\[
 \|f(q)\|\lesssim  \left(\Barint_{B(q,\lambda d(q,\partial B))}\|f(y)\|^p\,dy\right)^{\frac1p}\lesssim \frac{1}{d(q,\partial B)^{\frac4p}},
\]
where in the last inequality we also use that $f\in A^p$. Therefore,
\[
 \int_0^1 (1-r)^2 M(r,f)^{p'}\, dr \lesssim \int_0^1 (1-r)^{2-\frac4p p'}\, dr
\]
 and the integral is finite if $3-\frac4p p'>0$. By Proposition
 \ref{p:max-Hp} this yields that $f\in H^{p'}$.
\end{proof}

\begin{proof}[Proof of Proposition~\ref{p:max-Hp}] We prove the
first part of the proposition under the assumption $f(0)=0$. If
the integral on the left-hand side of \eqref{ineq:max-Hp} was
computed with respect to the measure $\sigma$, which arises
naturally from the modulus formula \eqref{eq:mod_formula},  then
the proof of the estimate would follow almost verbatim the first
part in the proof of \cite[Theorem 3.3]{MR2900163} with ``$n$''
replaced by ``$4$''. The main challenge is to prove the stronger
inequality with $\sigma$ replaced by $\mathcal{S}^3|_{\partial
B}$. Analogously as in the proof of  \cite[Theorem
3.3]{MR2900163}, we define
\begin{displaymath}
E_{\lambda}:= \{\omega \in \partial B:\,
\|f^{\ast}(\omega)\|>\lambda\}.
\end{displaymath}
As $M(0,f)=0$, and since the sets $\Sigma_r$, $0<r<1$, are
topological spheres foliating $B$ by Proposition \ref{p:TopSph},
there exists a unique $r(\lambda) \in (0,1)$ such that
\begin{equation}\label{eq:r(lambda)}
2 \, M(r(\lambda),f) = \lambda
\end{equation}
whenever $E_{\lambda}\neq \emptyset$.
 Since
\begin{align*}
\int_{\partial B} \|f^{\ast}(\omega)\|^{p}
\,d\mathcal{S}^3(\omega) =& p \int_0^{\infty}
\mathcal{S}^3\left(\{\omega \in \partial B:\, \|f^{\ast}(\omega)\|>\lambda\}\right)\,\lambda^{p-1}\,d\lambda\\
\overset{\eqref{eq:r(lambda)}}{\leq} &\mathcal{S}^3(\partial B)\,
2^p \, M(1/2,f)^p\\& + p \int_{\{\lambda\in
(0,\infty]:\,1/2<r(\lambda)<1\}} \mathcal{S}^3\left(\{\omega \in
\partial B:\,
\|f^{\ast}(\omega)\|>\lambda\}\right)\,\lambda^{p-1}\,d\lambda,
\end{align*}
it suffices to prove that
\begin{equation}\label{eq:Elambda_goal}
\mathcal{S}^3(E_{\lambda}) \lesssim_K
(1-r(\lambda))^{3},\quad\text{for all $\lambda>0$ such that
}1/2<r(\lambda)<1.
\end{equation}
Indeed, if we manage to show \eqref{eq:Elambda_goal}, the proof
can be concluded exactly as below \cite[(3.7)]{MR2900163}. In
order to establish \eqref{eq:Elambda_goal}, we divide each
relevant $E_{\lambda}$ into two ``good'' parts at safe distance
from the characteristic points, and a ``bad'' part close to the
characteristic points. To state the definition, for every $\omega
\in \partial B \setminus \{z=0\}$ and $\lambda$ as before, we let
$s_{\omega}:=s_{\lambda,\omega}\in (0,1)$ be such that
\begin{displaymath}
\gamma(s_{\omega}, \omega) \in \Sigma_{r(\lambda)}.
\end{displaymath}
Similarly as in Lemma and Definition \ref{l:curve_cone},
$s_{\omega}$, one can make a Borel measurable choice $\omega
\mapsto s_{\omega}$.  Then we define
\begin{displaymath}
G_{0,\lambda}:=\{(\sqrt{\cos \alpha}e^{\mathrm{i}\varphi},\sin
\alpha)\in E_{\lambda}:\, \cos \alpha\geq 1/2 \},
\end{displaymath}
\begin{displaymath}
G_{\lambda}:= \{\omega=(\sqrt{\cos
\alpha}e^{\mathrm{i}\varphi},\sin \alpha)\in E_{\lambda}\setminus
G_{0,\lambda}:\; 1-s_{\omega} \leq \cos \alpha\},
\end{displaymath}
and
\begin{displaymath}
B_{\lambda}:= \{\omega=(\sqrt{\cos
\alpha}e^{\mathrm{i}\varphi},\sin \alpha)\in E(\lambda)\setminus
G_{0,\lambda}:\; 1-s_{\omega} > \cos \alpha\}.
\end{displaymath}
First, we observe that \eqref{eq:Elambda_goal} holds with
``$E_{\lambda}$'' replaced by ``$G_{0,\lambda}$''. Using the
modulus formula \eqref{eq:mod_formula} and $1-s_{\omega} \sim 1-r$
for $\omega = (\sqrt{\cos \alpha}e^{\mathrm{i}\varphi},\sin
\alpha)$ with $\cos \alpha\geq 1/2$ (cf.,\ e.g., Lemma
\ref{l:radial_curve_dist}), this can be shown exactly as in the
proof of \cite[Theorem 3.3]{MR2900163}. The resulting estimate is
a priori stated in terms of the measure $\sigma$, but
$\sigma(G_{0,\lambda})\sim \mathcal{S}^3(G_{0,\lambda})$ since the
measures $\sigma$ and $\mathcal{S}^3|_{\partial B}$ are comparable
on the parametric region $\{|\cos \alpha|\geq 1/2\}$.

Second, we prove  \eqref{eq:Elambda_goal} with ``$E_{\lambda}$''
replaced by ``$G_{\lambda}$''. We let $\Gamma_{\lambda}$ be the
family of radial segments in $B$ that connect
$\Sigma_{r(\lambda)}$ to $G_{\lambda}$. We use the same modulus
argument as in \cite{MR2900163}, but more subtle estimates for
$1-s_{\omega}$. Indeed, the standard modulus argument yields
\begin{align*}
\mathrm{mod}_4(\Gamma_{\lambda})  \geq \int
\left(\int_{s_{\omega}}^1 \frac{1}{s}\,ds\right)^{-3}
\cos^2\alpha\, d\alpha d\varphi= \int
\left(\cos^{-1/2}\alpha\int_{s_{\omega}}^1
\frac{1}{s}\,ds\right)^{-3}\sqrt{ \cos\alpha}\, d\alpha d\varphi.
\end{align*}
Since $1-s_{\omega} \leq \cos \alpha$ for $\omega = (\sqrt{\cos
\alpha}e^{\mathrm{i}\varphi},\sin \alpha)\in G_{\lambda}$ and
$r(\lambda)>1/2$, we obtain by Lemma
\ref{l:radial_curve_dist_bdry} and the definition of $s_{\omega}$
that
\begin{displaymath}
\int_{s_{\omega}}^1  \frac{1}{s}\,ds \lesssim (1-s_{\omega})
\lesssim (1-r(\lambda)) \sqrt{\cos \alpha}.
\end{displaymath}
This shows that
\begin{displaymath}
1 \gtrsim \mathrm{mod}_4(f(\Gamma_{\lambda}))\sim_K
\mathrm{mod}_4(\Gamma_{\lambda}) \gtrsim (1-r(\lambda))^{-3}\,
\mathcal{S}^3(G_{\lambda}),
\end{displaymath}
as desired for \eqref{eq:Elambda_goal}.

Finally, we show that \eqref{eq:Elambda_goal} holds also with
``$E_{\lambda}$'' replaced by the bad set ``$B_{\lambda}$''. If
$\omega = (\sqrt{\cos \alpha}e^{\mathrm{i}\varphi},\sin \alpha)\in
B_{\lambda}$, then
\begin{displaymath}
\cos \alpha \leq 1-s_{\omega} \lesssim (1-r(\lambda))^2,
\end{displaymath}
where we invoke  Lemma \ref{l:radial_curve_dist_bdry} for the last
estimate. Thus we have the crude estimate:
\begin{align*}
\mathcal{S}^3(B_{\lambda})& \leq 2 \pi \int_{\{\alpha \in
(-\pi/2,\pi/2):\, \cos \alpha \leq (1-r(\lambda))^2\}} \sqrt{\cos
\alpha}\,d\alpha\\
&\lesssim (1-r(\lambda)) \, \left|\left\{\alpha\in
(-\pi/2,\pi/2):\,  \cos \alpha \leq
(1-r(\lambda))^2\right\}\right|.
\end{align*}
Since $r(\lambda)\geq 1/2$ by assumption, we only have to consider
$\alpha \in (-\pi/2,\pi/2)$ with $\cos \alpha \leq 1/4$, so that
$\alpha$ is either close to $\pi/2$ or to $-\pi/2$. In the first
case, we have
\begin{displaymath}
\left|\alpha-\tfrac{\pi}{2}\right| \lesssim \left|\cos \alpha
-\cos \tfrac{\pi}{2}\right| = \cos \alpha \leq (1-r(\lambda))^2,
\end{displaymath}
and in the second case we obtain analogously that $\alpha$ lies in
an interval of length $(1-r(\lambda))^2$ around $-\pi/2$. This
shows as desired that
\begin{displaymath}
\mathcal{S}^3(B_{\lambda}) \lesssim (1-r(\lambda))^3.
\end{displaymath}

 Summing the upper bounds for
$\mathcal{S}^3(G_{0,\lambda})$, $\mathcal{S}^3(G_{\lambda})$ and
$\mathcal{S}^3(B_{\lambda})$ yields  \eqref{eq:Elambda_goal} and
thus yields the first part  of the proposition. By Theorem
\ref{t:main2}, the established inequality \eqref{ineq:max-Hp}
shows that \eqref{eq:Impl_M} holds if $f(0)=0$. The full statement
of the proposition can be reduced to this one. Indeed, if $f$
satisfies the assumption in  \eqref{eq:Impl_M}, but $f(0)\neq 0$,
then consider the new map $\widetilde{f}:= L_{f(0)^{-1}}\circ f$,
which has the desired property $\widetilde{f}(0)=0$. Since
\begin{displaymath}
M(r,\widetilde{f})= \sup_{q\in \Sigma_r} \|f(0)^{-1} f(q)\|\leq
\sup_{q\in \Sigma_r}\|f(q)\|+ \|f(0)^{-1}\|= M(r,f) +
\|f(0)^{-1}\|
\end{displaymath}
for all $r\in [0,1)$, it follows that $\widetilde{f}$ satisfies
the assumption in  \eqref{eq:Impl_M} and by the first part of the
proof, we conclude that  $\widetilde{f}\in H^p$. Then it also
follows that $f\in H^p$.
\end{proof}

We conclude this section by proving the topological result that we
applied earlier.

\begin{proposition}\label{p:TopSph}
The sets \begin{displaymath}\Sigma_r:=\{q\in B:\, d(q,\partial B)
= 1-r\},\quad 0<r<1,\end{displaymath} are topological spheres.
\end{proposition}

\begin{proof}
We fix $0<r<1$ and analyze the intersection of $\Sigma_r$ with
planes parallel to the $xy$-plane. Clearly,
\begin{displaymath}
\Sigma_r \cap \left(\mathbb{R}^2\times \{t\}\right) =
\emptyset\quad \text{for }t\in (-\infty,-1+(1-r)^2) \cup
(1-(1-r)^2,\infty),
\end{displaymath}
so it suffices to consider $t\in [-1+(1-r)^2,1-(1-r)^2]$. Since
rotations $R_{\varphi}$ about the $t$-axis are isometries for the
Kor\'{a}nyi metric, and the set $B$ is invariant under such
rotations, we observe that
\begin{displaymath}
\Sigma_r \cap \left(\mathbb{R}^2\times \{t\}\right) =
\{R_{\varphi}(x,0,t):\, (x,0,t)\in \Sigma_r\text{ and }\varphi \in
[0,2\pi)\}.
\end{displaymath}
We claim that the function
\begin{displaymath}
\delta: \mathbb{R} \to [0,+\infty),\quad \delta(x):=
d\left((x,0,t),
\partial B\right)
\end{displaymath}
is strictly monotone increasing on the interval
$(-\sqrt[4]{1-t^2},0)$, has a local maximum at $x=0$, and then
decreases strictly monotonically on $(0,1+\sqrt[4]{1+t^2})$. To
see this, we fix
\begin{equation}\label{eq:choice_x_y}
x \in (0,+\sqrt[4]{1+t^2})\quad \text{and}\quad v \in (-x,x).
\end{equation} We denote
\begin{displaymath}
\mathbf{x}_{-}:=(-x,0,t),\quad \mathbf{v}:=(v,0,t),\quad
\mathbf{x}_{+}:=(x,0,t)
\end{displaymath}
By rotational symmetry, $\delta(-x)$ equals $\delta(x)$. We obtain
that
\begin{equation}\label{eq:containedBalls}
B(\mathbf{x}_{-},\delta(x))\cup B(\mathbf{x}_+,\delta(x))\subset
B.
\end{equation}
We aim to deduce that  $B(\mathbf{v},\delta(x))$ lies in the
convex hull of $B(\mathbf{x}_-,\delta(x))$ and
$B(\mathbf{x}_+,\delta(x))$. Indeed, the latter two balls are left
translates of $B(\mathbf{v},\delta(x))$ by $(-x-v,0,0)$ and
$(x-v,0,0)$, respectively and therefore, every $p\in
B(\mathbf{v},\delta(x))$ lies on a line segment
\begin{displaymath}
\ell_{v,p}:= \{(s,0,0)\cdot p:\, s\in [-x-v,x-v] \}
\end{displaymath}
 starting in $B(\mathbf{x}_{-},\delta(x))$ and ending in
$B(\mathbf{x}_+,\delta(x))$. By convexity of $B$, the segment
$\ell_{v,p}$ is entirely contained in $B$, and in particular,
$B(\mathbf{v},\delta(x))\subset B$. Since $\overline{B}$ is
strictly convex, we obtain in fact that the closure of
$B(\mathbf{v},\delta(x))$ is contained in $B$, and hence
$d(\mathbf{v},\partial B)>\delta(x)$, as desired. Since this
argument works for arbitrary points $x$ and $v$ as in
\eqref{eq:choice_x_y}, we conclude that the distance function
$\delta$ has the claimed strict monotonicity properties. This
implies that that $\Sigma_r$ is foliated by circles. More
precisely, there exists a function  $\rho_r$ such that
\begin{equation}\label{eq:foliation}
\Sigma_r = \bigcup_{t\in [-1+(1-r)^2,1-(1-r)^2]} \{(x,y)\in
\mathbb{R}^2:\, x^2 +y^2= \rho_r(t)^2\}\times \{t\},
\end{equation}
and $\rho_r(t)=0$ for $t= -1+(1-r)^2$ or $t=1-(1-r)^2$.

 To
conclude the argument, we will show that $\rho_r$ is a continuous
function. To this end, let $t\in [-1+(1-r)^2,1-(1-r)^2]$ be
arbitrary, and let $(t_n)_n$ be a sequence of points in
$(-1+(1-r)^2,1-(1-r)^2)$ converging to $t$. Since $\Sigma_r$ is
compact (it is closed because $d(\cdot,\partial B)$ is
continuous), the sequence $((\rho_r(t_n),0,t_n))_n \subset
\Sigma_r$ has a subsequence that converges to a point
$(x',0,t')\in \Sigma_r$. Since the original sequence $(t_n)_n$
converges to $t$, we have $t'=t$. Moreover, since $\rho_r(t_n)\geq
0$ for all $n$, we have $x'\geq 0$, and  $(x',0,t)\in \Sigma_r$
implies that $x'= \rho_r(t)$. Repeating the same reasoning for
every subsequence of $(t_n)_n$, we obtain $\lim_{n\to
\infty}\rho_r(t_n)=\rho_r(t)$, as desired.  Hence $\rho_r$ is
continuous and the proposition follows by \eqref{eq:foliation}.
\end{proof}

\appendix
\section{Radial curves and nontangential regions}\label{s:radialsDist}
The purpose of the Appendix is to prove
Proposition~\ref{p:curve_cone}, which states that there exists
$\kappa>0$ such that for every $\omega\in \partial B\setminus
\{z=0\}$, the radial segment $\gamma(s,\omega)$ is contained in
the nontangential approach region $\Gamma_{\kappa}(\omega)$ for
all $s\in (0,1)$. This can be easily verified for
$\omega=(e^{\mathrm{i}\varphi},0)$, in which case
$\gamma(\cdot,\omega)$ is a horizontal line and
$d(\gamma(s,\omega),\omega)=d(\gamma(s,\omega),\partial B)$ for
all $s\in [0,1]$. It is also straightforward to check in the
limiting case that unit segments on the vertical axis are
contained in $\Gamma_{\kappa}(\omega)$ for $\omega=(0,0,\pm 1)$
and any choice of $\kappa$. For arbitrary $\omega$, the curve
$\gamma(s,\omega)$ stays close to the horizontal normal of
$\partial B$ through $\omega$ for some time $s\in [s(\omega),1]$
since it is obtained by a flow along a vector field tangential to
$\nabla_H \|\cdot\|$, see \cite[(3.1)]{MR1942237}. The distance
estimates for the remaining curve segment
$\gamma(\cdot,\omega)|_{[0,s(\alpha)]}$ are similar to those for
the vertical line segment. This is reminiscent of the construction
of John curves in \cite[(1.3)]{MR2135732}, but our focus lies on
verifying the John property for the given radial curves. This is
achieved
 by estimating
$d(\gamma(s,\omega),\omega)$ (Lemma \ref{l:radial_curve_dist}) and
 $d(\gamma(s,\omega),\partial B)$ (Lemma
\ref{l:radial_curve_dist_bdry}); the complete proof of
Proposition~\ref{p:curve_cone}
 is given at the end.

First, we compute the distance between points in $\partial B$ and
in the paraboloid
\[
P_{\alpha}:=\{(x,y,t)\in\mathbb{H}^1:\,\tfrac{t}{x^2+y^2}=\tan\alpha\},\quad
\alpha \in (-\pi/2,\pi/2).
\]
Let $s>0$, $p=(s \sqrt{\cos
\alpha}e^{\mathrm{i}\varphi},\sin\alpha)\in
\partial B(0,s)\cap P_\alpha$, and $\widetilde{\omega}\in
\partial B$, $\widetilde{\omega}=(\sqrt{\cos
\widetilde{\alpha}}e^{\mathrm{i}\widetilde{\varphi}},\sin
\widetilde{\alpha})$ be arbitrary. By applying a rotation about
the $t$-axis, we obtain
\begin{displaymath}
d(p,\widetilde{\omega})
 =\|(-s\sqrt{\cos \alpha},0, -s^2\sin \alpha)\cdot(\sqrt{\cos \widetilde{\alpha}}\cos( \widetilde{\varphi}-\varphi),
 \sqrt{\cos \widetilde{\alpha}}\sin( \widetilde{\varphi}-\varphi),\sin
 \widetilde{\alpha})\|.
\end{displaymath} Denoting $\phi:= \widetilde{\varphi}-\varphi$,
a straightforward computation yields that
\begin{align}\label{eq:distTwoPoints-rot}
&d(p,\widetilde{\omega})^4 \\
&=\left((\sqrt{\cos \widetilde{\alpha}}\cos \phi-s\sqrt{\cos
\alpha})^2+\cos \widetilde{\alpha}\sin ^2\phi\right)^2
+\left(\sin \widetilde{\alpha}-s^2\sin \alpha+2s\sqrt{\cos \alpha}\sqrt{\cos \widetilde{\alpha}}\sin \phi\right)^2 \nonumber \\
&=1+s^4+s^2(6\cos \alpha \cos \widetilde{\alpha}-2\sin \alpha \sin
\widetilde{\alpha}) -4s\sqrt{\cos \alpha}\sqrt{\cos
\widetilde{\alpha}}\left(\cos(\widetilde{\alpha}+\phi)+s^2\cos(\alpha-\phi)\right).\notag
\end{align}
We will estimate $d(\gamma(s,\omega),\omega)$ in two different
ways, which will yield better estimates depending on which range
of parameters we consider.

\begin{lemma}\label{l:radial_curve_dist}
If $\omega=\left(\sqrt{\cos \alpha}\cos
\varphi,\sqrt{\cos\alpha}\sin\varphi,\sin\alpha\right)$ with
 $\varphi \in [0,2\pi)$,
$\alpha \in (-\pi/2,\pi/2)$, then
\begin{align*}
d\left(\gamma(s,\omega),\omega\right) \lesssim
\min\left\{\frac{1-s}{\sqrt{\cos
\alpha}},\sqrt{1-s}+\sqrt[4]{1-s}\sqrt[4]{\cos
\alpha}\right\},\quad s\in (0,1].
\end{align*}
\end{lemma}

\begin{proof}
The bound by the first expression in the minimum follows if we
estimate the distance between $\gamma(s,\omega)$ and $\omega$ from
above by the length  of the radial curve segment that connects the
two points:
\begin{displaymath}
d(\gamma(s,\omega),\omega)  \leq
\mathrm{length}(\gamma(\cdot,\omega)|_{[s,1]})
\overset{\eqref{eq:length_el}}{=} \int_s^1 \frac{1}{\sqrt{\cos
\alpha}}\,d\sigma = \frac{1-s}{\sqrt{\cos \alpha}}.
\end{displaymath}
To prove the second bound for $d(\gamma(s,\omega),\omega)$, we
recall  the formula for radial curves provided by
Theorem~\ref{t:polar} and apply \eqref{eq:distTwoPoints-rot} to
$p=\gamma(s,\omega)$ and $\widetilde{\omega}=\omega$. This yields:
\begin{align}\label{eq:dis_gamma_omega}
d\left(\gamma(s,\omega),\omega\right)^4=  1&+ s^4 + s^2\,\left[6
\cos^2 \alpha - 2 \sin^2 \alpha\right] \\&- 4 s \cos \alpha\,
\left[\cos\left(\alpha + \tan \alpha \ln s\right)+ s^2
\cos\left(\alpha - \tan \alpha \ln s\right)\right].\notag
\end{align}
It is convenient to write this formula as
\begin{align*}
d(\gamma(s,\omega),\omega)^4 = \sin^2\alpha (1-s^2)^2 + \cos^2
\alpha (1-s)^4- \Lambda
\end{align*}
where
\begin{displaymath}
\Lambda:= 4 s\cos \alpha \left[-\sin \alpha \sin (\tan \alpha \ln
s)(1-s^2)-\cos \alpha (1+s^2) (1-\cos(\tan \alpha \ln s))\right].
\end{displaymath}
To conclude the proof, it suffices to find a suitable upper bound
for $|\Lambda|$. If $s\in (0,1/4]$, then we simply use
$|\Lambda|\lesssim s \leq (1-s)^2$, which yields in that case
$d(\gamma(s,\omega),\omega)\lesssim \sqrt{1-s}$. On the other
hand, if $s\in (1/4,1]$, then by the mean value theorem, we find
that
\begin{align*}
|\Lambda|\lesssim &  \cos \alpha (1-s^2)\left|\sin(\tan \alpha \ln
s )-\sin(\tan \alpha \ln 1)\right| \\&+  \cos^2 \alpha (1+s^2)
\left|\cos(\tan \alpha \ln s)-\cos(\tan \alpha \ln
 1)\right|\\
\lesssim & (1-s) |\ln s - \ln 1| + \cos \alpha |\ln s - \ln
1|\\\lesssim &(1-s)^2 + (1-s) \cos \alpha,
\end{align*}
which yields the desired estimate in that case.
\end{proof}

\begin{lemma}\label{l:radial_curve_dist_bdry} There exists $s_0
\in (0,1)$ such that for all $s\in [s_0,1]$ the following holds.
If  $\omega=\left(\sqrt{\cos \alpha}\cos \varphi,\sqrt{\cos
\alpha}\sin\varphi,\sin\alpha\right)$ with $\varphi \in [0,2\pi]$,
$\alpha \in (-\pi/2,\pi/2)$, $|\alpha|\geq c$, then
\begin{displaymath}
d(\gamma(s,\omega),\partial B)^4 \gtrsim_c
\left\{\begin{array}{ll}(1-s)^2,&\text{if }1-s\geq \cos \alpha,\\
\frac{(1-s)^4}{\cos^2 \alpha},&\text{if }1-s\leq \cos \alpha.
\end{array} \right.
\end{displaymath}
\end{lemma}

\begin{proof} Using formula \eqref{eq:distTwoPoints-rot} for an arbitrary point
$\widetilde{\omega}=(\sqrt{\cos
\widetilde{\alpha}}e^{\mathrm{i}\widetilde{\varphi}},\sin
\widetilde{\alpha})\in
\partial B$, we write
\begin{align}\label{d-est:l:-radial}
d(\gamma(s,\omega),\widetilde{\omega})^4=&
\left(\sqrt{1+s^4+2s^2\cos(\alpha+\widetilde{\alpha})}\right)^2
+\left(2s\sqrt{\cos \alpha}\sqrt{\cos
\widetilde{\alpha}}\right)^2\notag\\
&-4s\sqrt{\cos \alpha}\sqrt{\cos
\widetilde{\alpha}}\left(\cos(\widetilde{\alpha}+\phi)+s^2\cos(\alpha-\phi)\right),
\end{align}
where $\phi=\widetilde{\varphi}-\varphi+\tan \alpha \ln s$. We
claim that for all $\phi\in \mathbb{R}$ it holds that
\begin{equation}\label{ineq:l:radial}
 [\cos(\widetilde{\alpha}+\phi)+s^2\cos(\alpha-\phi)]^2 \leq 1+s^4+2s^2\cos(\alpha+\widetilde{\alpha}).
\end{equation} Factoring out the left-hand side, the claim is
easily seen to be equivalent to
\begin{displaymath}
2s^2\cos(\widetilde{\alpha}+\phi)\cos(\alpha-\phi)
  \leq
  \sin^2(\widetilde{\alpha}+\phi)+2s^2\cos(\alpha+\widetilde{\alpha})+s^4\sin^2(\alpha-\phi).
\end{displaymath}
Using trigonometric formulas, it follows that
\eqref{ineq:l:radial} is further equivalent to
\begin{align*}
 2s^2\cos\left((\widetilde{\alpha}+\phi)+(\alpha-\phi)\right)+ 2s^2\sin(\widetilde{\alpha}+\phi)\sin(\alpha-\phi)\leq
 \sin^2(\widetilde{\alpha}+\phi)&+2s^2\cos(\alpha+\widetilde{\alpha})\\&+s^4\sin^2(\alpha-\phi),
 \end{align*}
which is clearly true since $0\leq
\left(\sin(\widetilde{\alpha}+\phi)-s^2\sin(\alpha-\phi)\right)^2$
holds for all $\phi$. Thus, claim~\eqref{ineq:l:radial} is proven.
Assume now that $\widetilde{\omega}\in \partial B$ realizes the
distance $d(\gamma(s,\omega),\partial B)$. Inserting
\eqref{ineq:l:radial} in~\eqref{d-est:l:-radial}, we find that
\begin{align}\label{eq:denom}
d(\gamma(s,\omega),\partial B)^4 \overset{\eqref{ineq:l:radial}
}{\geq} & \left(\sqrt{1+s^4+2 s^2 \cos(\alpha +
\widetilde{\alpha})}-2s
\sqrt{\cos \alpha \cos \widetilde{\alpha}}\right)^2\notag\\
\geq &  \left(\frac{1-2s^2 \cos(\alpha-\widetilde{\alpha})
+s^4}{\sqrt{1+2s^2 \cos(\alpha +\widetilde{\alpha})+s^4}+2s \sqrt{\cos \alpha \cos \widetilde{\alpha}}}\right)^2\notag\\
\geq & \frac{(1-s^2)^4}{\left(\cos \alpha + s^2 \cos
\widetilde{\alpha} + 2 s \sqrt{\cos \alpha \cos
\widetilde\alpha}+|\sin \alpha - s^2 \sin
\widetilde{\alpha}|\right)^2}.
\end{align}
The denominator of the last expression can be bounded from above,
recalling that $|\alpha|\geq c$ and that
\begin{displaymath}
d(\gamma(s,\omega),\partial B) =
d(\gamma(s,\omega),\widetilde{\omega}) \geq |s\sqrt{\cos
\alpha}-\sqrt{\cos \widetilde{\alpha}}|\geq |\sqrt{\cos
\alpha}-\sqrt{\cos \widetilde{\alpha}}|- (1-s)\sqrt{\cos \alpha}.
\end{displaymath}
Since $1-s\leq d(\gamma(s,\omega),\partial B)$, this yields the
following estimates
\begin{align*}
|\sqrt{\cos \alpha}-\sqrt{\cos \widetilde{\alpha}}|&\lesssim
d(\gamma(s,\omega),\partial B)\\
|\cos \alpha -\cos \widetilde{\alpha}|&\lesssim \left(\sqrt{\cos
\alpha}+ d(\gamma(s,\omega),\partial B)\right)
d(\gamma(s,\omega),\partial B)\\
|\alpha -\widetilde{\alpha}|&\lesssim_c \left(\sqrt{\cos \alpha}+
d(\gamma(s,\omega),\partial B)\right) d(\gamma(s,\omega),\partial
B).
\end{align*}
This, in turn, yields the following estimate for the denominator
in \eqref{eq:denom} above,
\begin{align*}
\Big|\cos \alpha + s^2 \cos \widetilde{\alpha}& + 2 s \sqrt{\cos
\alpha \cos \widetilde\alpha}+|\sin \alpha - s^2
\sin \widetilde{\alpha}|\Big| \\
&\leq s^2 (\cos \alpha -|\sin \alpha|) + 2s \cos \alpha + \cos
\alpha +|\sin\alpha|\\
&+ s^2 |\cos \alpha -\cos \widetilde{\alpha}|+ s^2|\sin
\alpha-\sin \widetilde{\alpha}|+ 2 s \sqrt{\cos \alpha}
|\sqrt{\cos \widetilde{\alpha}}-\sqrt{\cos \alpha}|\\
&\lesssim_c \cos \alpha + (1-s^2)|\sin \alpha| +\left(\sqrt{\cos
\alpha}+ d(\gamma(s,\omega),\partial B)\right)
d(\gamma(s,\omega),\partial B)\\
&=: I_1 + I_2 +I_3.
\end{align*}
To conclude the proof, it suffices to observe that if $1-s\geq \cos \alpha$, then $I_2 \gtrsim \max\{I_1,I_3\}$,
while if $1-s\leq \cos \alpha$, then $I_1 \gtrsim \max\{I_2,I_3\}$.
The first inequality is a consequence of the direct estimate relying on $|\alpha|\geq c$, $s\geq s_0$ and on the following observation:
\[
 1-s\geq \sqrt{\cos \alpha} \sqrt{1-s}\gtrsim \sqrt{\cos \alpha}\, d(\gamma(s,\omega),\partial B).
\]
Thus $I_2\gtrsim \max\{I_1, I_3\}$. If $1-s\leq \cos \alpha$, the
inequality $I_1 \gtrsim \max\{I_2,I_3\}$ follows similarly.
\end{proof}

With  Lemmas \ref{l:radial_curve_dist} and
\ref{l:radial_curve_dist_bdry} at hand, we are ready to prove
Proposition \ref{p:curve_cone}.

\begin{proof}[Proof of Proposition \ref{p:curve_cone}] The goal is to find $\kappa>0$ such that for all
$\omega=(\sqrt{\cos \alpha}e^{\mathrm{i}\varphi},\sin \alpha)$,
\begin{equation}\label{eq:GoalRadialInCone}
d(\gamma(s,\omega),\omega) \leq (1+\kappa)
d(\gamma(s,\omega),\partial B),\quad \text{for all }s\in (0,1).
\end{equation}
We fix a constant $0<c<\pi/2$. If  $|\alpha|\leq c$ (and hence
$\cos \alpha \sim_c 1$),  Lemma \ref{l:radial_curve_dist} yields
\begin{equation}\label{eq:A1}d(\gamma(s,\omega),\omega)\lesssim_c 1-s = 1-
\|\gamma(s,\omega)\|\leq d(\gamma(s,\omega),\partial B),\quad s\in
(0,1).
\end{equation}

\medskip If $|\alpha|\geq c$, we first prove the estimate under
the assumption that $s\in [s_0,1)$ for $s_0<1$ as in  Lemma
\ref{l:radial_curve_dist_bdry}. In this situation, we discuss
separately the cases
 $1-s\leq \cos \alpha$ and $1-s\geq \cos
\alpha$, using   the two different bounds provided by  Lemma
\ref{l:radial_curve_dist}. If $1-s\leq \cos \alpha$, we deduce
\begin{equation}\label{eq:A2}
d(\gamma(s,\omega),\omega)\lesssim \frac{1-s}{\sqrt{\cos \alpha}}
\lesssim_c d(\gamma(s,\omega),\partial B) \quad\text{for all }s\in
[s_0,1).
\end{equation}
If $1-s\geq \cos \alpha$, we obtain
\begin{equation}\label{eq:A3}
d(\gamma(s,\omega),\omega)\lesssim \sqrt{1-s} \lesssim_c
d(\gamma(s,\omega),\partial B) \quad\text{for all }s\in [s_0,1).
\end{equation}
Finally, if $|\alpha|\geq c$ and $s\in (0,s_0)$, then we simply
use the crude estimate
\begin{equation}\label{eq:A4}
d(\gamma(s,\omega),\omega) \lesssim 1 \leq \frac{1}{1-s_0}
d(\gamma(s,\omega),\partial B).
\end{equation}
Combining the estimates \eqref{eq:A1} - \eqref{eq:A4}, we find
$\kappa$ such that \eqref{eq:GoalRadialInCone} holds as desired.
\end{proof}

\bibliographystyle{plain}
\bibliography{references}

\end{document}